\documentclass[10pt]{siamltex}
\usepackage{amsmath,amsfonts,amscd,amssymb,bm,cite}
\usepackage{mathrsfs,listings,url}
\usepackage{graphics,color,ulem}
\usepackage{float}
\usepackage[titletoc,page]{appendix}
\usepackage{subfigure}
\usepackage[colorlinks=true, bookmarksopen,
pdfauthor={CST},
pdfcreator={pdftex},
pdfsubject={algorithms},
linkcolor={blue},
anchorcolor={black},
citecolor={firebrick},
filecolor={magenta},
menucolor={black},
pagecolor={red},
plainpages=false,pdfpagelabels,
urlcolor={db} ]{hyperref}
\usepackage[capitalise]{cleveref}
\usepackage{comment}
\usepackage{verbatim}
\usepackage{marginnote}
\usepackage{float}
\usepackage[makeroom]{cancel}
\usepackage[pdftex]{graphicx}
\usepackage[section]{placeins}
\usepackage{mathptmx}
\usepackage{cases}
\usepackage{subeqnarray}
\usepackage{hhline}
\usepackage{xcolor}
\usepackage{tikz}
\usepackage{caption} 
\usepackage{lmodern}
\usepackage{algorithm}
\usepackage{algpseudocode}

\usetikzlibrary{arrows}
\usetikzlibrary{positioning,shapes,snakes,calc,decorations,decorations.markings}

\newtheorem{remark}{Remark}
\newtheorem{sche}{Scheme}
\newtheorem{assumption}[theorem]{Assumption}
\newtheorem{defn}[theorem]{Definition}
\newtheorem{lem}[theorem]{Lemma}
\newtheorem{thm}[theorem]{Theorem}

\restylefloat{table}
\allowdisplaybreaks
\definecolor{db}{rgb}{0.0470,0,0.5294}
\definecolor{dg}{rgb}{0.0,0.392,0.0}
\definecolor{firebrick}{rgb}{0.698,0.133,0.133}
\definecolor{bl}{rgb}{0.0,0.0,0.0}
\definecolor{linen}{rgb}{0.980,0.941,0.902}
\definecolor{ivory}{rgb}{1.0,1.0,0.941}
\definecolor{aliceblue}{rgb}{0.941,0.973,1.0}
\definecolor{beige}{rgb}{0.961,0.961,0.863}
\definecolor{tan}{rgb}{0.824,0.706,0.549}
\definecolor{lightsteelblue}{rgb}{0.690,0.769,0.871}
\definecolor{paleturquoise}{rgb}{0.686,0.933,0.933}
\definecolor{lightblue}{rgb}{0.678,0.847,0.902}
\definecolor{skyblue}{rgb}{0.529,0.808,0.922}
\definecolor{palegoldenrod}{rgb}{0.933,0.910,0.667}
\definecolor{lightgoldenrod}{rgb}{0.933,0.867,0.510}
\definecolor{lightyellow}{rgb}{1.0,1.0,0.878}
\definecolor{yellow}{rgb}{1.0,1.0,0.0}
\definecolor{lightyellow1}{rgb}{1.0,1.0,0.878}
\definecolor{lemonchiffon}{rgb}{1.0,0.980,0.804}
\definecolor{myyellow}{rgb}{1,1,.9}
\definecolor{darkgreen}{rgb}{0.0,0.392,0.0}
\definecolor{darkviolet}{rgb}{0.580,0.0,0.827}
\definecolor{lightsalmon}{rgb}{1.0,0.627,0.478}
\definecolor{orange}{rgb}{1.0,0.647,0.0}
\definecolor{darkblue}{rgb}{0.00,0.00,0.55}
\numberwithin{equation}{section}
\Crefname{table}{Table}{Tables}
\Crefname{figure}{Figure}{Figures}

\DeclareMathOperator*{\esssup}{ess\,sup}

\newcommand\titlelowercase[1]{\texorpdfstring{\lowercase{#1}}{#1}}

\begin{document}
	
	\title{\large{A D\titlelowercase{ivergence-free} P\titlelowercase{reserving} M\titlelowercase{ixed} F\titlelowercase{inite} E\titlelowercase{lement} M\titlelowercase{ethod} \titlelowercase{for}  T\titlelowercase{hermally} D\titlelowercase{riven} 
    A\titlelowercase{ctive}   F\titlelowercase{luid} M\titlelowercase{odel}}} 
	\author{Nan Zheng
			\thanks{
            School of Mathematics, Shandong University, Jinan, Shandong 250100, People’s Republic of China. Email: \href{mailto:202311835@mail.sdu.edu.cn}{202311835@mail.sdu.edu.cn}. } 
            \and  
            Qingguang Guan
            \thanks{
            School of Mathematics and Natural Sciences, University of Southern Mississippi, 118 College Drive, Hattiesburg, MS, 39406, USA   
            Email: \href{mailto:qingguang.guan@usm.edu}{qingguang.guan@usm.edu}.
            }
            \and  
            Wenlong Pei
            \thanks{
            Department of Mathematics, The Ohio State University, Columbus, OH 43210,USA. Email: \href{mailto:pei.176@osu.edu}{pei.176@osu.edu}.} 
            \and  
            Wenju Zhao
            \thanks{
            School of Mathematics, Shandong University, Jinan, Shandong 250100, People’s Republic of China.
            Email: \href{mailto:zhaowj@sdu.edu.cn}{zhaowj@sdu.edu.cn}.
            }
		    }
	\date{\emty}
	\maketitle

    \begin{abstract}
		In this report, we propose a divergence-free preserving mixed finite element method (FEM) for the system of nonlinear fourth-order thermally driven active fluid equations.
        By introducing two auxiliary variables, we lower the complexity of the model and enhance the robustness of the algorithm.
        The auxiliary variable \( w = \Delta u \) is used to convert the original fourth-order system to an equivalent system of second-order equations, thereby easing the regularity constraints imposed on standard \( H^2 \)-conforming finite element spaces.
        The second variable \( \eta \), analogous to the pressure, helps the scheme preserve the divergence-free condition arising from the model.
        The two-step Dahlquist–Liniger–Nevanlinna (DLN) time integrator, unconditionally non-linear stable and second-order accurate under non-uniform time grids, is combined with the mixed FEM for fully discrete approximation.
        Due to the fine properties of the DLN scheme, we prove the boundedness of model energy and the associated error estimates under suitable regularity assumptions and mild time restrictions.
        Additionally, an adaptive time-stepping strategy based on a minimum-dissipation criterion is to balance computational costs and time efficiency. 
        Several numerical experiments validate the theoretical findings and demonstrate the method’s effectiveness and accuracy in simulating complex active fluid dynamics.

	\end{abstract}
	
	

    \section{Introduction}
    The system of thermally driven active fluid equations, composed of self-propelled particles, exhibits unique dynamics due to the interaction between internal propulsion and external forces such as temperature gradients. 
    The coupling of temperature with the velocity field introduces additional complexities, particularly in the modeling of heat-driven instabilities and convection, and thus becomes essential in a wide range of physical phenomena, such as thermally driven active suspensions, bacterial suspensions, and other biophysical systems \cite{pnas.1722505115, qi2022emergence, ramaswamy2019active,annurev-fluid-010816-060049}.
    Herein, we consider the system on a bounded domain $D \subset \mathbb{R}^2$ and time interval $[0,t^{\ast}]$:
    the velocit $u$, pressure $p$ and temperature $T$ are governed by 
    \begin{equation}\label{TDAFEs}
    \begin{cases}
    u_{t} - \mu \Delta u + \gamma \Delta^2 u+\nu (u \cdot \nabla)u  + \rho u +  \lambda |u|^2u+ \nabla p = \sigma \xi T,    
    &\text{in}  ~D \times (0, t^*], \\
    \nabla \cdot u = 0,   
    &\text{in} ~ D \times (0, t^*],\\
    T_{t} - \nabla \cdot(\kappa \nabla T) + u \cdot \nabla T = 0,   
    &\text{in}  ~D \times (0, t^*].
    \end{cases}
    \end{equation}
    subject to the following initial-boundary conditions: 
    \begin{align*}
        \begin{cases}
            u(x,0) = u_{0} ~ \text{in} ~D, \\
            u =\Delta u = 0  ~\text{on}~ \partial D,
        \end{cases}
        \quad
        \begin{cases}
            T(x,0) = T_{0} ~\text{in} ~D, \\
            T = 0 ~\text{on} ~ \Gamma_{1}, \ \ \partial_{n} T = 0 ~\text{on}   ~\Gamma_{2}.
        \end{cases}
    \end{align*}
    Here $\partial D = \Gamma_{1} \cup \Gamma_{2}$ is Lipschitz boundary and $\partial_{n}$ is the directional derivative in the direction of ourward normal vector on $\partial D$. 
    Non-negative parameters $\mu, \gamma, \nu$ represent the viscosity coefficient, the generic stability coefficient, and the density coefficient, respectively.
    Terms $\rho u$ ($\rho \in \mathbb{R}$) and $\lambda |u|^2 u$ ($\lambda \geq 0$) correspond to a quartic Landau velocity potential \cite{ramaswamy2019active,toner1998flocks}.
    In addition, \( \sigma = \mu \cdot Ra \) ($Ra > 0$ ) is a buoyancy-related parameter, $Ra > 0$ the Rayleigh number, 
    \( \kappa \) the thermal conductivity coefficient, and \( \xi \) a unit vector aligned with the direction of gravitational acceleration	\cite{bairi2014review,li2022optimal,PATEL2014197}.

    The model of active fluid dynamics often involves generalized forms of the Navier-Stokes equations augmented with higher-order dissipative terms and nonlinear active forcing, and hence effectively captures rich dynamical behaviors.
    Nevertheless, the non-linear system of fourth-order equations brings about substantial analytical and numerical challenges, especially when it incorporates temperature effects in \eqref{TDAFEs}.
    To address these issues and reduce the regularity requirements, we introduce an auxiliary variable  \( w = -\Delta u \), and reformulate the original fourth-order active fluid equations into the following equivalent system of second-order equations. 
    \begin{equation}\label{TDAFEs 2nd}
    \begin{cases}
    u_{t} - \mu \Delta u - \gamma \Delta w + \nu (u \cdot \nabla) u + \rho u + \lambda |u|^2 u + \nabla p =  \sigma \xi T, &\text{in} ~ D \times (0, t^*], \\
    w = -\Delta u, &\text{in} ~ D \times (0, t^*], \\
    \nabla \cdot u = 0, &\text{in} ~ D \times (0, t^*],\\
    T_{t} - \nabla \cdot(\kappa \nabla T) + u \cdot \nabla T = 0,   
    &\text{in}  ~D \times (0, t^*].
    \end{cases}
    \end{equation}
    subject to the following modified initial-boundary conditions:
    \begin{align*}
    \begin{cases}
    u(x,0) = u_{0} ~ \text{in} ~D, \\
    u =w = 0  ~\text{on}~ \partial D,
    \end{cases}
    \quad
    \begin{cases}
    T(x,0) = T_{0} ~\text{in} ~D, \\
    T = 0 ~\text{on} ~ \Gamma_{1}, \ \ \partial_{n} T = 0 ~\text{on}   ~\Gamma_{2}.
    \end{cases}
    \end{align*}
    The resulting system \eqref{TDAFEs 2nd} is eligible for the use of finite element methods based on $H_0^1$-conforming basis functions without the restrictive \( H^2 \)-regularity requirements.
    More importantly, the auxiliary variable inherently satisfies a divergence-free constraint $(\nabla \cdot w = 0)$, rendering the system \eqref{TDAFEs 2nd} preserve physical fidelity and the incompressibility condition, which are crucial for realistic simulations \cite{da2025error,da2017divergence}.

    Finite element methods (FEM), known for their robustness, flexibility, and effectiveness, are extensively employed in spatial discretizations of Navier-Stokes system-related fluid dynamics models \cite{abgrall2023hybrid,abgrall2020analysis,ayuso2005postprocessed,he2007stability,MR4293957,MR4835947}.
    For temporal discretization, a variety of approaches have been thoroughly analyzed, including the Euler scheme, Crank-Nicolson related scheme, Runge-Kutta scheme, and backward differentiation formula scheme \cite{BAI20123265,baker2024numerical,banjai2012runge,decaria2022general,ern2022invariant,AAMM-16-5,hou2025unconditionally,MR4835947}, etc. 
    Recently, much effort has been devoted to the numerical analysis of variable time-stepping schemes and their potential for time adaptivity \cite{ait2023time,MR4471049}. 
    The family of Dahlquist-Liniger-Nevanlinna (DLN) methods (with one parameter $\theta \in [0,1]$), which ensures stability and second-order accuracy for general dissipative nonlinear systems with arbitrary time grids \cite{dahlquist1983stability,LPT21_AML,LPT23_ACSE},
    has been proven successful in simulations of stiff differential equations and complex fluid models \cite{CLPX2025_JSC,layton2022analysis,pei2024semi,QHPL21_JCAM,SP24_IJNAM}. 

    Given the time interval $[0,t^*]$, we denote $\{ t_{n} \}_{n=0}^{M}$ its partition and $k_n = t_{n+1} - t_n$ the local time step size.
    For the initial value problem 
    $y'(t) = g(t,y(t))$ with $t \in [0,t^\ast], \ y(0) = y_{0} \in \mathbb{R}^{d}$, 
    the family of variable time-stepping DLN methods for the problem reads 
    \begin{gather}
    \sum_{\ell =0}^{2}{\alpha _{\ell }}y_{n-1+\ell }
    = \widehat{k}_{n} g \Big( \sum_{\ell =0}^{2}{\beta _{\ell }^{(n)}}t_{n-1+\ell } ,
    \sum_{\ell =0}^{2}{\beta _{\ell }^{(n)}}y_{n-1+\ell} \Big), \qquad n=1,\ldots,M-1.
    \label{eq:1leg-DLN}
    \end{gather}
    Here $y_{n}$ represents the DLN solution of $y(t)$ at time $t_{n}$. The coefficients of the DLN method in \eqref{eq:1leg-DLN} are 
    \begin{gather}
    \label{DLNcoeff}
    \begin{pmatrix}
    \alpha _{2} \vspace{0.2cm} \\
    \alpha _{1} \vspace{0.2cm} \\
    \alpha _{0} 
    \end{pmatrix}
    = 
    \begin{pmatrix}
    \frac{1}{2}(\theta +1) \vspace{0.2cm} \\
    -\theta \vspace{0.2cm} \\
    \frac{1}{2}(\theta -1)
    \end{pmatrix}, \ \ \ 
    \begin{pmatrix}
    \beta _{2}^{(n)}  \vspace{0.2cm} \\
    \beta _{1}^{(n)}  \vspace{0.2cm} \\
    \beta _{0}^{(n)}
    \end{pmatrix}
    = 
    \begin{pmatrix}
    \frac{1}{4}\Big(1+\frac{1-{\theta }^{2}}{(1+{%
            \varepsilon _{n}}{\theta })^{2}}+{\varepsilon _{n}}^{2}\frac{\theta (1-{%
            \theta }^{2})}{(1+{\varepsilon _{n}}{\theta })^{2}}+\theta \Big)\vspace{0.2cm%
    } \\
    \frac{1}{2}\Big(1-\frac{1-{\theta }^{2}}{(1+{\varepsilon _{n}}{%
            \theta })^{2}}\Big)\vspace{0.2cm} \\
    \frac{1}{4}\Big(1+\frac{1-{\theta }^{2}}{(1+{%
            \varepsilon _{n}}{\theta })^{2}}-{\varepsilon _{n}}^{2}\frac{\theta (1-{%
            \theta }^{2})}{(1+{\varepsilon _{n}}{\theta })^{2}}-\theta \Big)%
    \end{pmatrix},
    \end{gather}
    where $\varepsilon _{n} = (k_n - k_{n-1})/(k_n + k_{n-1}) \in (-1,1)$ is the step variability.
    The weighted average of time step  $\widehat{k}_n = \alpha_2 k_n - \alpha_0 k_{n-1}$ in \eqref{eq:1leg-DLN} is constructed for second-order accuracy. 
    The family of schemes in \eqref{eq:1leg-DLN} is reduced to the midpoint rule on $[t_{n}, t_{n+1}]$ if $\theta = 1$ and the midpoint rule on $[t_{n-1}, t_{n+1}]$ if $\theta = 0$ \cite{MR4432611,MR4092601}.

    Motivated by these advances, we integrate mixed finite element spatial discretization with DLN temporal discretization to propose a computationally efficient, stable, and accurate framework for analyzing the sophisticated dynamics exhibited by the system of thermally driven active fluid equations in \eqref{TDAFEs 2nd}.
    An adaptive time-stepping strategy based on the minimal dissipation criterion \cite{capuano2017minimum} is further developed to balance the conflict between accuracy and computational costs.
    The main contributions of this report are to:  
    \begin{itemize}
        \item construct a divergence-free preserving mixed finite element spatial discretization for the system of fourth-order thermally driven active fluid equations in \eqref{TDAFEs}, which reduces complexity and relaxes regularity requirements,
        \item     
        utilize the variable time-stepping DLN temporal integrator for full discretization and present rigorous proof that the fully discrete algorithm is stable in the model energy 
        \begin{align*}
            \mathcal{E}(t) = \frac{1}{2} \Big( \int_{D} \big( |u(t)|^2  + T(t)^2 \big) dx \Big)
        \end{align*}
        under arbitrary time grids,
        \item 
        carry out detailed proof of error estimates for velocity, pressure, and temperature under moderately relaxed regularity requirements and time step constraints,
        \item
        design a time-adaptive strategy based on the minimal dissipation criterion, which significantly enhances time efficiency in practice.
    \end{itemize}

    The remainder of this paper is structured as follows.
    Necessary preliminaries and notations are provided in Section \ref{sec:sec2}. 
    In Section \ref{sec:sec3}, we establish the fully discrete scheme for the fourth-order thermally driven active fluid equations \eqref{TDAFEs} based on a divergence-free preserving mixed finite element spatial discretization and variable time-stepping DLN time integrator. 
    With the help of a mild time step restriction, we prove rigorously that the scheme is long-time stable in model energy under non-uniform time grids.
    Error estimates of the resulting fully discrete scheme with detailed proof are presented in Section \ref{sec:sec4}.
    A series of numerical experiments in Section \ref{sec:sec5} fully support the theoretical findings.
    Section \ref{sec:sec6} summarizes the main results of the report and outlines potential directions for future research.

    \section{Notation and Preliminaries} \label{sec:sec2}
    The Sobolev space \( W^{r,p}(D) \) (\( 1 \leq p \leq \infty \), \( r \in \mathbb{N} \) ) equipped with the norm \( \|\cdot\|_{W^{r,p}} \) consists of functions of which weak derivatives up to order \( r \) belong to \( L^p(D) \).
    For $p=2$,  \( W^{r,2}(D) \) is the Hilbert space $H^r(D)$ with norm \( \|\cdot\|_r \).
    In particular, \( L^2(D) := H^{0}(D) \) is the Lebesgue space endowed with the usual inner product $(\cdot, \cdot)$  and norm \( \|\cdot\| \). 
    We need the following Bochner spaces for time-dependent functions
    \begin{align*}
        L^p([0,T]; H^r(D)) &= \Big\{ v(\cdot,t) \in H^r(D): \| v \|_{p,r} = \Big( \int_{0}^{T}\|v(t)\|_r^p\mathrm{d}t \Big)^\frac{1}{p} < \infty \Big\},  \\
        L^{\infty}([0,T]; H^r(D)) &= \Big\{ v(\cdot,t) \in H^r(D): \| v \|_{\infty,r} = \esssup_{0 \leq t \leq T} \|v(t)\|_r< \infty \Big\}. 
    \end{align*}
    The solution spaces for velocity $u$, pressure $p$ and temperature $T$ in \eqref{TDAFEs 2nd} are 
    \begin{align*}
    X &= \Big\{ u \in [H^1(D)]^2 : u = 0 ~\text{on}~ \partial D \Big\},
    \\
    Q &= \Big\{ p \in L^2(D) : \int_{D} p \,\mathrm{d}x = 0 \Big\},
    \\
    J &=
    \Big\{ T \in H^1(D) :  T = 0 ~\text{on}~ \Gamma_1 \Big\}.
    \end{align*}
    The divergence-free subspace for velocity is 
    \begin{equation}\label{V_space}
    V = \Big\{ v \in X; \nabla \cdot v =0 ~in ~D  \Big\},
    \end{equation}
    Moreover, we define two trilinear operators 
    \begin{align}\label{b form}
    b(u,v,w) &= \frac{1}{2} \big( (u \cdot \nabla) v, w \big) - \frac{1}{2} \big( (u \cdot \nabla) w, v \big), \quad 
    u,v,w \in X,
    \\
    b^*(u,\phi,\Psi) &= \frac{1}{2}\big( u \cdot \nabla \phi, \Psi \big) - \frac{1}{2} \big(u \cdot \nabla 	\Psi, \phi \big), \quad \quad \ u \in X , \quad \phi , \Psi \in J. 
    \end{align}
    We have the following estimates for $b(\cdot,\cdot,\cdot)$ (see \cite{he2007stability,MR1043610} for proof):
    % \begin{subequations} %\label{b inequality 1} 
	\begin{align}
	&b(u,v,w) \leq C \| \nabla u \| \|\nabla v\| \|\nabla w\|,  \quad \forall u,v,w \in X,\label{b 1 1 1}
	\\
	&b(u,v,w) \leq C \| u \|^{\frac{1}{2}} \|\nabla u\|^{\frac{1}{2}} \|\nabla v\| \|\nabla w\|,  \quad \forall u,v,w \in X, \label{b 0 1 1 1}
	\\
	&b(u,v,w) \leq C \| u \|_{1} \| v \|_{1} \| w \|^{\frac{1}{2}} 
	\| w \|_{1}^{\frac{1}{2}},   \label{b 1 1 0 1} \\
	&b(u,v,w) \leq C \|u\| \|v\|_2 \|w\|_1, \quad \forall u,w \in X, \ v \in [H^2(D)]^2, \label{b 0 2 1} \\
	&b(u,v,w) \leq C \|u\|_{2} \|\nabla v \| \|w\|, \quad \forall u,w \in X, \ 
	v \in [H^2(D)]^2  \cap X, \label{b 2 1 0} \\
	&b(u,v,w) \leq C \|u\|_{1} \| v \|_{2} \|w\|, \quad \forall u,w \in X, \ 
	v \in [H^2(D)]^2 \cap X, \label{b 1 2 0} 
	\end{align}
    % \end{subequations}
    where \( C > 0 \) is a positive constant which only depends on the domain \( D \). 
    The operator \( b^*(\cdot, \cdot, \cdot) \) has the same estimation results as 
    \eqref{b 1 1 1} - \eqref{b 1 2 0}.

    The weak form of \eqref{TDAFEs 2nd} is: 
    finding $(u,w,p,T) \in (X,X,Q,J) $ for all $t \in (0,t^*]$ , such that, for all $(v,\varphi,q,\vartheta) \in (X,X,Q,J)$,
    \begin{align} 
    \begin{split} \label{variational-formula-1} 
    &(u_{t}, v)
    + \mu(\nabla u, \nabla v) 
    +\gamma (\nabla w,\nabla v)
    + \nu b( u,  u, v)  
    + \rho (u,v) 
    + \lambda (|u|^2 u ,v)
    \\
    &
    - (p, \nabla \cdot v)
    = \sigma\xi(T, v), 
    \end{split}
    \\
    &(w,\varphi) = (\nabla u,\nabla \varphi), \label{variational-formula-2}
    \\
    &(\nabla \cdot u, q) = 0,  \label{variational-formula-3}
    \\
    &	(T_{t}, \vartheta)
    +\kappa (\nabla T, \nabla \vartheta)
    +b^{\ast}(u,  T, \vartheta)
    = 0.\label{variational-formula-4}
    \end{align} 
    An equivalent weak form of \eqref{variational-formula-1} - \eqref{variational-formula-4} can be derived without embedding the divergence-free conditions directly into the function spaces: 
    finding $(u,w,\phi,p,T) \in (X,X,Q,Q,J)$ such that for all $t \in (0,t^{\ast}]$ and all $(v,\varphi,\zeta,q,\vartheta) \in (X,X,Q,Q,J)$,
    \begin{align} 
    \begin{split} \label{variational_formula_second_1} 
    &(u_{t}, v)
    + \mu(\nabla u, \nabla v) 
    +\gamma (\nabla w,\nabla v)
    + \nu b( u,  u, v)  
    + \rho (u,v) 
    + \lambda (|u|^2 u ,v)
    \\
    &
    - (p, \nabla \cdot v)
    = \sigma\xi(T, v), 
    \end{split}
    \\
    &(w,\varphi)-(\phi,\nabla \cdot \varphi) = (\nabla u,\nabla \varphi), \label{variational_formula_second_2}
    \\
    &(\nabla \cdot u, q) = 0.  \label{variational_formula_second_3}
    \\
    &(\nabla \cdot w, \zeta) = 0,  \label{variational_formula_second_4}
    \\
    &	(T_{t}, \vartheta)
    +\kappa (\nabla T, \nabla \vartheta)
    +b^{\ast}(u,  T, \vartheta)
    = 0.\label{variational_formula_second_5}
    \end{align}

    We need the following lemma about monotonicity and continuity properties on $\mathbb{R}^{2}$ 
    (see \cite{deugoue2022numerical,glowinski1975approximation} for proof).
    \begin{lemma}\label{vector_norm_inequality_lem} 
        For all $x, y \in \mathbb{R}^2$, the following inequalities hold:  
        \begin{align}
        &\text{monotonicity: } 
        \big( |x|^{2} x - |y|^{2} y, x-y \big)_{\mathbb{R}^2}  \geq C |x-y|^4,  \label{vector_norm_inequality_1}  \\
        &\text{continuity: } \ \ \ \ 
        \big| |x|^{2} x - |y|^{2} y \big| \leq C \big( |x| + |y| \big)^{2} |y - x|,  \label{vector_norm_inequality_2}
        %\end{cases}
        \end{align}
        where $C > 0$ is a constant independent of $x$ and $y$, and \( (\cdot,\cdot)_{\mathbb{R}^2} \) denotes the standard Euclidean inner product on $\mathbb{R}^2$.
    \end{lemma}

    \section{Spatial and temporal discretization} \label{sec:sec3}
    \subsection{Spatial discretization}
    Let $\mathscr{T}_h$ ($0<h<1$) be a regular triangulation of $\overline{D} = \cup_{K\in\mathscr{T}_h}\overline{K}$ with the mesh size $h$.
    For any integer $r \geq 1$, $P_k(K)$ denotes the space of polynomials of degree less than or equal to $r$ on $K \in\mathscr{T}_h$. The finite element spaces for solution spaces are
    \begin{align}
    &X_{h} = \Big\{ v^{h} \in [C^{0}(\bar{\Omega})]^2 \cap X : v^{h} \mid_{K} \in P_{r+1}(K), \forall K \in \mathscr{T}_{h} \Big\},   \label{zeq:FEM-space-Xh-1}\\
    &Q_{h} = \Big\{ q^{h} \in C^{0}(\bar{\Omega}) \cap L^2(D) : q^{h} \mid_{K} \in P_{r}(K), \forall K \in \mathscr{T}_{h} \Big\},   \label{zeq:FEM-space-Qh-1} \\
    &J_{h} = \Big\{ \vartheta^{h} \in C^{0}(\bar{\Omega}) \cap J : \vartheta^{h} \mid_{K} \in P_{r+1}(K), \forall K \in \mathscr{T}_{h} \Big\},   \label{zeq:FEM-space-Jh-1} \\
    &V_{h} = \Big\{ v^{h} \in X_{h} : ( \nabla \cdot v^{h}, q^{h} ) = 0, \forall q^{h} \in Q_{h} \Big\}.  \label{zeq:FEM-space-Vh-1}
    \end{align}
    The space $X_h \times Q_h$ in \eqref{zeq:FEM-space-Xh-1}-\eqref{zeq:FEM-space-Qh-1}  satisfies the Ladyzhenskaya-Babuska-Brezzi condition ($LBB^h$ condition): 
    for any $ q^h \in Q_h$ 
    \begin{equation}\label{LBB}
    \|q^h\|_{L^2(D)} \leq C \sup_{v^h \in X_h \setminus \{0\}} \frac{(\nabla \cdot v^h, q^h)_{L^2(D)}}{\|\nabla v^h\|_{L^2(D)}}.
    \end{equation}
    where $C>0$ is independent of $q^h$. 
    Taylor--Hood element space, Mini element space, and {Scott--Vogelius element} space~\cite{MR813691} are all typical finite element spaces having $LBB^h$ condition in \eqref{LBB}. 
    We choose Taylor--Hood element space ($r=1$) for spatial discretization throughout our work.
    $X_h$ satisfies the inverse inequality
    \begin{align}
    \label{inverse-ineq}
    \| \nabla v^h \| \leq C h^{-1} \| v^h \|, \qquad \forall v^h \in X_h.
    \end{align}
    Here $C$ is a positive constant independent of the mesh diameter $h$.
    We introduce Ritz projection $\mathcal{R}_h:J\rightarrow J_h$ 
    \begin{equation}\label{Ritz-projection-defination}
    \big( \nabla(T- \mathcal{R}_h T),\nabla \vartheta^h \big)=0,  \quad 
    \big(  T - \mathcal{R}_h T, 1 \big) = 0,   \quad \forall \vartheta^h \in J_h.
    \end{equation}
    The Ritz projection satisfies the following estimates \cite{Cia02_SIAM,he2005stabilized,Tho06_Springer}
    \begin{equation}\label{Ritz projection}
    \| T- \mathcal{R}_h T\| +h \|\nabla(T- \mathcal{R}_h T)\|\leq C h ^{r+2}\| T \|_{r+2},
    \end{equation}
    We define a Stokes-type projection \( \mathcal{S}_h : (X, X, Q, Q) \rightarrow (X_h, X_h, Q_h, Q_h) \). 
    Given the pair of function \( (u, w, \phi, p) \in (X, X, Q, Q) \), 
    $\mathcal{S}_h (u, w, \phi, p) = (\mathcal{S}_h u, \mathcal{S}_h w, \mathcal{S}_h \phi, \mathcal{S}_h p)$ uniquely solves: for all \( (v^h, \varphi^h, \zeta^h, q^h) \in (X_h, X_h, Q_h, Q_h) \) 
    \begin{equation}\label{Stokes-type-projection-defination}
    \begin{cases}
    \mu(\nabla(u - \mathcal{S}_{h} u), \nabla v^h) 
    + \gamma (\nabla(w - \mathcal{S}_{h} w), \nabla v^h) 
    - (p - \mathcal{S}_{h} p, \nabla \cdot v^h) = 0, 
    \\
    (w - \mathcal{S}_{h} w, \varphi^h) - (\phi - \mathcal{S}_{h} \phi, \nabla \cdot \varphi^h)
    = (\nabla(u - \mathcal{S}_{h} u), \nabla \varphi^h),
    \\
    (\nabla \cdot \mathcal{S}_h u, q^h) = 0,
    \\
    (\nabla \cdot \mathcal{S}_h w, \zeta^h) = 0.
    \end{cases}
    \end{equation}
    The projection $\mathcal{S}_h$ in \eqref{Stokes-type-projection-defination} satisfies the following approximation 
    \begin{equation}\label{Stokes-type projection}
    \begin{split}
    &\| u - \mathcal{S}_h u\| + \| w - \mathcal{S}_h w\|
    + h\left\{
    \|\nabla (u - \mathcal{S}_h u)\| + \|\nabla (w - \mathcal{S}_h w)\| + \|p -  \mathcal{S}_h p\|
    \right\} 
    \\
    &\leq C h ^{r+2} \left\{
    \| u \|_{r+2} + \| w \|_{r+2} + \| p \|_{r+1}
    \right\}.
    \end{split}
    \end{equation}
    The proof of \eqref{Stokes-type projection} is similar to the approximation of Stokes projection, we refer to \cite{GR86_Springer,Joh16_Springer}.

    \subsection{DLN fully discrete scheme} 
    To formulate the DLN fully discrete scheme, we introduce the following notation. 
    Given arbitrary sequence $\{ z_{n} \}_{n=0}^{\infty}$,  we adopt the following notations for convenience
    \begin{align*}
    z_{n,\alpha} = \alpha_{2} z_{n+1} + \alpha_{1} z_{n} + \alpha_{0} z_{n-1}, \qquad
    z_{n,\beta} = \beta _{2}^{(n)} z_{n+1} + \beta _{1}^{(n)} z_{n} 
    + \beta _{0}^{(n)} z_{n-1}, 
    \end{align*}
    where $\{ \alpha_{\ell} \}_{\ell=0}^{2}$ and $\{ \beta_{\ell}^{(n)} \}_{\ell=0}^{2}$ are coefficents of the DLN method in \eqref{DLNcoeff}.
    The fully discrete scheme of \eqref{variational_formula_second_1} - \eqref{variational_formula_second_5} is 
    \begin{sche}[DLN fully discrete scheme] \label{fully discrete formulations scheme}
        Given $u_n^h, u_{n-1}^h \in X_h,T_n^h, T_{n-1}^h \in J_h$ for $n= 1, 2, \dots, M-1$, 
        find $(u_{n+1}^h, w_{n+1}^h,\phi_{n+1}^h, p_{n+1}^h,T_{n+1}^h) \in( X_h, X_h, Q_h,Q_h,J_h)$ such that, for all $(v^h, \varphi^h,\zeta^h,q^h,\vartheta^h) \in 
        ( X_h, X_h,Q_h,Q_h,J_h)$,
        \begin{align} \label{fully discrete formulations 1}
        \begin{split}
        &\frac{1}{\widehat{k}_n}(u_{n,\alpha}^h, v^h)
        + \mu(\nabla u_{n,\beta}^h, \nabla v^h) 
        +\gamma (\nabla w_{n,\beta}^h,\nabla v^h)
        + \nu b( u_{n,\beta}^h,  u_{n,\beta}^h, v^h) 
        \\
        &\quad  
        + \rho (u_{n,\beta}^h,v^h) 
        + \lambda (|u^h_{n,\beta}|^2 u^h_{n,\beta} ,v^h)
        - (p_{n,\beta}^h, \nabla \cdot v^h)
        = ( \sigma \xi T_{n,\beta}, v^h), 
        \end{split}
        \\
        \begin{split}
        &(w_{n+1}^h,\varphi^h) -(\phi_{n+1}^h,\nabla \cdot \varphi^h)= (\nabla u_{n+1}^h,\nabla \varphi^h),\label{fully discrete formulations 2}
        \end{split}
        \\
        \begin{split}
        &(\nabla \cdot u_{n+1}^h, q^h) = 0, \label{fully discrete formulations 3}
        \end{split}
        \\
        \begin{split}
        &(\nabla \cdot w_{n+1}^h, \zeta^h) = 0, \label{fully discrete formulations 4}
        \end{split}
        \\
        \begin{split}
        &\frac{1}{\widehat{k}_n}(T_{n,\alpha}^h, \vartheta^h)
        +\kappa (\nabla T_{n,\beta}^h, \nabla \vartheta^h)
        +b^*(u_{n,\beta}^h,  T_{n,\beta}^h, \vartheta^h)
        = 0. \label{fully discrete formulations 5}
        \end{split}
        \end{align} 
    \end{sche}
    In the beginning, we set $u_0^h = \mathcal{S}_h u_0,  T^h_0 = \mathcal{R}_h T_0$ and employ the fully-implicit Crank-Nicolson scheme to solve for $u_1^h, T_1^h$, which ensures second-order accuracy and preserves numerical stability.

    \subsection{Stability Analysis}
    We define the discrete model energy of \Cref{fully discrete formulations scheme} at time $t_n$
    \begin{align*}
        \mathcal{E}_n = \frac{1}{2} \big( \| u_{n}^h \|^2 + \| T_n^h \|^2 \big),
    \end{align*}
    and prove the boundedness of discrete model energy by the $G$-stability property of the family of DLN methods \eqref{eq:1leg-DLN} in this subsection. 
    \begin{defn}
        For $\theta \in [0 , 1]$ and  $u, v \in [L^2(D)]^2$, the $G$-norm $\| \cdot \|_{G(\theta)}$ is defined as
        \begin{equation} \label{G-norm}
        \begin{Vmatrix}
        u \vspace{0.5mm} \\ 
        v
        \end{Vmatrix}^2_{G(\theta )}
        =
        \int_{D}
        \big[u^\top \ v^\top \big] G(\theta) 
        \begin{bmatrix}
        u \vspace{0.5mm} \\
        v
        \end{bmatrix}
        \mathrm{d}x
        =\frac{1}{4}(1+\theta)\|u\|^2 +\frac{1}{4}(1-\theta)\|v\|^2,
        \end{equation}
        where the notation $\top$ represents the transpose of a vector, $G(\theta)$ is a symmetric semi-positive definite matrix 
        \[
        G(\theta) = \begin{bmatrix}
        \frac{1}{4}(1+\theta)\mathbb{I}_{2\times 2} & 0 \\
        0 & \frac{1}{4}(1-\theta)\mathbb{I}_{2\times 2}
        \end{bmatrix},
        \]
    \end{defn}
    and $\mathbb{I}_{2\times 2}$ identity matrix in $\mathbb{R}^{2 \times 2}$
    \begin{lem} \label{G-stable-lemma}
        For any sequence $\{ v_{n} \}_{n=0}^{M} \subset [L^2(D)]^2$, the following identity holds
        \begin{equation}\label{G-stable}
        \big( v_{n,\alpha}, v_{n,\beta} \big)  
        =
        \begin{Vmatrix}
        v_{n+1} \vspace{0.5mm} \\
        v_n
        \end{Vmatrix}^2_{G(\theta )}
        -
        \begin{Vmatrix}
        v_{n} \vspace{0.5mm} \\
        v_{n-1}
        \end{Vmatrix}^2_{G(\theta )}
        +
        \Big\| \sum_{\ell=0}^{2} a_{\ell}^{(n)} v_{n-1+{\ell}} \Big\|^2,
        \end{equation}
        holds for all $n = 1,2, \cdots, M-1$ and any fixed $\theta \in [0,1]$. Here $\{a_{\ell}^{(n)}\}_{\ell=0}^2$ 
        in \eqref{G-stable} are 
        \begin{equation*} 
        a_1^{(n)}  = -\frac{\sqrt{\theta (1-\theta^2)}}{\sqrt{2}(1+\varepsilon_n \theta)},  \quad
        a_2^{(n)} = -\frac{1-\varepsilon_n}{2} a_1^{(n)} ,   \quad
        a_0^{(n)} = -\frac{1+\varepsilon_n}{2} a_1^{(n)}.   
        \end{equation*}
    \end{lem}
    \begin{proof}
        The proof of $G$-stability identity in \eqref{G-stable} is just a algebraic calculation. 
    \end{proof}

    \begin{thm}[Boundedness] \label{Stability of DLN method}
        Assume that $ u_0^h, u_1^h \in X_h, T_0^h, T_1^h \in J_h$ and time steps satisfy
        \begin{align} \label{time-cond-stab}
        C_{\beta}^{(n)} \Big( |\rho|+ \frac{C_P^2 \sigma^2 | \xi |^2}{2\mu} \Big)\widehat{k}_{n} < \frac{1+\theta}{4}, \quad \forall \ n = 1, \cdots M-1, 
        \end{align}
        where $\displaystyle C_{\beta}^{(n)} = \sum_{\ell = 0}^{2} \big( \beta_{\ell}^{(n)} \big)^{2}$ and $C_P$ is the constant in Poincar\'e inequality.
        Scheme \Cref{fully discrete formulations scheme} is bounded in the discrete model energy, i.e. for $ 1 <m \leq M $, 
        \begin{align} \label{Stability of DLN inequality}
        \begin{split}
            &\frac{1+\theta}{2} \mathcal{E}_{m}
            +
            \sum_{n=1}^{m-1} \Big[ \Big\| \sum_{\ell=0}^{2} a_{\ell}^{(n)} u^h_{n-1+ \ell} \Big\|^2 
            + \Big\| \sum_{\ell=0}^{2} a_{\ell}^{(n)} T^h_{n-1+ \ell} \Big\|^2 
            \\&
            \qquad \qquad \qquad \ \ \  + \widehat{k}_n \Big( \frac{\mu}{2} \| \nabla u^h_{n,\beta}\|^2
            + \gamma \| w_{n,\beta}^h \|^2 + \lambda \| u^h_{n,\beta} \|_{L^4}^4+ \kappa \| \nabla T^h_{n,\beta} \|^2 \Big) \Big]
            \\
            &\leq
            \exp \big( C(\theta) T \big) \Big(  \frac{1+\theta}{2} \mathcal{E}_{1} + \frac{1-\theta}{2} \mathcal{E}_{0} \Big). 
        \end{split}
        \end{align}
    \end{thm}
    \begin{proof}
        By \eqref{fully discrete formulations 2} and \eqref{fully discrete formulations 4},
        \begin{align*}
        (w_{n,\beta}^h,\varphi^h) - (\phi_{n,\beta}^h,\nabla \cdot \varphi^h) 
        &= (\nabla u_{n,\beta}^h,\nabla \varphi^h), 
        \quad \forall \varphi^h \in X_h, \\
        (\nabla \cdot w_{n,\beta}^h, \zeta^h) &= 0, 
        \quad \forall \zeta^h \in Q_h.
        \end{align*}
        Choosing $\varphi^h = w_{n,\beta}^h$ and $\zeta^h = \phi_{n,\beta}^h$ in the above two equations gives
        \begin{align}
            (w_{n,\beta}^h, w_{n,\beta}^h) = (\nabla u_{n,\beta}^h,\nabla w_{n,\beta}^h).
            \label{eq:Stab-eq1}
        \end{align}
        Similarly, we set $q^h = p^h_{n,\beta}$ in \eqref{fully discrete formulations 3} and obtain
        \begin{align}
            (\nabla \cdot u_{n,\beta}^h, p^h_{n,\beta})=0.
            \label{eq:Stab-eq2}
        \end{align}
        We set 
        $v^h = u^h_{n,\beta}$ in \eqref{fully discrete formulations 1}, 
        $\vartheta^h = T^h_{n,\beta}$ in \eqref{fully discrete formulations 5}, 
        add the two equations together, and use \eqref{eq:Stab-eq1} - \eqref{eq:Stab-eq2} to derive
        \begin{align*}
        &\frac{1}{\widehat{k}_n} (u_{n,\alpha}^h, u^h_{n,\beta}) 
        + \mu \| \nabla u^h_{n,\beta} \|^2 + \gamma \| w_{n,\beta}^h \|^2
        +\rho \| u^h_{n,\beta} \|^2 + \lambda \| u^h_{n,\beta} \|_{L^4}^4 
        + \nu b(u_{n,\beta}^h, u_{n,\beta}^h, u^h_{n,\beta}) \notag \\
        &
        +\frac{1}{\widehat{k}_n} (T_{n,\alpha}^h, T^h_{n,\beta}) 
        + \kappa \| \nabla T^h_{n,\beta} \|^2 
        +b^*(u_{n,\beta}^h, T_{n,\beta}^h, T^h_{n,\beta})
        = (\sigma \xi T_{n,\beta}^h, u^h_{n,\beta}).
        \end{align*}
        Since $b(u_{n,\beta}^h, u_{n,\beta}^h, u^h_{n,\beta}) =0$ and $b^*(u_{n,\beta}^h, T_{n,\beta}^h, T^h_{n,\beta}) =0$ by definition, we apply the Cauchy-Schwarz inequality and Poincar\'e inequality to the above equality and derive
        \begin{align*}
        &\frac{1}{\widehat{k}_n} (u_{n,\alpha}^h, u^h_{n,\beta}) 
        + \mu \| \nabla u^h_{n,\beta} \|^2 
        + \gamma \| w_{n,\beta}^h \|^2
        +\rho \| u^h_{n,\beta} \|^2
        + \lambda \| u^h_{n,\beta} \|_{L^4}^4
        \\
        &+ \frac{1}{\widehat{k}_n} (T_{n,\alpha}^h, T^h_{n,\beta}) 
        + \kappa \| \nabla T^h_{n,\beta} \|^2 
        \leq C_P \sigma |\xi | \|T_{n,\beta}^h \| \|\nabla u^h_{n,\beta}\|.
        \end{align*}
        We apply $G$-stability identity in \eqref{G-stable} and Young's inequality to the above inequality 
        \begin{align*}
            &
            \begin{Vmatrix}
                u^h_{n+1} \vspace{0.5mm} \\ 
                u^h_{n}
            \end{Vmatrix}^2_{G(\theta )}
            \!\!\!-\! 
            \begin{Vmatrix}
                u^h_{n} \vspace{0.5mm} \\ 
                u^h_{n-1}
            \end{Vmatrix}^2_{G(\theta )}
            \!\!\!+\!\! \bigg\|  \sum_{\ell=0}^{2} a_{\ell}^{(n)} u^h_{n-1+\ell} \bigg\|^2
            \!\!\!+\! \frac{\mu \widehat{k}_n}{2} \| \nabla u^h_{n,\beta} \|^2 
            \!+\! \rho \widehat{k}_n \| u^h_{n,\beta} \|^2 
            \!+\! \lambda \widehat{k}_n \| u^h_{n,\beta} \|_{L^4}^4
            \\
            &
            \begin{Vmatrix}
                T^h_{n+1} \vspace{0.5mm} \\ 
                T^h_{n}
            \end{Vmatrix}^2_{G(\theta )}
            \!\!\!-\! 
            \begin{Vmatrix}
                T^h_{n} \vspace{0.5mm} \\ 
                T^h_{n-1}
            \end{Vmatrix}^2_{G(\theta )}
            \!\!\!+\! \bigg\| \sum_{\ell=0}^{2} a_{\ell}^{(n)} T^h_{n-1+\ell} \bigg\|^2
            \!+\! \kappa \widehat{k}_n \| \nabla T^h_{n,\beta} \|^2 
            \!+\! \gamma \widehat{k}_n \| w_{n,\beta}^h \|^2 \\
            &\leq \frac{C_P^2}{2 \mu} \widehat{k}_n \sigma^2 | \xi |^2 \| T_{n,\beta}^h \|^2.
        \end{align*}
        We apply the definition of the $G$-norm in \eqref{G-norm}, and sum the above inequality over $n$ from $1$ to $m-1$.
        \begin{align*}
        \begin{split}
        &\frac{1+\theta}{2} \mathcal{E}_{m}
        +
        \sum_{n=1}^{m-1} \Big[ \Big\| \sum_{\ell=0}^{2} a_{\ell}^{(n)} u^h_{n-1+ \ell} \Big\|^2 
        + \Big\| \sum_{\ell=0}^{2} a_{\ell}^{(n)} T^h_{n-1+ \ell} \Big\|^2 
        \\&
        \qquad \qquad \qquad \ \ \ + \widehat{k}_n \Big( \frac{\mu}{2} \| \nabla u^h_{n,\beta}\|^2
        + \gamma \| w_{n,\beta}^h \|^2 + \lambda \| u^h_{n,\beta} \|_{L^4}^4+ \kappa \| \nabla T^h_{n,\beta} \|^2 \Big) \Big]
        \\
        \leq& C_{\beta}^{(m-1)} |\rho| \widehat{k}_{m-1} \big(  \| u^h_{m} \|^{2} + \| u^h_{m-1} \|^{2} + \| u^h_{m-2} \|^{2}  \big) + \sum_{n=1}^{m-2} \widehat{k}_n |\rho| \| u^h_{n,\beta} \|^2 
        \\
        & + \frac{C_P^2 C_{\beta}^{(m-1)} \sigma^2 | \xi |^2}{2 \mu} \widehat{k}_{m-1} \big(  \| T^h_{m} \|^{2} + \| T^h_{m-1} \|^{2} + \| T^h_{m-2} \|^{2}  \big) + \sum_{n=1}^{m-2} \frac{C_P^2 \widehat{k}_n\sigma^2 | \xi|^2}{2\mu}\| T^h_{n,\beta} \|^2 
        \\
        &+ \frac{1+\theta}{2} \mathcal{E}_{1} + \frac{1-\theta}{2} \mathcal{E}_{0}.
        \end{split}
        \end{align*}
        We apply discrete Gronwall’s inequality \cite[p.369]{MR1043610} to above inequality to achieve \eqref{Stability of DLN inequality}.
    \end{proof}

    \section{Error analysis}
    \label{sec:sec4}
    Throughtout this section, we denote $u_{n}$, $w_{n}$, $\phi_{n}$, $p_{n}$, $T_n$ to be the true solutions of velocity $u$, auxiliary variable $w$, auxiliary variable $\phi$,  pressure $p$, temperature $T$ at time $t_{n}$ respectively. 
    We decompose the error functions for velocity $e^u_n$, auxiliary variable $e^w_n$, pressure $e^p_n$ and temperature $e^T_n$ at time $t_{n}$ as:
    \begin{align*}
    &e^u_n =u_n^h - u(t_{n})  = (u_n^h - \mathcal{S}_h u(t_{n})) - (u_n - \mathcal{S}_h u(t_{n})) :=\psi^u_n - \eta^u_n,
    \\
    &e^w_n =w_n^h - w(t_{n}) = (w_n^h - \mathcal{S}_h w(t_{n})) - (w_n - \mathcal{S}_h w(t_{n})) := \psi^w_n - \eta^w_n ,
    \\
    &e^p_n =p_n^h - p(t_{n})=  (p_n^h - \mathcal{S}_h p(t_{n})) - (p_n - \mathcal{S}_h p(t_{n})) := \psi^p_n -\eta^p_n,
    \\
    %&e^\phi_n =\phi_n^h - \phi(t_{n})=  (\phi_n^h - \mathcal{S}_h \phi(t_{n})) - (\phi(t_{n}) - \mathcal{S}_h \phi(t_{n})) = \psi^\phi_n -\eta^\phi_n,
    %\\
    &e^T_n =T_n^h - T(t_{n})=  (T_n^h - \mathcal{R}_h T(t_{n})) - (T(t_{n}) - \mathcal{R}_h T(t_{n})) = \psi^T_n -\eta^T_n
    \end{align*} 
    where $\mathcal{R}_h$ represents Ritz projection in \eqref{Ritz-projection-defination} and $\mathcal{S}_h$ the Stokes-type projection in \eqref{Stokes-type-projection-defination} respectively.
    The following lemma about the consistency property of the family of DLN methods is necessary in error analysis. 
    \begin{lemma}[consistency] \label{n beta int tn-1 tn+1}
        Let $u(\cdot, t)$ be the mapping from $[0,T]$ to $H^{r}(D)$. Assuming that the mapping $u(\cdot, t)$ is third-order differentiable in time, then for any $\theta \in [0,1)$
        \begin{align*}
        &\big\| u_{n,\beta} - u(t_{n,\beta}) \big\|_{r}^2
        \leq
        C(\theta)(k_n+k_{n-1})^3 \int_{t_{n-1}}^{t_{n+1}}\|u_{tt}\|_{r}^2 \mathrm{d}t, \\
        &\Big\| \frac{u_{n,\alpha}}{\widehat{k}_n} - u_t(t_{n,\beta}) \Big\|_{r}^2
        \leq
        C(\theta) (k_n+k_{n-1})^3 \int_{t_{n-1}}^{t_{n+1}}\|u_{ttt}\|_{r}^2 \mathrm{d}t.
        \end{align*}
        For $\theta = 1$, the family of DLN methods reduces to the midpoint rule. We have 
        \begin{align*}
        &\big\| u_{n,\beta} - u(t_{n,\beta}) \big\|_{r}^2
        = \Big\| \frac{u_{n+1} + u_{n}}{2} - u \Big(\frac{t_{n+1} + t_{n}}{2} \Big) \Big\|_{r}
        \leq
        C k_n^3 \int_{t_{n-1}}^{t_{n+1}}\|u_{tt}\|_{r}^2 \mathrm{d}t, \\
        &\Big\| \frac{u_{n,\alpha}}{\widehat{k}_n} - u_t(t_{n,\beta}) \Big\|_{r}^2
        = \Big\| \frac{u_{n+1} - u_{n}}{k_n} - u_{t} \Big(\frac{t_{n+1} + t_{n}}{2} \Big) \Big\|_{r}
        \leq C k_n^3 \int_{t_{n-1}}^{t_{n+1}}\|u_{ttt}\|_{r}^2 \mathrm{d}t.
        \end{align*}
    \end{lemma}
    \begin{proof}
        See \cite[Appendix B.]{CLPX2025_JSC} for the complete proof.
    \end{proof}

    \subsection{Error estimates for velocity and temperature in $L^{2}$-norm} 
    \ 
    \begin{assumption}\label{u_p_T_space}
        We assume that the exact solution $(u,w,p,T)$ in \eqref{TDAFEs} and the external body force $f$ satisfy the following regularity assumptions:
        \begin{align*}
        &u \in W^{3,2}\big([0,T]; [H^{r+2}(D)]^2 \big) \cap L^{\infty}\big([0,T]; [H^{r+2}(D)]^2 \big),
        \\
        &w \in W^{2,2}\big([0,T]; [H^{r+2}(D)]^2 \big) \cap L^{\infty}\big([0,T]; [H^{r+2}(D)]^2 \big),
        \\
        &p \in W^{2,2}\big([0,T]; H^{r+1}(D) \big) \cap L^\infty\big([0,T]; H^{r+1}(D)\big)
        \\
        &\phi  \in W^{2,2}\big([0,T]; H^{r+1}(D) \big) \cap L^\infty\big([0,T]; H^{r+1}(D)\big)
        \\
        &T  \in W^{2,2}\big([0,T]; H^{r+2}(D) \big) \cap L^\infty\big([0,T]; H^{r+2}(D)\big).
        \end{align*}
    \end{assumption}
    \begin{assumption}[Time step constraint] \label{time_condition_2}
        We assume that the weighted average time step $\widehat{k}_{n}$ satisfies: 
        for all $1 \leq n \leq M-1$ 
        \begin{align}
        \label{time_condition_eq2}
        C_{\beta}^{(n)} \widehat{k}_{n} \Big[ \Big(|\rho| +
        \frac{ 2{C^{\ast}}^{4} \nu^{4}}{\mu^{3}}  \|\nabla u_{n,\beta}\|^4
        +\frac{2 {C^{\ast}}^{4} }{9 \kappa^2 \mu} \|\nabla T_{n,\beta} \|^4
        \Big) + \frac{C_P^2 \sigma^2 |\xi|^2 }{\mu} \Big]
        < \frac{1+\theta}{4},
        \end{align}
        where $C_{\beta}^{(n)}$ is the consatnt defined in \eqref{time-cond-stab}, $C^{\ast} > 0$ the constant $C$ in \eqref{b 0 1 1 1} and $C_P$ the constant in Poincar\'e inequality.
    \end{assumption}

    \begin{theorem}\label{Intermediate conclusion}
        Suppose that Assumptions \ref{u_p_T_space} and \ref{time_condition_2} hold. For sufficiently small $k_n>0$, there exists a constant \( C \), independent of the mesh size \( h \) and the time step \( k_n \), such that the following error estimate holds:
        \begin{align} \label{proof of max u-u^h T-T^h error}
        \begin{split}
        &\max_{0 \leq n \leq M}  \big( \| e_n^u \|^2 +  \| e_n^T \|^2 \big)
        +  \sum_{n=1}^{M-1} \widehat{k}_n \big( \|\nabla e_{n,\beta}^u \|^2 + \|\nabla e_{n,\beta}^T \|^2  \big)
        +  \sum_{n=1}^{M-1} \widehat{k}_n \| e_{n,\beta}^{w} \|^2 \\
        &\leq \mathcal{O}(h^{2r+2} + k_{\max}^4),
        \end{split}
        \end{align}
        where $\displaystyle k_{\max} = \max_{0 \leq n \leq M-1} \{ k_{n} \}$.
    \end{theorem}

    We decompose the proof of Theorem \ref{Intermediate conclusion} into the proof of Lemma \ref{error estimate L2 lem} - \ref{truncation error estimate}. 
    \begin{lem}\label{error estimate L2 lem}
        Scheme \ref{fully discrete formulations scheme} has the following estimate: for $1 < m \leq M$,
        \begin{align}  \label{error estimate L2 inequality}
        &\frac{1}{4}(1+\theta)\|\psi^u_m\|^2 + \frac{1}{4}(1-\theta)\|\psi^u_{m-1}\|^2 
        + \sum_{n=1}^{m-1} \Big\| \sum_{\ell=0}^{2} a_{\ell}^{(n)} \psi^u_{n-1+\ell} \Big\|^2  \\
        &+\frac{1}{4}(1+\theta)\|\psi^T_m\|^2 + \frac{1}{4}(1-\theta)\|\psi^T_{m-1}\|^2 
        + \sum_{n=1}^{m-1} \Big\| \sum_{\ell=0}^{2} a_{\ell}^{(n)} \psi^T_{n-1+\ell} \Big\|^2 \notag \\
        &+ \sum_{n=1}^{m-1}  \widehat{k}_n \big\{
        \mu\| \nabla \psi^u_{n,\beta}\|^2
        +\gamma \|\psi^w_{n,\beta}\|^2
        +C \lambda  \|\psi^u_{n,\beta}\|^4_{L^4}
        + \kappa \| \nabla \psi^T_{n,\beta}\|^2
        \big\} \notag \\
        \leq &\! \sum_{n=1}^{m-1} |\rho|\widehat{k}_n \|\psi^u_{n,\beta}\|^2
        \!+\! \sum_{n=1}^{m-1} | (\eta^u_{n,\alpha} , \psi^u_{n,\beta})|
        \!+\! \sum_{n=1}^{m-1} |\rho|\widehat{k}_n |(\eta^u_{n,\beta}, \psi^u_{n,\beta})|
        \!+\! \sum_{n=1}^{m-1}\widehat{k}_n\sigma |\xi| |(\psi^T_{n,\beta},\psi^u_{n,\beta})|
        \notag \\
        & +\sum_{n=1}^{m-1}\widehat{k}_n \sigma |\xi| |(\eta^T_{n,\beta},\psi^u_{n,\beta}) |
        + \sum_{n=1}^{m-1} \widehat{k}_n \nu |b(u_{n,\beta}, u_{n,\beta}, \psi^u_{n,\beta}) - b(u_{n,\beta}^h, u_{n,\beta}^h, \psi^u_{n,\beta})| 
        \notag \\
        &+ \sum_{n=1}^{m-1} \widehat{k}_n  |b^*(u_{n,\beta}, T_{n,\beta}, \psi^T_{n,\beta}) - b^*(u_{n,\beta}^h, T_{n,\beta}^h, \psi^T_{n,\beta})| 
        \notag \\
        & + \sum_{n=1}^{m-1} \lambda \widehat{k}_n |(|u_{n,\beta}|^2 u_{n,\beta} - |\mathcal{S}_h u_{n,\beta}|^2\mathcal{S}_h u_{n,\beta}, u^h_{n,\beta} - \mathcal{S}_h u_{n,\beta})|
        \notag \\
        & + \sum_{n=1}^{m-1}\widehat{k}_n |\tau_n(u,w,p,\psi^u_{n,\beta},\psi^T_{n,\beta})|	
        + \frac{1}{4}(1+\theta)\|\psi^u_1\|^2 + \frac{1}{4}(1-\theta)\|\psi^u_{0}\|^2 \notag \\
        &+ \frac{1}{4}(1+\theta)\|\psi^T_1\|^2 + \frac{1}{4}(1-\theta)\|\psi^T_{0}\|^2, \notag
        \end{align}
        where the total truncation error $\tau_n$ is 
        \begin{align} \label{truncation}
        &\tau_n(u,w,p,T,\psi^u_{n,\beta},\psi^T_{n,\beta}) \\
        &= \! (\frac{1}{\widehat{k}_n} u_{n,\alpha} \!-\! u_t(t_{n,\beta}), v^h) 
        \!+\!\mu (\nabla (u_{n,\beta} \!-\! u(t_{n,\beta})), \nabla v^h) 
        \!+\!\nu b(u_{n,\beta}, u_{n,\beta}, v^h) \notag \\
        &-\! \nu b(u(t_{n,\beta}), u(t_{n,\beta}), v^h)
        +\!\rho (u_{n,\beta}-u(t_{n,\beta}), v^h)
        +\!\gamma(\nabla (w_{n,\beta} - w(t_{n,\beta})), \nabla v^h)
        \notag \\
        &+\lambda (|u_{n,\beta}|^2 u_{n,\beta}  - |u(t_{n,\beta})|^2 u(t_{n,\beta}), v^h)
        -\!(p_{n,\beta} - p(t_{n,\beta}), \nabla \cdot v^h) \!-\! \sigma \xi (T_{n,\beta} - T(t_{n,\beta}),v^h) \notag \\
        &+\!(\frac{1}{\widehat{k}_n} T_{n,\alpha} \!-\! T_t(t_{n,\beta}), \vartheta^h) 
        +\kappa (\nabla (T_{n,\beta} \!-\!T(t_{n,\beta})), \nabla \vartheta^h) 
        +\!(b^*(u_{n,\beta}, T_{n,\beta}, \vartheta^h) \notag \\
        &-\! b^*(u(t_{n,\beta}), T(t_{n,\beta})), \vartheta^h). \notag
        \end{align}
    \end{lem}
    \begin{proof}
        The true solution pair $(u,w, \phi,p,T)$ in \eqref{variational_formula_second_1} -\eqref{variational_formula_second_5} satisfies the following weak formulation: for all $(v^h, \varphi^h,\zeta^h, q^h,\vartheta^h) \in X_h \times  X_h \times Q_h \times Q_h \times J_h$
        \begin{align} 
        \begin{split}\label{formulation 1}
        &\frac{1}{\widehat{k}_n}(u_{n,\alpha}, v^h) 
        + \mu(\nabla u_{n,\beta}, \nabla v^h)
        + \gamma (\nabla w_{n,\beta}, \nabla v^h)
        + \nu b(u_{n,\beta}, u_{n,\beta}, v^h) 
        \\
        &~+\rho (u_{n,\beta}, v^h)
        +\lambda (|u_{n,\beta}|^2u_{n,\beta}, v^h)
        - (p_{n,\beta}, \nabla \cdot v^h) 
        \\
        &~= \sigma\xi(T_{n,\beta},v^h) +  \mathcal{T}_{n,1}(u,w,p,T, v^h),
        \end{split}	
        \\
        &(w_{n,\beta}, \varphi^h) -(\phi_{n,\beta},\nabla \cdot \varphi^h) - (\nabla u_{n,\beta},\nabla \varphi^h)
        =0,  \label{formulation 2}
        \\
        &(\nabla \cdot u_{n,\beta} , q^h)=0, \label{formulation 3}
        \\
        &(\nabla \cdot w_{n,\beta} , \zeta^h)=0, \label{formulation 4}
        \\
        &\frac{1}{\widehat{k}_n}(T_{n,\alpha}, \vartheta^h)
        +\kappa (\nabla T_{n,\beta}, \nabla \vartheta^h)
        +b^*(u_{n,\beta},  T_{n,\beta}, \vartheta^h)
        = \mathcal{T}_{n,2}(T,\vartheta^h).\label{formulation 5}
        \end{align}
        where $\mathcal{T}_{n,1}(u,w,p,T, v^h)$ and $\mathcal{T}_{n,2}(T,\vartheta^h)$ denote the truncation error
        \begin{align} 
        \begin{split}
        \mathcal{T}_{n,1}(u,w,p,T, v^h) = &\Big( \frac{u_{n,\alpha}}{\widehat{k}_n} - u_t(t_{n,\beta}), v^h \Big) 
        +\mu \big( \nabla (u_{n,\beta} -u(t_{n,\beta})), \nabla v^h \big) 
        \\
        &
        +\nu \big( b(u_{n,\beta}, u_{n,\beta}, v^h) - b(u(t_{n,\beta}), u(t_{n,\beta}), v^h) \big)
        \\
        &+\rho \big( u_{n,\beta}-u(t_{n,\beta}), v^h \big)
        +\gamma \big( \nabla (w_{n,\beta} - w(t_{n,\beta})), \nabla v^h \big)
        \\
        &+\lambda \big( |u_{n,\beta}|^2 u_{n,\beta}  - |u(t_{n,\beta})|^2 u(t_{n,\beta}), v^h \big)
        \\
        &- \big( p_{n,\beta} - p(t_{n,\beta}), \nabla \cdot v^h \big)- \sigma \xi \big( T_{n,\beta} - T(t_{n,\beta}),v^h \big).
        \end{split}
        \\
        \begin{split}\label{truncation T}
        \mathcal{T}_{n,2}(T,\vartheta^h)
        = & \Big( \frac{T_{n,\alpha}}{\widehat{k}_n} - T_t(t_{n,\beta}), \vartheta^h \Big) 
        + \kappa \big( \nabla (T_{n,\beta} -T(t_{n,\beta})), \nabla \vartheta^h \big) 
        \\
        &+ b^*(u_{n,\beta}, T_{n,\beta}, \vartheta^h) - b^* \big( u(t_{n,\beta}), T(t_{n,\beta}), \vartheta^h \big).
        \end{split}
        \end{align}
        We subtract \eqref{formulation 1} from \eqref{fully discrete formulations 1}, \eqref{formulation 5} from
        \eqref{fully discrete formulations 5} and use the definions of Ritz projection \eqref{Ritz-projection-defination}, Stokes-type projection \eqref{Stokes-type-projection-defination} to obtain
        \begin{align}
            &\frac{1}{\widehat{k}_n}(\psi^u_{n,\alpha}, v^h) 
            + \mu (\nabla \psi^u_{n,\beta}, \nabla v^h) 
            +\gamma (\nabla \psi^w_{n,\beta},\nabla v^h)
            + \rho (\psi^u_{n,\beta},v^h) 
            \label{error equation 1} \\
            =& \Big(\frac{\eta^u_{n,\alpha}}{\widehat{k}_n}, v^h \Big) 
            \!+\!\nu b(u_{n,\beta}, u_{n,\beta}, v^h)
            \!-\!\nu b(u_{n,\beta}^h, u_{n,\beta}^h, v^h)
            \!+\! (\psi^p_{n,\beta}, \nabla \cdot v^h)
            \!+\! \rho (\eta^u_{n,\beta},v^h)
            \notag  \\
            &+ \! \lambda (|u_{n,\beta}|^2 u_{n,\beta}  - |u_{n,\beta}^h|^2 u_{n,\beta}^h ,v^h) 
            \!+\! \sigma\xi (\psi^T_{n,\beta} - \eta^T_{n,\beta},v^h) - \mathcal{T}_{n,1}(u,w,p,T, v^h), \notag \\
            &\frac{1}{\widehat{k}_n} (\psi^T_{n,\alpha}, \vartheta^h) 
            + \kappa (\nabla \psi^T_{n,\beta}, \nabla \vartheta^h) 
            \label{error equation 5} \\
            &= \Big(\frac{\eta^T_{n,\alpha}}{\widehat{k}_n}, \vartheta^h \Big)
            + b^*(u_{n,\beta}, T_{n,\beta}, \vartheta^h)
            - b^*(u_{n,\beta}^h, T_{n,\beta}^h, \vartheta^h)
            - \mathcal{T}_{n,2}(T,\vartheta^h). \notag 
        \end{align}
        We subtract \eqref{formulation 2} - \eqref{formulation 4} from a suitable linear combination of \eqref{fully discrete formulations 2} - \eqref{fully discrete formulations 4} respectively and use the definition of Stokes-type projection \eqref{Stokes-type-projection-defination} to have 
        \begin{align}
            &(\psi^w_{n,\beta}, \varphi^h) 
            - \big( \phi_{n,\beta}^h - \mathcal{S}_h (\phi_{n,\beta}),\nabla \cdot \varphi^h \big)
            - (\nabla \psi^u_{n,\beta}, \nabla \varphi^h)
            =0, \label{error equation 2}
            \\
            &(\nabla \cdot \psi^u_{n,\beta} , q^h)=0.
            \label{error equation 3}
            \\
            &(\nabla \cdot \psi^w_{n,\beta} , \zeta^h)=0,
            \label{error equation 4}
        \end{align}
        We set \( v^h = \psi^u_{n,\beta} \) in \eqref{error equation 1},  \( \vartheta^h = \psi^T_{n,\beta} \) in \eqref{error equation 5}, \( \varphi^h = \psi^w_{n,\beta} \) in \eqref{error equation 2}, \( q^h = \psi^p_{n,\beta} \) in \eqref{error equation 3}, \( \zeta^h = \phi_{n,\beta}^h - \mathcal{S}_h (\phi_{n,\beta}) \) in \eqref{error equation 4}, multiply \eqref{error equation 1}, \eqref{error equation 5} by \( \widehat{k}_n \), add \eqref{error equation 1}, \eqref{error equation 5} and use \eqref{error equation 2} - \eqref{error equation 4} 
        \begin{align} \label{total-error-eq1}
        \begin{split}
            &(\psi^u_{n,\alpha}, v^h) \!+\! (\psi^T_{n,\alpha}, \vartheta^h) 
            \!+\! \lambda (|u_{n,\beta}^h|^2 u_{n,\beta}^h \!-\! |\mathcal{S}_h (u_{n,\beta})|^2 \mathcal{S}_h (u_{n,\beta}) ,v^h) \\
            &+\! \rho \widehat{k}_n \|\psi^u_{n,\beta}\|^2
            \!+\! \mu \widehat{k}_n \| \nabla \psi^u_{n,\beta}\|^2
            \!+\! \gamma \widehat{k}_n \|\psi^w_{n,\beta}\|^2
            \!+\! \kappa  \widehat{k}_n \| \nabla \psi^T_{n,\beta}\|^2
            \\
            &=(\eta^u_{n,\alpha} , \psi^u_{n,\beta}) + (\eta^T_{n,\alpha} , \psi^T_{n,\beta})
            + \rho \widehat{k}_n (\eta^u_{n,\beta}, \psi^u_{n,\beta})
            +\widehat{k}_n \sigma\xi(\psi^T_{n,\beta},\psi^u_{n,\beta}) \\
            &-\widehat{k}_n \sigma\xi(\eta^T_{n,\beta},\psi^u_{n,\beta})  
            + \!\lambda ( |\mathcal{S}_h (u_{n,\beta})|^2 \mathcal{S}_h (u_{n,\beta}) \!-\! |u_{n,\beta}|^2 u_{n,\beta} ,v^h)
            \\
            &+ \widehat{k}_n \nu \big( b(u_{n,\beta}, u_{n,\beta}, \psi^u_{n,\beta}) - b(u_{n,\beta}^h, u_{n,\beta}^h, \psi^u_{n,\beta}) \big)
            - \widehat{k}_n \mathcal{T}_{n,1}(u,w,p,\psi^u_{n,\beta})  \\
            &+ \widehat{k}_n \big( b^*(u_{n,\beta}, T_{n,\beta}, \psi^T_{n,\beta}) - b^*(u_{n,\beta}^h, T_{n,\beta}^h, \psi^T_{n,\beta}) \big)
            - \widehat{k}_n \mathcal{T}_{n,2}(T,\psi^T_{n,\beta}).
        \end{split}
        \end{align}
        We apply $G$-stability identity in \eqref{G-stable}, monotonicity property in \eqref{vector_norm_inequality_1} to \eqref{total-error-eq1} and the fact that 
        \begin{align*}
            \tau_n(u,w,p,T,\psi^u_{n,\beta},\psi^T_{n,\beta}) 
            = \mathcal{T}_{n,1}(u,w,p,\psi^u_{n,\beta}) + \mathcal{T}_{n,2}(T,\psi^T_{n,\beta}),
        \end{align*}
        to obtain
        \begin{align*}
            &
            \begin{Vmatrix}
            \psi^u_{n+1} \vspace{0.5mm} \\ 
            \psi^u_{n}
            \end{Vmatrix}^2_{G(\theta)}
            -
            \begin{Vmatrix}
            \psi^u_{n} \vspace{0.5mm} \\ 
            \psi^u_{n-1}
            \end{Vmatrix}^2_{G(\theta)}
            + \Big\| \sum_{\ell=0}^{2} a_{\ell}^{(n)} \psi^u_{n-1+\ell} \Big\|^2
            \\
            &+
            \begin{Vmatrix}
            \psi^T_{n+1} \vspace{0.5mm} \\ 
            \psi^T_{n}
            \end{Vmatrix}^2_{G(\theta)}
            -
            \begin{Vmatrix}
            \psi^T_{n} \vspace{0.5mm} \\ 
            \psi^T_{n-1}
            \end{Vmatrix}^2_{G(\theta)}
            + \Big\| \sum_{\ell=0}^{2} a_{\ell}^{(n)} \psi^T_{n-1+\ell} \Big\|^2
            \\
            &+ \widehat{k}_n \big( \mu \| \nabla \psi^u_{n,\beta}\|^2 + \gamma  \|\psi^w_{n,\beta}\|^2
            + C \lambda   \|\psi^u_{n,\beta}\|^4_{L^4} + \kappa   \| \nabla \psi^T_{n,\beta}\|^2  \big) \\
            \leq& |\rho|\widehat{k}_n \|\psi^u_{n,\beta}\|^2
            + | (\eta^u_{n,\alpha} , \psi^u_{n,\beta})|
            + |\rho|\widehat{k}_n |(\eta^u_{n,\beta}, \psi^u_{n,\beta})|
            + \widehat{k}_n\sigma |\xi| |(\psi^T_{n,\beta},\psi^u_{n,\beta})|
            \notag
            \\
            &+ \widehat{k}_n \sigma |\xi| |(\eta^T_{n,\beta},\psi^u_{n,\beta}) |
            +  | (\eta^T_{n,\alpha}, \psi^T_{n,\beta})|
            + \widehat{k}_n \nu |b(u_{n,\beta}, u_{n,\beta}, \psi^u_{n,\beta}) - b(u_{n,\beta}^h, u_{n,\beta}^h, \psi^u_{n,\beta})| 
            \notag \\
            &+ \widehat{k}_n  |b^*(u_{n,\beta}, T_{n,\beta}, \psi^T_{n,\beta}) - b^*(u_{n,\beta}^h, T_{n,\beta}^h, \psi^T_{n,\beta})| 
            \notag
            \\
            &+ \widehat{k}_n |(|u_{n,\beta}|^2 u_{n,\beta} - |\mathcal{S}_h u_{n,\beta}|^2\mathcal{S}_h u_{n,\beta}, u^h_{n,\beta} - \mathcal{S}_h u_{n,\beta})|
            \notag
            \\
            &+ \widehat{k}_n |\tau_n(u,w,p,\psi^u_{n,\beta},\psi^T_{n,\beta})|. 
        \end{align*}
        We sum the above inquality over $n$ from $1$ to $m-1$ and utilize the definition of $G$-norm in \eqref{G-norm} to achieve \eqref{error estimate L2 inequality}.
    \end{proof}

    We address the terms on the right side of inequality \eqref{error estimate L2 inequality} through Lemma \ref{derivation term} - \ref{truncation error estimate}.
    \begin{lem}\label{derivation term}
        Suppose Assumption \ref{u_p_T_space} holds. For $ 1 \leq n \leq M-1 $, 
        \begin{align} \label{derivation-term-conclusion}
        & |(\eta^u_{n,\alpha}, \psi^u_{n,\beta})| + |\rho| \widehat{k}_n |(\eta^u_{n,\beta}, \psi^u_{n,\beta})|
        +| (\eta^T_{n,\alpha}, \psi^T_{n,\beta})| \\
        \leq& \frac{ C(\theta) \widehat{k}_n \rho^2 h^{2r+4}}{\mu} \Big[ \big( \| u \|_{\infty,r+2} +  \| w \|_{\infty,r+2} + \| p \|_{\infty,r+1} \big)^{2} + \| T \|_{\infty,r+2}^2 \Big] \notag \\
        &+\frac{C(\theta) h^{2r+4}}{\mu} \int_{t_{n-1}}^{t_{n+1}} \big( \| u_t \|_{r+2}^2 + \| w_t \|_{r+2}^2 + \| p_t \|_{r+1}^2 \big) \mathrm{d}t \notag \\
        &+ \frac{C(\theta) h^{2r+4}}{\kappa} \int_{t_{n-1}}^{t_{n+1}} \| T_t \|_{r+2}^2 \mathrm{d}t 
        + \frac{\mu}{64}\widehat{k}_n \| \nabla \psi^u_{n,\beta} \|^2 + \frac{\kappa}{32}\widehat{k}_n \| \nabla \psi^T_{n,\beta} \|^2, \notag
        \end{align}
        where constants $C(\theta) > 0$ only depend on the parameter $\theta$.
    \end{lem}
    \begin{proof}
        We first prove the case $\theta \in [0,1)$. By Cauchy-Schwarz inequality, approximation of Stokes-type projection $\mathcal{S}_h$ in \eqref{Stokes-type projection}
        \begin{align} \label{derivation-term-eq1}
        \begin{split}
        |(\eta^u_{n,\alpha}, \psi^u_{n,\beta})| 
        &\leq C h^{r+2} \big( \| u_{n,\alpha} \|_{r+2} + \| w_{n,\alpha} \|_{r+2} 
        + \| p_{n,\alpha} \|_{r+1} \big) \| \nabla \psi^u_{n,\beta} \|.
        \end{split}
        \end{align}
        By fundamental theorem of Calculus, triangle inequality and H\"older's inequality
        \begin{align} \label{derivation-term-eq2}
        \begin{split}
        \| u_{n,\alpha} \|_{r+2} =&
        \Big\| \frac{1}{2}(\theta - 1) u_{n-1} - \theta u_{n} + \frac{1}{2}(\theta + 1) u_{n+1} \Big\|_{r+2}    \\
        \leq & \frac{1}{2}(\theta + 1) \Big\| \int_{t_n}^{t_{n+1}} u_t \, \mathrm{d}t \Big\|_{r+2} + \frac{1}{2}(1 - \theta) \Big\| \int_{t_{n-1}}^{t_n} u_t \, \mathrm{d}t \Big\|_{r+2}  \\
        \leq & C(\theta) k_{n}^{\frac{1}{2}} \Big( \int_{t_n}^{t_{n+1}} \| u_t \|_{r+2}^2 \, \mathrm{d}t \Big)^{\frac{1}{2}} + C(\theta) k_{n-1}^{\frac{1}{2}} \Big( \int_{t_{n-1}}^{t_n} \| u_t \|_{r+2}^2 \, \mathrm{d}t \Big)^{\frac{1}{2}},  \\
        \| w_{n,\alpha} \|_{r+2} 
        \leq & C(\theta) k_{n}^{\frac{1}{2}} \Big( \int_{t_n}^{t_{n+1}} \| w_t \|_{r+2}^2 \, \mathrm{d}t \Big)^{\frac{1}{2}} + C(\theta) k_{n-1}^{\frac{1}{2}} \Big( \int_{t_{n-1}}^{t_n} \| w_t \|_{r+2}^2 \, \mathrm{d}t \Big)^{\frac{1}{2}},  \\
        \| p_{n,\alpha} \|_{r+1} 
        \leq & C(\theta) k_{n}^{\frac{1}{2}} \Big( \int_{t_n}^{t_{n+1}} \| p_t \|_{r+1}^2 \, \mathrm{d}t \Big)^{\frac{1}{2}} + C(\theta) k_{n-1}^{\frac{1}{2}} \Big( \int_{t_{n-1}}^{t_n} \| p_t \|_{r+1}^2 \, \mathrm{d}t \Big)^{\frac{1}{2}}. 
        \end{split}
        \end{align}
        We combine \eqref{derivation-term-eq1}, \eqref{derivation-term-eq2} and use Young's inequality to have  
        \begin{equation} \label{derivation-term-eq3}
        |(\eta^u_{n,\alpha}, \psi^u_{n,\beta})| 
        \leq \frac{C(\theta) h^{2r+4}}{\mu} \int_{t_{n-1}}^{t_{n+1}} \big( \| u_t \|_{r+2}^2 + \| w_t \|_{r+2}^2 + \| p_t \|_{r+1}^2\big) \, \mathrm{d}t 
        +  \frac{\mu}{128}\widehat{k}_n \|\nabla  \psi^u_{n,\beta} \|^2.
        \end{equation}
        Similarly
        \begin{equation} \label{derivation-term-T-eq3}
        |(\eta^T_{n,\alpha}, \psi^T_{n,\beta})| 
        \leq \frac{C(\theta)  h^{2r+4}}{\kappa} \int_{t_{n-1}}^{t_{n+1}} \| T_t \|_{r+2}^2 \mathrm{d}t 
        +  \frac{\kappa}{32}\widehat{k}_n \|\nabla  \psi^T_{n,\beta} \|^2.
        \end{equation}
        By Cauchy-Schwarz inequality and Young's inequality
        \begin{align} \label{derivation-term-eq4}
        \begin{split}
        &|\rho| \widehat{k}_n |(\eta^u_{n,\beta}, \psi^u_{n,\beta})| \\
        \leq& C \rho^2 \mu^{-1} \widehat{k}_n \|\eta^u_{n,\beta}\|^2 
        + \frac{\mu}{128}\widehat{k}_n \|\nabla\psi^u_{n,\beta}\|^2  \\
        \leq& \frac{C(\theta) \widehat{k}_n \rho^2 h^{2r+4}}{\mu} \big( \| u \|_{\infty,r+2} +  \| w \|_{\infty,r+2} + \| p \|_{\infty,r+1} \big)^{2} 
        + \frac{\mu}{128}\widehat{k}_n \|\nabla \psi^u_{n,\beta}\|^2. 
        \end{split}
        \end{align}
        We combine \eqref{derivation-term-eq3} - \eqref{derivation-term-eq4} to 
        obtain \eqref{derivation-term-conclusion}.
        If $\theta = 1$, $\widehat{k}_n = k_{n}$ and \eqref{derivation-term-eq2} becomes  
        \begin{align*}
        \| u_{n,\alpha} \|_{r+2} =
        \big\| u_{n+1} - u_{n} \big\|_{r+2} \leq k_{n}^{\frac{1}{2}} \Big( \int_{t_n}^{t_{n+1}} \| u_t \|_{r+2}^2 \, \mathrm{d}t \Big)^{\frac{1}{2}}.
        \end{align*}
        Then the proof is very similar to the case of $\theta \in [0,1)$.
    \end{proof}

    \begin{lem}\label{temperature term}
        Under Assumption \ref{u_p_T_space}, the following bound holds:
        \begin{align} \label{temperature term-conclusion}
        \begin{split}
        &\widehat{k}_n \sigma |\xi| (\psi^T_{n,\beta},\psi^u_{n,\beta})
        + \widehat{k}_n \sigma |\xi| (\eta^T_{n,\beta},\psi^u_{n,\beta}) \\
        \leq &  \frac{ C(\theta) \widehat{k}_n \sigma^2 |\xi |^2 h^{2r+4}}{\mu}  \| T \|_{\infty,r+2} ^{2}
        + \frac{C_P^2 \widehat{k}_n \sigma^2 |\xi|^2}{\mu} \| \psi^T_{n,\beta}\|^2 
        +\frac{17 \mu}{64}\widehat{k}_n \| \nabla \psi^u_{n,\beta}\|^2,
        \end{split}
        \end{align}
        for $1 \leq n \leq M-1$.
    \end{lem}
    \begin{proof}
        By Cauchy-Schwarz inequality, Young's inequality, Poincar\'e inequality and approximation of Stokes-type projection $\mathcal{S}_h$ in \eqref{Stokes-type projection}
        \begin{align} \label{temperature term 1}
        \widehat{k}_n \sigma |\xi| (\eta^T_{n,\beta},\psi^u_{n,\beta})
        \leq& \frac{ C(\theta) \widehat{k}_n \sigma^2 |\xi |^2 }{\mu} \|\eta^T_{n,\beta}\|^2
        + \frac{\mu}{64}\widehat{k}_n \|\nabla\psi^u_{n,\beta}\|^2 
        \\
        \leq &  \frac{ C(\theta) \widehat{k}_n \sigma^2 |\xi |^2 h^{2r+4}}{\mu}  \| T \|_{\infty,r+2} ^{2}+\frac{\mu}{64}\widehat{k}_n \| \nabla \psi^u_{n,\beta}\|^2. \notag 
        \end{align}
        \begin{align} \label{temperature term 2}
        \widehat{k}_n \sigma |\xi| (\psi^T_{n,\beta},\psi^u_{n,\beta})
        \leq
        \frac{C_P^2}{\mu}\widehat{k}_n \sigma^2 |\xi|^2 \| \psi^T_{n,\beta}\|^2 
        + \frac{\mu \widehat{k}_n }{4}  \| \nabla \psi^u_{n,\beta}\|^2 .
        \end{align}
        We combine \eqref{temperature term 1} - \eqref{temperature term 2} 
        to derive \eqref{temperature term-conclusion}.
    \end{proof}

    \begin{lem}\label{analytical b - numerical b}
        Under Assumption \ref{u_p_T_space}, the following bounds hold:
        \begin{align} 
        \begin{split} \label{analytical b - numerical b-conclusion1} 
        &\nu \big( b(u_{n,\beta}, u_{n,\beta}, \psi^u_{n,\beta}) - b(u_{n,\beta}^h, u_{n,\beta}^h, \psi^u_{n,\beta}) \big) 
        \\
        \leq& \frac{ 2 {C^{\ast}}^4 \nu^{4}}{\mu^{3}}  \|\nabla u_{n,\beta}\|^4 \|\psi^u_{n,\beta}\|^2 
        + \frac{25 \mu}{64} \|\nabla \psi^u_{n,\beta}\|^2 \\
        &+ \frac{C(\theta) \nu^{2} h^{2r+4}}{\mu} \| u \|_{\infty,2}^{2} 
        \big( \| u \|_{\infty,r+2} + \| w \|_{\infty,r+2} 
        + \| p \|_{\infty,r+1} \big)^{2}  \\
        &+ \!\frac{C  (\theta) \nu^{2} h^{2r+2}}{\mu} \| \nabla u_{n,\beta}^h \|^{2}
        \big( \| u \|_{\infty,r+2} + \| w \|_{\infty,r+2} + \| p \|_{\infty,r+1} \big)^{2}.
        \end{split}
        \\
        \begin{split} \label{analytical b - numerical b-conclusion2} 
        & b^*(u_{n,\beta}, T_{n,\beta}, \psi^T_{n,\beta}) - b^*(u_{n,\beta}^h, T_{n,\beta}^h, \psi^T_{n,\beta}) 
        \\
        \leq & \frac{2 {C^{\ast}}^4}{9 \kappa^2  \mu} \|\psi_{n,\beta}^u\|^2 \|\nabla T_{n,\beta}\|^4
        + \frac{1}{8}\mu \|\nabla \psi^u_{n,\beta}\|^2 + \frac{13}{16}\kappa \|\nabla \psi^T_{n,\beta}\|^2 \\
        &+ \frac{C(\theta) h^{2r+4}}{\kappa} \| T \|_{\infty,2}^{2} 
        \big( \| u \|_{\infty,r+2} + \| w \|_{\infty,r+2} 
        + \| p \|_{\infty,r+1} \big)^{2}
        \\
        &+ \frac{C(\theta) h^{2r+2}}{\kappa} \| \nabla u_{n,\beta}^h \|^{2}
        \| T \|_{\infty,r+2}^{2}.
        \end{split}
        \end{align}
        for $ 1 \leq n \leq M-1 $. Here $C^{\ast} > 0$ denotes the constant $C$ in \eqref{b 0 1 1 1}.
    \end{lem}
    \begin{proof}
        The skew-symmetric property of the operator $b$ yields
        \begin{align}
        \begin{split}\label{b-b inequality}
        & \nu b(u_{n,\beta}, u_{n,\beta}, \psi^u_{n,\beta}) - \nu b(u_{n,\beta}^h, u_{n,\beta}^h, \psi^u_{n,\beta}) \\
        =& \nu b(u_{n,\beta} - u_{n,\beta}^h, u_{n,\beta}, \psi^u_{n,\beta}) 
        + \nu b(u_{n,\beta}^h, u_{n,\beta} - u_{n,\beta}^h, \psi^u_{n,\beta}) \\
        =& \nu b(\eta^u_{n,\beta}, u_{n,\beta}, \psi^u_{n,\beta}) - \nu b(\psi^u_{n,\beta}, u_{n,\beta}, \psi^u_{n,\beta}) + \nu b(u_{n,\beta}^h, \eta^u_{n,\beta}, \psi^u_{n,\beta}) 
        %=& \mathcal{I}_1+\mathcal{I}_2+\mathcal{I}_3.
        \end{split}
        \\
        \begin{split}\label{b-b inequality *}
        & b^*(u_{n,\beta}, T_{n,\beta}, \psi^T_{n,\beta}) - b^*(u_{n,\beta}^h, T_{n,\beta}^h, \psi^T_{n,\beta}) \\
        =& b^*(\eta^u_{n,\beta}, T_{n,\beta}, \psi^T_{n,\beta}) - b^*(\psi^u_{n,\beta}, T_{n,\beta}, \psi^T_{n,\beta}) + b^*(u_{n,\beta}^h, \eta^T_{n,\beta}, \psi^T_{n,\beta}) 
        \end{split}
        \end{align}
        %For the term $\mathcal{I}_1 = b(\eta^u_{n,\beta}, u_{n,\beta}, \psi^u_{n,\beta})$, 
        We utilize the \eqref{b 0 2 1}, Poincar\'e inequality and Young's inequality to obtain:
        \begin{align} 
        \begin{split}
        \nu b(\eta^u_{n,\beta}, u_{n,\beta}, \psi^u_{n,\beta})
        &\leq C \nu \|\eta^u_{n,\beta}\| \|u_{n,\beta}\|_2 \|\psi^u_{n,\beta}\|_1 
        \label{b-term1} \\
        &\leq C \nu^{2} \mu^{-1} \|\eta^u_{n,\beta}\|^2 \|u_{n,\beta}\|_2^2 + \frac{1}{128}\mu \|\nabla \psi^u_{n,\beta}\|^2. 
        \end{split}	
        \\
        \begin{split}
        b^*(\eta^u_{n,\beta}, T_{n,\beta}, \psi^T_{n,\beta})
        &\leq C \kappa^{-1} \|\eta^u_{n,\beta}\|^2 \|T_{n,\beta}\|_2^2 + \frac{1}{32}\kappa \|\nabla \psi^T_{n,\beta}\|^2,
        \end{split}
        \end{align}
        We make use of \eqref{b 0 1 1 1} and Young's inequality
        \begin{align} \label{b-term2} 
        \begin{split}
        - \nu b(\psi^u_{n,\beta}, u_{n,\beta}, \psi^u_{n,\beta}) 
        \leq& C^{\ast} \nu \|\psi^u_{n,\beta}\|^{\frac{1}{2}}\|\nabla \psi^u_{n,\beta}\|^{\frac{1}{2}} \|\nabla u_{n,\beta}\| \|\nabla \psi^u_{n,\beta}\| \\
        \leq& \frac{{ 2 C^{\ast}}^{4} \nu^{4} } {\mu^3}  \|\psi^u_{n,\beta}\|^2 \|\nabla u_{n,\beta}\|^4 + \frac{3}{8}\mu \|\nabla \psi^u_{n,\beta}\|^2, \\
        - b^*(\psi^u_{n,\beta}, T_{n,\beta}, \psi^T_{n,\beta}) 
        \leq& C^{\ast} \|\psi^u_{n,\beta}\|^{\frac{1}{2}}\|\nabla \psi^u_{n,\beta}\|^{\frac{1}{2}} \|\nabla T_{n,\beta}\| \|\nabla \psi^T_{n,\beta}\| 
        \\
        \leq&  \frac{{C^{\ast}}^{2}}{3\kappa} \|\psi_{n,\beta}^u\| \|\nabla \psi_{n,\beta}^u\| \|\nabla T_{n,\beta}\|
        + \frac{3}{4}\kappa \|\nabla \psi^T_{n,\beta}\|^2
        \\
        \leq& \frac{2}{9} {C^{\ast}}^{4} \kappa^{-2}\mu^{-1} \|\psi_{n,\beta}^u\|^2 \|\nabla T_{n,\beta}\|^4
        + \frac{1}{8}\mu \|\nabla \psi^u_{n,\beta}\|^2 \\
        &+ \frac{3}{4}\kappa \|\nabla \psi^T_{n,\beta}\|^2.
        \end{split}
        \end{align}
        By \eqref{b 1 1 1} and Young's inequality
        \begin{align} \label{b-term3} 
        \begin{split}
            \nu b(u_{n,\beta}^h, \eta^u_{n,\beta}, \psi^u_{n,\beta}) 
            \leq& C \nu \| \nabla u_{n,\beta}^h \| \| \nabla \eta^u_{n,\beta} \| \| \nabla \psi^u_{n,\beta} \| \\
            \leq& C \nu^2 \mu^{-1}\| \nabla u_{n,\beta}^h \|^{2}  \| \nabla \eta^u_{n,\beta} \|^{2} + \frac{1}{128}\mu \|\nabla \psi^u_{n,\beta}\|^2, \\
            b^*(u_{n,\beta}^h, \eta^T_{n,\beta}, \psi^T_{n,\beta})
            \leq &C \kappa^{-1}\|\nabla u_{n,\beta}^h\|^2 \|\nabla \eta^T_{n,\beta}\|^2 + \frac{1}{32}\kappa \|\nabla \psi^T_{n,\beta}\|^2.
        \end{split}
        \end{align}
        We combine \eqref{b-b inequality} - \eqref{b-term3} and use approximation of Stokes-type projection in \eqref{Stokes-type projection} again
        to derive \eqref{analytical b - numerical b-conclusion1} and \eqref{analytical b - numerical b-conclusion2}.
    \end{proof}

    \begin{lem}\label{cubic term}
        Under Assumption \ref{u_p_T_space}, the following bound holds 
        \begin{align}
        &\lambda |(|u_{n,\beta}|^2 u_{n,\beta} - |\mathcal{S}_h u_{n,\beta}|^2\mathcal{S}_h u_{n,\beta},\psi_{n,\beta}^u )|
        \label{cubic term inequality} \\
        \leq & \frac{C(\theta) \lambda^2 h^{2r+2}}{\mu} \Big[ \| u \|_{L^{\infty}(L^{4})}^{4} 
        + h^{4r+4} \big( \| u \|_{\infty,r+2} 
        + \| w \|_{\infty,r+2} + \| p \|_{\infty,r+1} \big)^{4} \Big] \times \notag \\
        &\times \big( \| u \|_{\infty,r+2} 
        + \| w \|_{\infty,r+2} + \| p \|_{\infty,r+1} \big)^{2}
        + \frac{1}{64} \mu \|\nabla \psi^u_{n,\beta}\|^2. \notag 
        \end{align}
        for $1 \leq n \leq M-1$.
    \end{lem}
    \begin{proof}
        By \eqref{vector_norm_inequality_2}, H\"older's inequality, Sobolev imbedding inequality, Poincar\'e inequality and triangle inequality 
        \begin{align}\label{cubic term inequality 1}
        &|(|u_{n,\beta}|^2 u_{n,\beta} - |\mathcal{S}_h u_{n,\beta}|^2\mathcal{S}_h u_{n,\beta}, \psi_{n,\beta}^u)| \notag
        \\
        \leq
        & C \Big(
        \int_{D} (|u_{n,\beta}|^2 u_{n,\beta} -|\mathcal{S}_h u_{n,\beta}|^2 \mathcal{S}_h u_{n,\beta})^{\frac{4}{3}}\mathrm{d}x
        \Big)^\frac{3}{4}
        \Big(
        \int_{D}| \psi_{n,\beta}^u|^4 \mathrm{d}x
        \Big)^\frac{1}{4} \notag
        \\
        \leq
        &
        C \bigg[
        \Big(
        \int_{D} (|u_{n,\beta}| + |\mathcal{S}_h u_{n,\beta}| )^4 \mathrm{d}x
        \Big)^\frac{2}{3}
        \Big(
        \int_{D} |u_{n,\beta} - \mathcal{S}_h u_{n,\beta}|^4 \mathrm{d}x
        \Big)^\frac{1}{3}
        \bigg]^\frac{3}{4}
        \|\psi_{n,\beta}^u\|_{L^4}
        \\
        \leq
        & C 
        \Big(
        \int_{D} (|u_{n,\beta}| + |\mathcal{S}_h u_{n,\beta}| )^4 \mathrm{d}x
        \Big)^\frac{1}{2} \|u_{n,\beta}-\mathcal{S}_h u_{n,\beta}\|_{L^4}
        \|\psi_{n,\beta}^u\|_{L^4} \notag
        \\ 
        \leq
        & C (\|u_{n,\beta}\|_{L^4(D)}^2 +\|\mathcal{S}_h u_{n,\beta}\|_{L^{4}(D)}^2 )
        \|\nabla \eta^u_{n,\beta}\| \|\nabla \psi^u_{n,\beta} \| \notag \\
        \leq& C (\|u_{n,\beta}\|_{L^4(D)}^2 + \|\eta^{u}_{n,\beta}\|_{1}^2 ) \|\nabla \eta^u_{n,\beta}\| \|\nabla \psi^u_{n,\beta}\|. \notag 
        \end{align}
        We apply Young's inequality and approximation in \eqref{Stokes-type projection} to \eqref{cubic term inequality 1} and  
        obtain \eqref{cubic term inequality}.
    \end{proof} 

    Now we deal with the terms in truncation error $\tau_n(u,w,p,T,\psi^u_{n,\beta},\psi^T_{n,\beta})$ in \eqref{truncation}. 
    \begin{lem}\label{analytical b - exact b with t n beta}
        Under \Cref{u_p_T_space}, we have 
        \begin{align}
        \begin{split}\label{analytical b - exact b with t n beta equation 1}
        &\nu \big(
        b(u_{n,\beta}, u_{n,\beta}, \psi^u_{n,\beta}) - b(u(t_{n,\beta}), u(t_{n,\beta}), \psi^u_{n,\beta})
        \big)
        \\
        \leq&
        \frac{\mu}{128}  \| \nabla \psi^u_{n,\beta} \|^2
        + \frac{C(\theta) \nu^2}{\mu} \|u\|_{\infty,1}^2 k_{\rm{max}}^3 
        \int_{t_{n-1}}^{t_{n+1}} \| \nabla u_{tt}\|^2 \mathrm{d}t,
        \end{split}
        \\
        \begin{split}\label{analytical b - exact b with t n beta equation 2}
        &b^*(u_{n,\beta}, T_{n,\beta}, \psi^T_{n,\beta}) - b^*(u(t_{n,\beta}), T(t_{n,\beta}), \psi^T_{n,\beta})
        \\
        \leq
        &\frac{\kappa}{64}  \| \nabla \psi^T_{n,\beta} \|^2
        + \frac{C(\theta)}{\kappa}
        \|u\|_{\infty,1}^2
        k_{\rm{max}}^3 
        \int_{t_{n-1}}^{t_{n+1}}
        \| \nabla T_{tt}\|^2 \mathrm{d}t
        \\&+   \frac{C(\theta)}{\kappa}
        \|T\|_{\infty,1}^2
        k_{\rm{max}}^3 
        \int_{t_{n-1}}^{t_{n+1}}
        \| \nabla u_{tt}\|^2 \mathrm{d}t,
        \end{split}
        \end{align}
        for $1 \leq n \leq M-1$.
    \end{lem}
    \begin{proof}
        We reformulate the difference between the two nonlinear terms as follows:
        \begin{align}
        \begin{split}\label{b-b n beta inequality}
        &\nu \big(
        b(u_{n,\beta}, u_{n,\beta}, \psi^u_{n,\beta}) - b(u(t_{n,\beta}), u(t_{n,\beta}), \psi^u_{n,\beta})
        \big)
        \\
        =&\nu b \big(u_{n,\beta} - u(t_{n,\beta}), u_{n,\beta}, \psi^u_{n,\beta} \big) 
        + \nu  b \big(u(t_{n,\beta}), u_{n,\beta} - u(t_{n,\beta}), \psi^u_{n,\beta} \big).
        \end{split}
        \end{align}
        We apply \eqref{b 1 1 1}, Lemma \ref{n beta int tn-1 tn+1} and Young's inequality to \eqref{b-b n beta inequality}
        \begin{align} \label{b-b n beta inequality second}
        \begin{split}
        &\nu \big( b(u_{n,\beta} - u(t_{n,\beta}), u_{n,\beta}, \psi^u_{n,\beta}) \big)
        + \nu \big( b(u(t_{n,\beta}), u_{n,\beta} - u(t_{n,\beta}), \psi^u_{n,\beta}) \big)
        \\
        \leq&
        C\nu (\| \nabla u_{n,\beta}\| + \| \nabla u(t_{n,\beta})\| )  \| \nabla (u_{n,\beta} - u(t_{n,\beta}))\| \| \nabla \psi^u_{n,\beta}\| 		
        \\ 
        \leq&
        \frac{\mu}{128}  \| \nabla \psi^u_{n,\beta} \|^2
        + C \nu^2 \mu^{-1}
        \|u\|_{\infty,1}^2
        \|\nabla (u_{n,\beta} - u(t_{n,\beta}))\|^2
        \\
        \leq&
        \frac{\mu}{128}  \| \nabla \psi^u_{n,\beta} \|^2
        + C(\theta) \nu^2 \mu^{-1}
        \|u\|_{\infty,1}^2 (k_{n-1}+k_n)^3 
        \int_{t_{n-1}}^{t_{n+1}} \| \nabla u_{tt}\|^2 \mathrm{d}t,
        \end{split}
        \end{align}
        which completes the proof of \eqref{analytical b - exact b with t n beta equation 1}. 
        The estimate \eqref{analytical b - exact b with t n beta equation 2} can be established analogously by applying the bound \eqref{b 1 1 1} together with Lemma \ref{n beta int tn-1 tn+1} and Young's inequality, and is thus omitted here for brevity.
    \end{proof}

    \begin{lem} \label{truncation error estimate}
        Under Assumption \ref{u_p_T_space},
        the truncation term $|\tau_n(u,w,p,T,\psi^u_{n,\beta},\psi^T_{n,\beta}) |$ in \eqref{truncation} has following error estimate
        \begin{align} \label{truncation error estimate equation}  
        \begin{split}
        &|\tau(u,w,p,\psi^u_{n,\beta})|   \\
        &\leq 
        C(\theta) \big( \frac{\nu^2}{\mu} \| u \|_{\infty,2}^2 
        + \frac{\lambda^2}{\mu} \| u \|_{L^{\infty}(L^{4})}^{4} 
        + \frac{1}{\kappa}\| T \|_{\infty,2}^2  \big)
        k_{\rm{max}}^3 \int_{t_{n-1}}^{t_{n+1}}\|\nabla u_{tt}\|^2  \mathrm{d}t  \\
        &
        + \frac{C(\theta)}{\kappa}
        \|u\|_{\infty,2}^2
        k_{\rm{max}}^3 
        \int_{t_{n-1}}^{t_{n+1}}
        \| \nabla T_{tt}\|^2 \mathrm{d}t+
        \frac{\mu}{4} \|\nabla \psi^u_{n,\beta}\|^2 +\frac{\kappa}{32} \|\nabla \psi^T_{n,\beta}\|^2 \\
        &+ \frac{C(\theta) k_{\rm{max}}^3}{\mu} \int_{t_{n-1}}^{t_{n+1}} \!\big(
        \|u_{ttt}\|^2 \!+\! \mu^2 \|\nabla u_{tt}\|^2 \!+\! \rho^2 \|u_{tt}\|^2 \!+\! \gamma^{2} \|\nabla w_{tt}\|^2 
        \\
        &\qquad \qquad \qquad \qquad + \| p_{tt}\|^2 \!+\sigma^2 |\xi |^2\|T_{tt}\|^2 
        \big)	 \mathrm{d}t   \\
        &+ \frac{C(\theta) k_{\rm{max}}^3}{\kappa} \int_{t_{n-1}}^{t_{n+1}} \!\big(
        \|T_{ttt}\|^2 \!+\! \kappa^2 \|\nabla T_{tt}\|^2 \big)	 \mathrm{d}t.  
        \end{split}
        \end{align}
        for $1 \leq n \leq M-1$.
    \end{lem}
    \begin{proof}
        By similar argument to \eqref{cubic term inequality 1}, 
        \begin{align} \label{cubic term truncation inequality 2}
        \begin{split}
            &\lambda|(|u_{n,\beta}|^2 u_{n,\beta} - |u(t_{n,\beta})|^2u(t_{n,\beta}), \psi^u_{n,\beta})|
            \\
            \leq
            & C \lambda \big( \| u_{n,\beta} \|_{L^{4}}^2 + \| u(t_{n,\beta}) \|_{L^{4}}^2  \big) \big\| \nabla \big( u_{n,\beta} - u(t_{n,\beta}) \big) \big\| \| \nabla \psi^u_{n,\beta} \| \\
            \leq& C(\theta) \frac{\lambda^2}{\mu} \big( \| u_{n,\beta} \|_{L^{4}}^4 + \| u(t_{n,\beta}) \|_{L^{4}}^4  \big) 
            (k_{n-1}+k_{n})^3 \int_{t_{n-1}}^{t_{n+1}}\|\nabla u_{tt}\|^2  \mathrm{d}t \\
            &+ \frac{\mu}{128} \| \nabla \psi^u_{n,\beta}\|^2. 
        \end{split}
        \end{align}
        We utilize Cauchy-Schwarz inequality, Poincar\'e inequality, Young's inequality, along with Lemma \ref{n beta int tn-1 tn+1}
        \begin{align}
            \Big( \frac{1}{\widehat{k}_n} u_{n,\alpha} - u_t(t_{n,\beta}), \psi_{n,\beta}^u \Big) 
            \leq& C \mu^{-1} \Big\| \frac{1}{\widehat{k}_n} u_{n,\alpha} - u_t(t_{n,\beta}) \Big\|^2 + \frac{\mu}{128} \| \nabla \psi_{n,\beta}^u\|^2
            \label{cubic term truncation inequality 3} \\
            \leq&
            \frac{C(\theta)}{\mu} (k_{n-1}+k_{n})^3 \int_{t_{n-1}}^{t_{n+1}}\|u_{ttt}\|^2  \mathrm{d}t
            + \frac{\mu}{128} \| \nabla \psi_{n,\beta}^u\|^2,
            \notag \\
            (\frac{1}{\widehat{k}_n} T_{n,\alpha} - T_t(t_{n,\beta}), \psi_{n,\beta}^T) 
            \leq&
            \frac{C(\theta)}{\kappa} (k_{n-1}+k_{n})^3 \int_{t_{n-1}}^{t_{n+1}}\|T_{ttt}\|^2  \mathrm{d}t
            +\frac{\kappa}{64} \| \nabla \psi_{n,\beta}^T\|^2, \notag
            \\
            \mu (\nabla (u_{n,\beta} -u(t_{n,\beta})), \nabla \psi_{n,\beta}^u) 
            \leq& 
            C(\theta) \mu (k_{n-1}+k_{n})^3 \int_{t_{n-1}}^{t_{n+1}}\|\nabla u_{tt}\|^2  \mathrm{d}t
            +\frac{\mu}{128}\| \nabla \psi_{n,\beta}^u\|^2,
            \notag 
            \\
            \kappa (\nabla (T_{n,\beta} -T(t_{n,\beta})), \nabla \psi_{n,\beta}^T)
            \leq &
            C(\theta) \kappa (k_{n-1}+k_{n})^3 \int_{t_{n-1}}^{t_{n+1}}\| \nabla T_{tt}\|^2  \mathrm{d}t
            +\frac{\kappa}{64}\| \nabla \psi_{n,\beta}^T\|^2, \notag
            \\
            \rho (u_{n,\beta}-u(t_{n,\beta}), \psi_{n,\beta}^u)
            \leq &
            C(\theta) \frac{\rho^2}{\mu} (k_{n-1}+k_{n})^3 \int_{t_{n-1}}^{t_{n+1}}\|u_{tt}\|^2  \mathrm{d}t
            + \frac{\mu}{128} \| \nabla \psi_{n,\beta}^u\|^2,
            \notag \\
            \gamma(\nabla (w_{n,\beta} - w(t_{n,\beta})), \nabla\psi_{n,\beta}^u)
            \leq & 
            C(\theta) \frac{\gamma^2}{\mu} (k_{n-1}+k_{n})^3 \int_{t_{n-1}}^{t_{n+1}}\|\nabla w_{tt}\|^2  \mathrm{d}t
            +\frac{\mu}{128}  \| \nabla \psi_{n,\beta}^u\|^2,
            \notag \\
            (p_{n,\beta} - p(t_{n,\beta}), \nabla \cdot \psi_{n,\beta}^u)
            \leq & 
            \frac{C(\theta)}{\mu} (k_{n-1}+k_{n})^3 \int_{t_{n-1}}^{t_{n+1}}\| p_{tt}\|^2  \mathrm{d}t
            +\frac{\mu}{128} \| \nabla \psi_{n,\beta}^u\|^2,
            \notag \\
            \sigma \xi(T_{n,\beta} - T(t_{n,\beta}),\psi_{n,\beta}^u)
            \leq &
            \frac{C(\theta)\sigma^2 |\xi |^2 }{\mu}(k_{n-1}+k_{n})^3 \int_{t_{n-1}}^{t_{n+1}}\|T_{tt}\|^2  \mathrm{d}t
            +\frac{\mu}{128} \|\nabla \psi_{n,\beta}^u\|^2. \notag 
        \end{align}
        We combine \eqref{cubic term truncation inequality 2}, \eqref{cubic term truncation inequality 3} and Lemma \ref{analytical b - exact b with t n beta} to have \eqref{truncation error estimate equation}.
    \end{proof}

    Now we prove Theorem \ref{Intermediate conclusion}.
    \begin{proof}
        We sum \eqref{formulation 1} over $n$ from $1$ to $m-1$, combine Lemma \ref{derivation term} - \ref{truncation error estimate} and use \eqref{Stability of DLN inequality} in \Cref{Stability of DLN method}
        \begin{align*}
            &\frac{1+\theta}{4} \big( \|\psi^u_m\|^2 + \|\psi^T_m\|^2 \big)
            +\frac{ \widehat{k}_n}{8} \sum_{n=1}^{m-1} \big( \mu \| \nabla \psi^u_{n,\beta}\|^2 + \kappa \| \nabla \psi^T_{n,\beta}\|^2 \big)
            +\gamma \sum_{n=1}^{m-1}\widehat{k}_n\| \psi^w_{n,\beta}\|^2 
            \\
            &\leq  C_{\beta}^{(m-1)} \widehat{k}_{m-1}   \Big\{ \Big(|\rho| +
            \frac{ 2{C^{\ast}}^{4} \nu^{4}}{\mu^{3}}  \|\nabla u_{m-1,\beta}\|^4
            +\frac{2 {C^{\ast}}^{4} }{9 \kappa^2 \mu} \|\nabla T_{m-1,\beta} \|^4
            \Big) \times
            \\
            & \times \big( \|\psi^u_{m}\|^2 + \|\psi^u_{m-1}\|^2 + \|\psi^u_{m-2}\|^2 \big)  
            + \frac{C_P^2 \sigma^2 |\xi|^2 }{\mu}
            \big( \|\psi^T_{m}\|^2 + \|\psi^T_{m-1}\|^2 + \|\psi^T_{m-2}\|^2 \big)  \Big\}
            \\
            &+ \sum_{n=1}^{m-2} \Big( |\rho| + \frac{ 2{C^{\ast}}^{4} \nu^{4}}{\mu^{3}}  \|\nabla u_{n,\beta}\|^4
            +\frac{2 {C^{\ast}}^{4} }{9 \kappa^2 \mu} \|\nabla T_{n,\beta} \|^4\Big)  \widehat{k}_n \|\psi^u_{n,\beta}\|^2 
            + \frac{C_P^2 \sigma^2 |\xi|^2 }{\mu} \sum_{n=1}^{m-2}   \widehat{k}_n \|\psi^T_{n,\beta}\|^2  \\
            &+ \frac{ C(\theta) t^{\ast} \rho^2 h^{2r+4}}{\mu} \Big[ \big( \| u \|_{\infty,r+2} +  \| w \|_{\infty,r+2} + \| p \|_{\infty,r+1} \big)^{2} + \| T \|_{\infty,r+2}^2 \Big] \\
            &+\frac{C(\theta) h^{2r+4}}{\mu} \big( \| u_t \|_{2,r+2}^2 + \| w_t \|_{2,r+2}^2 + \| p_t \|_{2,r+1}^2  \big)
            + \frac{C(\theta) h^{2r+4}}{\kappa} \| T_t \|_{2,r+2}^2 \\
            &+ \frac{ C(\theta) t^{\ast} \sigma^2 |\xi |^2 h^{2r+4}}{\mu}  \| T \|_{\infty,r+2} ^{2} 
            + \frac{C(\theta, \mathcal{E}_1, \mathcal{E}_0, t^{\ast}) h^{2r+2}}{\kappa} \| T \|_{\infty,r+2}^{2}  \\
            &+ \frac{C(\theta) t^{\ast} \nu^{2} h^{2r+4}}{\mu} \| u \|_{\infty,2}^{2} 
            \big( \| u \|_{\infty,r+2} + \| w \|_{\infty,r+2} + \| p \|_{\infty,r+1} \big)^{2}  \\
            &+ \!\frac{C  (\theta, \mathcal{E}_1, \mathcal{E}_0, t^{\ast}) \nu^{2} h^{2r+2}}{\mu} 
            \big( \| u \|_{\infty,r+2} + \| w \|_{\infty,r+2} + \| p \|_{\infty,r+1} \big)^{2} \\
            &+ \frac{C(\theta) t^{\ast} h^{2r+4}}{\kappa} \| T \|_{\infty,2}^{2} 
            \big( \| u \|_{\infty,r+2} + \| w \|_{\infty,r+2} + \| p \|_{\infty,r+1} \big)^{2} \\
            &+ \frac{C(\theta) t^{\ast}  \lambda^2 h^{2r+2}}{\mu} \Big[ \| u \|_{L^{\infty}(L^{4})}^{4} 
            + h^{4r+4} \big( \| u \|_{\infty,r+2} + \| w \|_{\infty,r+2} + \| p \|_{\infty,r+1} \big)^{4} \Big] \times \\
            &\times \big( \| u \|_{\infty,r+2} + \| w \|_{\infty,r+2} + \| p \|_{\infty,r+1} \big)^{2} \\
            &+ C(\theta) k_{\rm{max}}^4 \big( \frac{\nu^2}{\mu} \| u \|_{\infty,2}^2 + \frac{\lambda^2}{\mu} \| u \|_{L^{\infty}(L^{4})}^{4} + \frac{1}{\kappa}\| T \|_{\infty,2}^2  \big)  \|\nabla u_{tt}\|_{2,0}^2  \\
            &+ \frac{C(\theta) k_{\rm{max}}^4}{\kappa} \|u\|_{\infty,2}^2 \| \nabla T_{tt}\|_{2,0}^2
            + \frac{C(\theta) k_{\rm{max}}^4}{\kappa} \big( \|T_{ttt}\|_{2,0}^2 \!+\! \kappa^2 \|\nabla T_{tt}\|_{2,0}^2 \big) \\
            &+ \frac{C(\theta) k_{\rm{max}}^4}{\mu} \big( \|u_{ttt}\|_{2,0}^2 +  \mu^2 \|\nabla u_{tt}\|_{2,0}^2
            + \rho^2 \|u_{tt}\|_{2,0}^2 \!+\! \gamma^{2} \|\nabla w_{tt}\|_{2,0}^2 \\
            &\qquad \qquad \qquad + \| p_{tt}\|_{2,0}^2 + \sigma^2 |\xi |^2 \|T_{tt}\|_{2,0}^2 \big).
        \end{align*}
        We assume that $\|\psi^u_1\|^2$, $\|\psi^u_0\|^2$, $\|\psi^T_1\|^2$ and   $\|\psi^T_0\|^2$ are of order $\mathcal{O} (h^{2r+2}, k_{\rm{max}}^4)$ and apply discrete Gr\"onwall inequality \cite[p.369]{MR1043610} along with time step restriction \eqref{time_condition_eq2} in 
        \Cref{time_condition_2} to the above inequality 
        \begin{align} \label{error estimate psi gamma inequality}
        \begin{split}
            &\big( \|\psi^u_m\|^2 +\|\psi^T_m\|^2 \big)
            + \sum_{n=1}^{m-1} \widehat{k}_n \big( \|\nabla \psi^u_{n,\beta}\|^2 + \|\nabla \psi^T_{n,\beta}\|^2 \big)
            + \sum_{n=1}^{m-1} \widehat{k}_n \| \psi^w_{n,\beta}\|^2 \\
            &\leq \mathcal{O}(h^{2r+2} + k_{\max}^4).
        \end{split}
        \end{align}
        By triangle inequality, \eqref{Ritz projection}, \eqref{Stokes-type projection} and \eqref{error estimate psi gamma inequality}
        \begin{align}\label{bound on the velocity error}
        \begin{split}
            &\max_{0 \leq n \leq M}  \| e_n^u \|^2  
            \leq \max_{0 \leq n \leq M} \|\eta^u_n\|^2 + \max_{0 \leq n \leq M} \|\psi^u_n\|^2 \leq \mathcal{O}(h^{2r+2}+ k_{\max}^4), \\
            &\max_{0 \leq n \leq M}  \|e^T_m\|^2 
            \leq \max_{0 \leq n \leq M} \|\eta^u_T\|^2 + \max_{0 \leq n \leq M} \|\psi^T_n\|^2 \leq \mathcal{O}(h^{2r+2}+ k_{\max}^4), \\
        \end{split}
        \end{align}
        Similarly, 
        \begin{align}\label{nabla u-u t n beta inequality}
            &\sum_{n=1}^{M-1} \widehat{k}_n \big( \|\nabla e_{n,\beta}^u \|^2 + \|\nabla e_{n,\beta}^T \|^2 \big)
            + \sum_{n=1}^{M-1} \widehat{k}_n \| e_{n,\beta}^{w} \|^2 \\
            \leq& \sum_{n=1}^{M-1} \widehat{k}_n \big( \|\nabla \psi^u_{n,\beta}\|^2 
            + \|\nabla \eta^u_{n,\beta}\|^2 \big)
            + \sum_{n=1}^{M-1} \widehat{k}_n \big( \|\nabla \psi^T_{n,\beta}\|^2 
            + \|\nabla \eta^T_{n,\beta}\|^2 \big) \notag \\
            &+ \sum_{n=1}^{M-1} \widehat{k}_n \big( \|\nabla \psi^w_{n,\beta}\|^2 
            + \|\nabla \eta^w_{n,\beta}\|^2 \big) \notag \\
            \leq& \mathcal{O}(h^{2r+2} + k_{\max}^4). \notag 
        \end{align}
        We combine \eqref{bound on the velocity error} and \eqref{nabla u-u t n beta inequality} to achieve \eqref{proof of max u-u^h T-T^h error}.
    \end{proof}

    \subsection{Error estimate for velocity and temperature in $H^1$-norm} \ 
    To carry out the error estimate for velocity $u$ in $H^1$-norm, we need the following extra restrictions about the uniform mesh diameter and time step size
    \begin{align}
    \label{time-diameter-condition}
    k_{\max}^{3} \leq h, \quad h^{2r+2} \leq k_{\min} = \min_{0 \leq n \leq M-1} \{ k_{n} \},
    \quad \frac{k_{\max}}{k_{\min}} \leq C
    \end{align}
    for some $C >0$. 

    \begin{theorem}\label{Intermediate conclusion 2}
        Suppose Assumptions \ref{u_p_T_space}, \ref{time_condition_2} and 
        conditions \eqref{time-diameter-condition} hold. 
        If the time step size and uniform mesh diameter are sufficiently small such that for all $1 \leq n \leq M-1$
        \begin{align}
        \label{time-condition-H1}
        C(\theta) \Big[ \frac{\nu^2 }{h} \| \nabla e_{n,\beta}^u \|^{2} +\frac{1}{h} \| \nabla e_{n,\beta}^T \|^{2} 
        + \lambda^2 \big( \|u_{n,\beta}\|_{1}^4 + \| \nabla e_{n,\beta}^u \|_{1}^4 \big) \Big]
        \widehat{k}_{n} \leq \frac{\mu(1+\theta)}{4}, 
        \end{align}
        for some $C(\theta)>0$, then Scheme \ref{fully discrete formulations scheme} has the following error estimates 
        \begin{equation} \label{proof of max nabla u-u^h T-T^h error}
        \max_{0 \leq n \leq M}  \|e_{n}^{u} \|_1^{2}+\max_{0 \leq n \leq M}  \|e_{n}^{T} \|_1^{2}
        + \max_{0 \leq n \leq M}  \|e_{n}^{w} \|^{2}
        + \sum_{n=1}^{M-1} \widehat{k}_n \|\widehat{k}_n^{-1} e^u_{n,\alpha}\|^2
        \leq
        \mathcal{O}(h^{2r+2}  + k_{\max}^4).
        \end{equation}
    \end{theorem}
    \begin{remark}
        With the help of error estimates in \eqref{proof of max u-u^h T-T^h error} and conditions in \eqref{time-diameter-condition}, the restriction in \eqref{time-condition-H1} is available since 
        \begin{align*}
        &\frac{\widehat{k}_n}{h} (\| \nabla e_{n,\beta}^u \|^{2}+ \| \nabla e_{n,\beta}^T \|^{2})
        = \mathcal{O} (h^{2r+1} + k_{\max}), \\
        &\widehat{k}_{n} \| \nabla e_{n,\beta}^u \|^{4} 
        \leq \frac{\big(\widehat{k}_{n} \| \nabla e_{n,\beta}^u \|^{2} \big)^{2}}{k_{\min}}
        = \mathcal{O} (k_{\max}).
        \end{align*}
    \end{remark}
    \begin{proof}
        By similar argument to \eqref{error equation 1} - \eqref{error equation 4}, we derive 
        \begin{align}
        \begin{split}\label{error equation 1 n alpha}
        &\frac{1}{\widehat{k}_n}(\psi^u_{n,\alpha}, v^h) + \mu (\nabla \psi^u_{n,\beta}, \nabla v^h) 
        +\gamma (\nabla \psi^w_{n,\beta},\nabla v^h) + \rho (e^u_{n,\beta},v^h) 
        \\
        =& \frac{1}{\widehat{k}_n}(\eta^u_{n,\alpha}, v^h) 
        +\nu b(u_{n,\beta}, u_{n,\beta}, v^h)
        -\nu b(u_{n,\beta}^h, u_{n,\beta}^h, v^h) 
        +  (\psi^p_{n,\beta}, \nabla \cdot v^h)
        \\
        &+ \lambda (|u_{n,\beta}|^2 u_{n,\beta}  - |u_{n,\beta}^h|^2 u_{n,\beta}^h ,v^h)
        +\! \sigma \xi (e^T_{n,\beta},v^h) - \mathcal{T}_{n,1}(u,w,p,T, v^h)
        \end{split}
        \\
        &\gamma (\psi^w_{n,\alpha}, \varphi^h) 
        - \gamma \big( \phi_{n,\alpha}^h - \mathcal{S}_h (\phi_{n,\alpha}),\nabla \cdot \varphi^h \big)
        - \gamma (\nabla \psi^u_{n,\alpha}, \nabla \varphi^h)
        =0, \label{error equation 2 n alpha}
        \\
        &\frac{1}{\widehat{k}_n} ( \nabla \cdot \psi^u_{n,\alpha} , q^h)=0.
        \label{error equation 3 n alpha}
        \\
        &\frac{\gamma}{\widehat{k}_n}  (\nabla \cdot \psi^w_{n,\beta} , \zeta^h)=0,
        \label{error equation 4 n alpha}
        \\
        \begin{split} \label{error equation 5 n alpha}
            &\frac{1}{\widehat{k}_n} (\psi^T_{n,\alpha}, \vartheta^h) 
            + \kappa (\nabla \psi^T_{n,\beta}, \nabla \vartheta^h) \\
            =& \Big(\frac{\eta^T_{n,\alpha}}{\widehat{k}_n}, \vartheta^h \Big)
            + b^*(u_{n,\beta}, T_{n,\beta}, \vartheta^h)
            - b^*(u_{n,\beta}^h, T_{n,\beta}^h, \vartheta^h)
            - \mathcal{T}_{n,2}(T,\vartheta^h).
        \end{split}
        \end{align}
        We set \( v^h = \widehat{k}_n^{-1} \psi^u_{n,\alpha} \) in \eqref{error equation 1 n alpha}, \( \varphi^h = \widehat{k}_n^{-1} \psi^w_{n,\beta} \) in \eqref{error equation 2 n alpha}, \( q^h = \psi^p_{n,\beta} \) in \eqref{error equation 3 n alpha}, \( \zeta^h = \phi_{n,\alpha}^h - \mathcal{S}_h (\phi_{n,\alpha}) \) in \eqref{error equation 4 n alpha} and \( \vartheta^h = \widehat{k}_n^{-1} \psi^T_{n,\alpha} \) in \eqref{error equation 5 n alpha}. We then combine the resulting equations and sum the corresponding inequalities from \( n = 1 \) to \( m-1 \), $1 \leq m \leq M$,. By applying the definition of the \( G \)-norm in \eqref{G-norm}, we deduce the following estimate:
        \begin{align}\label{error estimate nabla L2 inequality}
            &\frac{\mu}{4} (1+\theta) \|\nabla \psi^u_m\|^2 
            + \frac{\mu}{4} (1-\theta)  \|\nabla \psi^u_{m-1}\|^2 
            + \sum_{n=1}^{m-1} \widehat{k}_n \|\widehat{k}_n^{-1} \psi^u_{n,\alpha}\|^2
            \\
            &+\frac{\kappa}{4} (1+\theta) \|\nabla \psi^T_m\|^2 
            + \frac{\kappa}{4} (1-\theta)  \|\nabla \psi^T_{m-1}\|^2 
            + \sum_{n=1}^{m-1} \widehat{k}_n \|\widehat{k}_n^{-1} \psi^T_{n,\alpha}\|^2
            \notag \\
            &+\frac{\gamma}{4} (1+\theta) \| \psi^w_M\|^2 + \frac{\gamma}{4} (1-\theta)  \| \psi^w_{m-1}\|^2 
            \notag \\
            &\leq			
            \sum_{n=1}^{m-1}\widehat{k}_n (\widehat{k}_n^{-1} \eta^u_{n,\alpha},
            \widehat{k}_n^{-1} \psi^u_{n,\alpha})
            + \sum_{n=1}^{m-1}\widehat{k}_n (\widehat{k}_n^{-1} \eta^T_{n,\alpha},
            \widehat{k}_n^{-1} \psi^T_{n,\alpha})  \notag \\
            &- \sum_{n=1}^{m-1}	\widehat{k}_n \rho (e^u_{n,\beta}, \widehat{k}_n^{-1} \psi^u_{n,\alpha})
            + \sum_{n=1}^{m-1}	\widehat{k}_n \sigma \xi (e^T_{n,\beta}, \widehat{k}_n^{-1} \psi^u_{n,\alpha})
            \notag \\
            &+\sum_{n=1}^{m-1} \widehat{k}_n  \nu
            \big(b(u_{n,\beta}, u_{n,\beta},\widehat{k}_n^{-1} \psi^u_{n,\alpha})
            -b(u_{n,\beta}^h, u_{n,\beta}^h, \widehat{k}_n^{-1} \psi^u_{n,\alpha}) 
            \big) 
            \notag \\
            &+\sum_{n=1}^{m-1} \widehat{k}_n  
            \big(b^*(u_{n,\beta}, T_{n,\beta},\widehat{k}_n^{-1} \psi^T_{n,\alpha})
            -b^*(u_{n,\beta}^h, T_{n,\beta}^h, \widehat{k}_n^{-1} \psi^T_{n,\alpha}) \big) 
            \notag \\
            &+ \sum_{n=1}^{m-1}	\widehat{k}_n \lambda (|u_{n,\beta}|^2 u_{n,\beta} - |u^h_{n,\beta}|^2u^h_{n,\beta}, \widehat{k}_n^{-1} \psi^u_{n,\alpha})
            \notag \\
            &- \sum_{n=1}^{m-1}\widehat{k}_n \tau_{n}(u,w,p, \widehat{k}_n^{-1} \psi^u_{n,\alpha},\widehat{k}_n^{-1} \psi^T_{n,\alpha}) \!+\! \frac{\mu}{4}(1+\theta)\|\nabla \psi^u_1\|^2 + \frac{\mu}{4}(1-\theta)\|\nabla \psi^u_{0}\|^2
            \notag
            \\
            &+ \frac{\kappa}{4}(1+\theta)\|\nabla \psi^T_1\|^2 + \frac{\kappa}{4}(1-\theta)\|\nabla \psi^T_{0}\|^2 
            + \frac{\gamma}{4}(1+\theta)\|\psi^w_1\|^2 + \frac{\gamma}{4}(1-\theta)\|\psi^w_{0}\|^2. \notag 
        \end{align} 
        By similar argument to \eqref{derivation-term-eq1}-\eqref{derivation-term-eq3}
        \begin{align} \label{error-H1-term1}
        \begin{split}
            &(\widehat{k}_n^{-1} \eta^u_{n,\alpha}, \widehat{k}_n^{-1} \psi^u_{n,\alpha})   \\
            \leq&
            C(\theta) \frac{h^{2r+4}}{\widehat{k}_n} \int_{t_{n-1}}^{t_{n+1}} \big( \| u_t \|_{r+2}^2 + \| w_t \|_{r+2}^2 + \| p_t \|_{r+1}^2\big) \, \mathrm{d}t
            + \frac{1}{16} \| \widehat{k}_n^{-1} \psi^u_{n,\alpha}\|^2, \\
            &(\widehat{k}_n^{-1} \eta^T_{n,\alpha}, \widehat{k}_n^{-1} \psi^T_{n,\alpha})   
            \leq C(\theta) \frac{h^{2r+4}}{\widehat{k}_n} \int_{t_{n-1}}^{t_{n+1}} \| T_t \|_{r+2}^2  \mathrm{d}t
            + \frac{1}{16} \| \widehat{k}_n^{-1} \psi^T_{n,\alpha}\|^2. 
        \end{split}
        \end{align}
        By Cauchy-Schwarz inequality and Poincar\'e inequality 
        \begin{align} \label{u_error-H1-term2}
        \begin{split}
            &(e^u_{n,\beta}, \widehat{k}_n^{-1} \psi^u_{n,\alpha})
            \leq  C \| \nabla e_{n,\beta}^u \|^{2} + 	\frac{1}{32} \| \widehat{k}_n^{-1} \psi^u_{n,\alpha}\|^2,
            \\
            &(e^T_{n,\beta}, \widehat{k}_n^{-1} \psi^T_{n,\alpha})
            \leq
            C \| \nabla e_{n,\beta}^T \|^{2}  + 	\frac{1}{16} \| \widehat{k}_n^{-1} \psi^T_{n,\alpha}\|^2,  \\
            &\sigma \xi (e^T_{n,\beta}, \widehat{k}_n^{-1} \psi^u_{n,\alpha})
            \leq C \sigma^2 |\xi|^2 \| \nabla e_{n,\beta}^T \|^{2}  + \frac{1}{32} \| \widehat{k}_n^{-1} \psi^T_{n,\alpha}\|^2. 
        \end{split}
        \end{align}
        By algebraic calculation, we have
        \begin{align}
            &\nu b(u_{n,\beta}, u_{n,\beta},\widehat{k}_n^{-1} \psi^u_{n,\alpha})
            -\nu b(u_{n,\beta}^h, u_{n,\beta}^h, \widehat{k}_n^{-1} \psi^u_{n,\alpha}) 
            \label{b-terms-H1} \\
            =& - \nu b(e_{n,\beta}^u, u_{n,\beta}, \widehat{k}_n^{-1} \psi^u_{n,\alpha})
            - \nu b(u_{n,\beta}^h, e_{n,\beta}^u, \widehat{k}_n^{-1} \psi^u_{n,\alpha})
            \notag \\
            =& - \nu b(e_{n,\beta}^u, u_{n,\beta}, \widehat{k}_n^{-1} \psi^u_{n,\alpha})
            - \nu b(\psi_{n,\beta}^u, e_{n,\beta}^u, \widehat{k}_n^{-1} \psi^u_{n,\alpha})
            + \nu b(\eta_{n,\beta}^u, e_{n,\beta}^u, \widehat{k}_n^{-1} \psi^u_{n,\alpha}) \notag \\
            &- \nu b(u_{n,\beta}, e_{n,\beta}^u, \widehat{k}_n^{-1} \psi^u_{n,\alpha}). \notag 
        \end{align}  
        By \eqref{b 2 1 0}, \eqref{b 1 2 0}, Poincar\'e inequality and Young's inequality 
        \begin{align}
            &- \nu b(e_{n,\beta}^u, u_{n,\beta}, \widehat{k}_n^{-1} \psi^u_{n,\alpha})
            - \nu b(u_{n,\beta}, e_{n,\beta}^u, \widehat{k}_n^{-1} \psi^u_{n,\alpha}) 
            \label{b-terms-H1-eq1} \\
            &\leq C \nu^{2} \| \nabla e_{n,\beta}^u \|^{2} \| u_{n,\beta} \|_{2}^{2}
            + \frac{1}{32} \| \widehat{k}_n^{-1} \psi^u_{n,\alpha}\|^2. \notag 
        \end{align}
        By \eqref{b 1 1 0 1}, inverse inequality in \eqref{inverse-ineq}, Poincar\'e inequality and Young's inequality 
        \begin{align}
            &- \nu b(\psi_{n,\beta}^u, e_{n,\beta}^u, \widehat{k}_n^{-1} \psi^u_{n,\alpha})
            + \nu b(\eta_{n,\beta}^u, e_{n,\beta}^u, \widehat{k}_n^{-1} \psi^u_{n,\alpha})  
            \label{b-terms-H1-eq2} \\
            \leq& C \nu h^{-1/2}\big( \| \nabla \psi_{n,\beta}^u \| + \| \eta_{n,\beta}^u \|_{1} )
            \| \nabla e_{n,\beta}^u \| 
            \| \widehat{k}_n^{-1} \psi^u_{n,\alpha} \| \notag \\
            \leq& C \nu^{2} h^{-1}\big( \| \nabla \psi_{n,\beta}^u \|^{2} + \| \eta_{n,\beta}^u \|_{1}^{2} ) \| \nabla e_{n,\beta}^u \|^{2} 
            + \frac{1}{32} \| \widehat{k}_n^{-1} \psi^u_{n,\alpha}\|^2. \notag 
        \end{align}
        We combine \eqref{b-terms-H1} - \eqref{b-terms-H1-eq2} and use approximation 
        in \eqref{Stokes-type projection} to obtain
        \begin{align}
            &\nu b(u_{n,\beta}, u_{n,\beta},\widehat{k}_n^{-1} \psi^u_{n,\alpha})
            -\nu b(u_{n,\beta}^h, u_{n,\beta}^h, \widehat{k}_n^{-1} \psi^u_{n,\alpha}) 
            \label{b-terms-H1-conclusion} \\
            \leq& C \nu^{2} \| \nabla e_{n,\beta}^u \|^{2} \| u_{n,\beta} \|_{2}^{2}
            + \frac{C \nu^2}{h} \big( \| \nabla \psi_{n,\beta}^u \|^{2} + \| \eta_{n,\beta}^u \|_{1}^{2} ) \| \nabla e_{n,\beta}^u \|^{2} + \frac{1}{16} \| \widehat{k}_n^{-1} \psi^u_{n,\alpha}\|^2 \notag  \\
            \leq& C \nu^2 h^{-1} \| \nabla e_{n,\beta}^u \|^{2} \| \nabla \psi_{n,\beta}^u \|^{2} 
            + \frac{1}{16} \| \widehat{k}_n^{-1} \psi^u_{n,\alpha}\|^2 \notag  \\
            &+ C \nu^2 \Big[ \| u \|_{\infty,2}^{2} + h^{2r+1} \big( \| u \|_{\infty,r+2}^{2} 
            + \| w \|_{\infty,r+2}^{2} + \| p \|_{\infty,r+1}^{2} \big) \Big] \| \nabla e_{n,\beta}^u \|^{2}. \notag 
        \end{align}
        An analogous estimate for the term involving \( b^*(\cdot,\cdot,\cdot) \) follows in the same manner:
        \begin{align*}
        \begin{split}
            &b^*(u_{n,\beta}, T_{n,\beta},\widehat{k}_n^{-1} \psi^T_{n,\alpha})
            -b^*(u_{n,\beta}^h, T_{n,\beta}^h, \widehat{k}_n^{-1} \psi^T_{n,\alpha}) 
            \\
            =& - b^*(e_{n,\beta}^u, T_{n,\beta},\widehat{k}_n^{-1} \psi^T_{n,\alpha}) 
            - b^*(\psi_{n,\beta}^u - \eta_{n,\beta}^{u}, e_{n,\beta}^T,\widehat{k}_n^{-1} \psi^T_{n,\alpha}) 
            -  b^*( u_{n,\beta}, e_{n,\beta}^T,\widehat{k}_n^{-1} \psi^T_{n,\alpha}) 
            \\
            \leq& C  \| \nabla e_{n,\beta}^u \|^{2} \| T_{n,\beta} \|_{2}^{2}
            + C \| \nabla e_{n,\beta}^T \|^{2} \| u_{n,\beta} \|_{2}^{2}
            + \frac{C }{h} \big( \| \nabla \psi_{n,\beta}^u \|^{2} + \| \eta_{n,\beta}^u \|_{1}^{2} ) \| \nabla e_{n,\beta}^T \|^{2} \\
            &+ \frac{1}{16} \| \widehat{k}_n^{-1} \psi^T_{n,\alpha}\|^2 \notag  \\
            \leq& \frac{C}{h} \| \nabla e_{n,\beta}^T \|^{2} \| \nabla \psi_{n,\beta}^u \|^{2} 
            + \frac{1}{16} \| \widehat{k}_n^{-1} \psi^T_{n,\alpha}\|^2 \notag  + C(\theta) \big( \| u \|_{\infty,2}^{2} \| \nabla e_{n,\beta}^T \|^{2} \!+\! \| T \|_{\infty,2}^{2} \| \nabla e_{n,\beta}^u \|^{2} \big)
            \\
            &+ C \Big[ h^{2r+1} \big( \| u \|_{\infty,r+2}^{2} 
            + \| w \|_{\infty,r+2}^{2} + \| p \|_{\infty,r+1}^{2} \big) \Big] \| \nabla e_{n,\beta}^T \|^{2}
        \end{split}
        \end{align*}
        By Cauchy-Schwarz inequality, continuity property in \eqref{vector_norm_inequality_2},
        Sobolev embedding inequality, Poincar\'e inequality and approximation 
        in \eqref{Stokes-type projection}
        \begin{align} 
        &\lambda |(|u_{n,\beta}|^2 u_{n,\beta} - | u_{n,\beta}^h|^2 u_{n,\beta}^h, \widehat{k}_n^{-1} \psi^u_{n,\alpha})| \notag \\
        \leq
        & 
        \lambda \Big(
        \int_{D} (|u_{n,\beta}|^2 u_{n,\beta} -|u_{n,\beta}^h|^2 u_{n,\beta}^h)^{2}\mathrm{d}x
        \Big)^\frac{1}{2} \| \widehat{k}_n^{-1} \psi^u_{n,\alpha} \|  \notag
        \\
        \leq& 
        C \lambda \Big( \int_{D} (|u_{n,\beta}| + |u_{n,\beta}^h| )^4 |u_{n,\beta} - u_{n,\beta}^h|^2 \mathrm{d}x
        \Big)^\frac{1}{2} \| \widehat{k}_n^{-1} \psi^u_{n,\alpha} \|  \notag \\
        \leq
        &
        C \lambda \Big(
        \int_{D} (|u_{n,\beta}| + |u_{n,\beta}^h| )^8 \mathrm{d}x
        \Big)^\frac{1}{4}
        \Big(
        \int_{D} |u_{n,\beta} - u_{n,\beta}^h|^4 \mathrm{d}x
        \Big)^\frac{1}{4} 
        \| \widehat{k}_n^{-1} \psi^u_{n,\alpha} \|
        \label{cubic-term-H1}    \\
        \leq
        & C \lambda 
        \Big( \| u_{n,\beta} \|_{L^8(D)}^2 + \| e_{n,\beta}^u \|_{L^8(D)}^2 \Big) 
        \|\nabla e_{n,\beta}^u \| \| \widehat{k}_n^{-1} \psi^u_{n,\alpha} \| \notag
        \\ 
        \leq
        & C \lambda (\|u_{n,\beta}\|_{1}^2 + \| \nabla e_{n,\beta}^u \|^2 )
        \big( \|\nabla \psi_{n,\beta}^{u} \| + \|\nabla \eta_{n,\beta}^{u} \| \big)
        \| \widehat{k}_n^{-1} \psi^u_{n,\alpha} \| \notag \\
        \leq& C \lambda^2 (\|u_{n,\beta}\|_{1}^4 + \| \nabla e_{n,\beta}^u \|^4 )
        \big( \|\nabla \psi_{n,\beta}^{u} \|^{2} + \|\nabla \eta_{n,\beta}^{u} \|^{2} \big)
        + \frac{1}{16} \| \widehat{k}_n^{-1} \psi^u_{n,\alpha}\|^2 \notag \\
        \leq& C \lambda^2 (\|u_{n,\beta}\|_{1}^4 + \| \nabla e_{n,\beta}^u \|^4 ) 
        \|\nabla \psi_{n,\beta}^{u} \|^{2} + \frac{1}{16} \| \widehat{k}_n^{-1} \psi^u_{n,\alpha}\|^2 \notag \\
        &+ C \lambda^2 h^{2r+2} (\|u_{n,\beta}\|_{1}^4 + \| \nabla e_{n,\beta}^u \|^4 ) 
        \big( \| u \|_{\infty,r+2}^{2} + \| w \|_{\infty,r+2}^{2} + \| p \|_{\infty,r+1}^{2} \big).
        \notag 
        \end{align}
        Now we address terms in $\tau_{n}(u,w,p, \widehat{k}_n^{-1} \psi^u_{n,\alpha},\widehat{k}_n^{-1} \psi^T_{n,\alpha})$. 
        By similar argument to \eqref{cubic-term-H1}
        \begin{align} \label{tau-H1-eq1} 
        \begin{split}
            &\lambda \big| \big( |u_{n,\beta}|^2 u_{n,\beta} - | u(t_{n,\beta}) |^2 u(t_{n,\beta}), \widehat{k}_n^{-1} \psi^u_{n,\alpha} \big) \big|  \\
            \leq& C \lambda^2 (\|u_{n,\beta}\|_{1}^4 + \| u(t_{n,\beta}) \|_{1}^4 )
            \| u_{n,\beta} - u(t_{n,\beta}) \|_{1}^{2}
            + \frac{1}{32} \| \widehat{k}_n^{-1} \psi^u_{n,\alpha}\|^2 \\
            \leq& C(\theta) \lambda^2 \| u \|_{\infty,1}^{4} (k_{n} + k_{n-1})^{3} 
            \int_{t_{n-1}}^{t_{n+1}} \| u_{tt} \|_{1}^{2} dt 
            + \frac{1}{32} \| \widehat{k}_n^{-1} \psi^u_{n,\alpha}\|^2.
        \end{split}
        \end{align}
        By \eqref{b 2 1 0}, \eqref{b 1 2 0}, Young's inequality and Lemma \ref{n beta int tn-1 tn+1}
        \begin{align} \label{tau-H1-eq2}
        \begin{split}
            &\nu \big(b(u_{n,\beta}, u_{n,\beta}, \widehat{k}_n^{-1} \psi^u_{n,\alpha})
            - b(u(t_{n,\beta}), u(t_{n,\beta}), \widehat{k}_n^{-1} \psi^u_{n,\alpha}) \big)
            \\
            \leq&
            C\nu (\| u_{n,\beta} \|_{2} + \| u(t_{n,\beta}) \|_{2} )  \| u_{n,\beta} - u(t_{n,\beta})\|_{1} \| \widehat{k}_n^{-1} \psi^u_{n,\alpha} \| 		
            \\ 
            \leq&
            C(\theta) \nu^{2} \| u \|_{\infty,2}^{2} (k_{n} + k_{n-1})^{3}  
            \int_{t_{n-1}}^{t_{n+1}} \| u_{tt} \|_{1}^{2} dt
            + \frac{1}{64} \| \widehat{k}_n^{-1} \psi^u_{n,\alpha}\|^2. 
            \end{split}
            \\
            \begin{split}
            & b^{\ast}(u_{n,\beta}, T_{n,\beta}, \widehat{k}_n^{-1} \psi^u_{n,\alpha})
            - b^{\ast}(u(t_{n,\beta}), T(t_{n,\beta}), \widehat{k}_n^{-1} \psi^u_{n,\alpha}) 
            \\
            \leq&
            C (\| T_{n,\beta} \|_{2}\| u_{n,\beta} - u(t_{n,\beta})\|_{1} + \| u(t_{n,\beta}) \|_{2} \| T_{n,\beta} - T(t_{n,\beta})\|_{1})  \| \widehat{k}_n^{-1} \psi^u_{n,\alpha} \| 		
            \\ 
            \leq&
            C(\theta)  \| T \|_{\infty,2}^{2} (k_{n} + k_{n-1})^{3}  
            \int_{t_{n-1}}^{t_{n+1}} \| u_{tt} \|_{1}^{2} dt
            + \frac{1}{64} \| \widehat{k}_n^{-1} \psi^u_{n,\alpha}\|^2
            \\
            &+C(\theta)  \| u \|_{\infty,2}^{2} (k_{n} + k_{n-1})^{3}  
            \int_{t_{n-1}}^{t_{n+1}} \| T_{tt} \|_{1}^{2} dt. 
        \end{split}
        \end{align}
        We utilize Cauchy-Schwarz, inequality, Gauss divergence theorem, Poincar\'e inequality, Young's inequalities, along with \Cref{n beta int tn-1 tn+1}
        \begin{align}
            &\Big( \frac{u_{n,\alpha}}{\widehat{k}_n} - u_t(t_{n,\beta}), \widehat{k}_n^{-1} \psi^u_{n,\alpha} \Big) 
            \leq
            C(\theta) k_{\max}^3 \int_{t_{n-1}}^{t_{n+1}} \|u_{ttt}\|^2  \mathrm{d}t
            \!+\! \frac{1}{32} \| \widehat{k}_n^{-1} \psi^u_{n,\alpha} \|^2,
            \label{tau-H1-eq3}  \\
            &\Big( \frac{T_{n,\alpha}}{\widehat{k}_n} - T_t(t_{n,\beta}), \widehat{k}_n^{-1} \psi^T_{n,\alpha} \Big) 
            \leq
            C(\theta) k_{\max}^3 \int_{t_{n-1}}^{t_{n+1}} \|T_{ttt}\|^2  \mathrm{d}t
            + \frac{1}{32} \| \widehat{k}_n^{-1} \psi^T_{n,\alpha} \|^2,
            \notag \\
            &\mu (\nabla (u_{n,\beta} -u(t_{n,\beta})), \nabla \widehat{k}_n^{-1} \psi^u_{n,\alpha} ) 
            \leq 
            C(\theta) \mu^{2} k_{\max}^3 \int_{t_{n-1}}^{t_{n+1}}\|\Delta u_{tt}\|^2  \mathrm{d}t
            +\frac{1}{32}\| \widehat{k}_n^{-1} \psi^u_{n,\alpha} \|^2,
            \notag \\
            & \kappa (\nabla (T_{n,\beta} -T(t_{n,\beta})), \nabla \widehat{k}_n^{-1} \psi^T_{n,\alpha} ) 
            \leq 
            C(\theta) \kappa^2 k_{\max}^3 \int_{t_{n-1}}^{t_{n+1}}\|\Delta T_{tt}\|^2  \mathrm{d}t
            +\frac{1}{32}\| \widehat{k}_n^{-1} \psi^T_{n,\alpha} \|^2,
            \notag \\
            &\rho (u_{n,\beta}-u(t_{n,\beta}), \widehat{k}_n^{-1} \psi^u_{n,\alpha})
            \leq 
            C(\theta) \rho^2 k_{\max}^3 \int_{t_{n-1}}^{t_{n+1}}\|u_{tt}\|^2  \mathrm{d}t
            + \frac{1}{32} \| \widehat{k}_n^{-1} \psi^u_{n,\alpha}\|^2,
            \notag \\
            &\gamma(\nabla (w_{n,\beta} - w(t_{n,\beta})), \nabla \widehat{k}_n^{-1} \psi^u_{n,\alpha})
            \leq  
            C(\theta) \gamma^2 k_{\max}^3 \int_{t_{n-1}}^{t_{n+1}}\|\Delta w_{tt}\|^2  \mathrm{d}t
            +\frac{1}{32}  \| \widehat{k}_n^{-1} \psi^u_{n,\alpha} \|^2,
            \notag \\
            &(p_{n,\beta} - p(t_{n,\beta}), \nabla \cdot \widehat{k}_n^{-1} \psi^u_{n,\alpha})
            \leq  
            C(\theta) k_{\max}^3 \int_{t_{n-1}}^{t_{n+1}}\|\nabla p_{tt}\|^2  \mathrm{d}t
            +\frac{1}{32} \|\widehat{k}_n^{-1} \psi^u_{n,\alpha} \|^2,
            \notag \\
            &\sigma \xi(T{n,\beta} - T(t_{n,\beta}),\widehat{k}_n^{-1} \psi^u_{n,\alpha})
            \leq   
            C(\theta) \sigma^2 |\xi |^2 k_{\max}^3 \int_{t_{n-1}}^{t_{n+1}}\|T_{tt}\|^2  \mathrm{d}t
            +\frac{1}{32} \|\widehat{k}_n^{-1} \psi^u_{n,\alpha}\|^2. \notag 
        \end{align}
        We combine \eqref{tau-H1-eq1} - \eqref{tau-H1-eq3} to have 
        \begin{align} 
            &\sum_{n=1}^{m-1} \widehat{k}_n  |\tau(u,w,p,\widehat{k}_n^{-1} \psi^u_{n,\alpha})|
            \label{tau-H1-conclusion} \\
            &\leq C(\theta) \big( \lambda^2 \| u \|_{\infty,1}^{4} + \nu^{2} \| u \|_{\infty,2}^{2} +\| T \|_{\infty,2}^{2} \big) k_{\max}^4 \| u_{tt} \|_{2,1}^{2}+ C(\theta)   \| u \|_{\infty,2}^{2}  k_{\max}^4 \| T_{tt} \|_{2,1}^{2}\notag
            \\
            &+ \frac{1}{4} \|\widehat{k}_n^{-1} \psi^u_{n,\alpha}\|^2 + \frac{1}{16} \|\widehat{k}_n^{-1} \psi^T_{n,\alpha}\|^2 
            + \!C(\theta) k_{\max}^4 \big( \|u_{ttt}\|_{2,0}^2 +\|T_{ttt}\|_{2,0}^2 \notag
            + \mu^{2} \|\Delta u_{tt}\|_{2,0}^2 \notag \\
            &+ \kappa^2 \|\Delta T_{tt}\|_{2,0}^2+ \rho^2 \|u_{tt}\|_{2,0}^2 + \gamma^2 \|\Delta w_{tt}\|_{2,0}^2
            + \|\nabla p_{tt}\|_{2,0}^2 +\sigma^2 |\xi|^2 \|T_{tt}\|_{2,0}^2 \big). \notag 
        \end{align}
        By \eqref{error-H1-term1}, \eqref{u_error-H1-term2}, \eqref{b-terms-H1-conclusion}, \eqref{cubic-term-H1} and \eqref{tau-H1-conclusion}, 
        \eqref{error estimate nabla L2 inequality} becomes 
        \begin{align}
            &\frac{\mu}{4} (1+\theta) \|\nabla \psi^u_m\|^2 +\frac{\kappa}{4} (1+\theta) \|\nabla \psi^T_m\|^2 
            +\frac{\gamma}{4} (1+\theta) \| \psi^w_m\|^2 \label{error estimate H1 conclusion}  \\
            &+ \sum_{n=1}^{m-1} \frac{\widehat{k}_n}{2} \|\widehat{k}_n^{-1} \psi^u_{n,\alpha}\|^2
            + \sum_{n=1}^{m-1} \frac{\widehat{k}_n}{2} \|\widehat{k}_n^{-1} \psi^T_{n,\alpha}\|^2
            \notag \\
            &\leq  C(\theta) \Big[ \frac{\nu^2 \| \nabla e_{m-1,\beta}^u \|^{2} + \| \nabla e_{m-1,\beta}^T \|^{2} }{h} 
            + \lambda^2 \big( \|u_{m-1,\beta}\|_{1}^4 + \| \nabla e_{m-1,\beta}^u \|_{1}^4 \big) \Big]
            \widehat{k}_{m-1} \| \nabla \psi_{m}^u \|^2  \notag \\
            &+ C(\theta) \Big[ \frac{\nu^2 \| \nabla e_{m-1,\beta}^u \|^{2} + \| \nabla e_{m-1,\beta}^T \|^{2} }{h} 
            \!+\! \lambda^2 \big( \|u_{m-1,\beta}\|_{1}^4 + \| \nabla e_{m-1,\beta}^u \|_{1}^4 \big) \Big] \times \notag \\ & \times \widehat{k}_{m-1}
            \big( \| \nabla \psi_{m-1}^u \|^2 \!+\! \| \nabla \psi_{m-2}^u \|^2 \big) \notag \\
            &+C(\theta) \sum_{n=1}^{m-2} \Big[ \frac{\nu^2 \| \nabla e_{n,\beta}^u \|^{2} +\| \nabla e_{n,\beta}^T \|^{2}  }{h} + \lambda^2 \big( \|u_{n,\beta}\|_{1}^4 + \| \nabla e_{n,\beta}^u \|_{1}^4 \big) \Big]
            \widehat{k}_{n} \| \nabla \psi_{n,\beta}^u \|^2		\notag \\
            &+C(\theta) h^{2r+4} \big( \| u_t \|_{2,r+2}^2 + \| w_t \|_{2,r+2}^2 
            + \| p_t \|_{2,r+1}^2 + \| T_t \|_{2,r+2}^2 \big)  \notag \\
            &+ C(\theta) (1 + \sigma^2 |\xi|^2 + \| u \|_{\infty,2}^{2} ) \sum_{n=1}^{m-1} \widehat{k}_{n} \| \nabla e_{n,\beta}^{T} \|^{2}  
            + C(\theta) \big( 1+ \| T \|_{\infty,2}^{2} \big) \sum_{n=1}^{m-1} \widehat{k}_{n} \|\nabla e_{n,\beta}^{u} \|^{2}
            \notag \\
            &+ C \nu^2 \Big[ \| u \|_{\infty,2}^{2} + h^{2r+1} \big( \| u \|_{\infty,r+2}^{2} 
            + \| w \|_{\infty,r+2}^{2} + \| p \|_{\infty,r+1}^{2} \big) \Big] 
            \sum_{n=1}^{m-1} \widehat{k}_{n} \|\nabla e_{n,\beta}^{u} \|^{2}
            \notag \\
            &+ C h^{2r+1} \big( \| u \|_{\infty,r+2}^{2}  + \| w \|_{\infty,r+2}^{2} + \| p \|_{\infty,r+1}^{2} \big) 
            \sum_{n=1}^{m-1} \widehat{k}_{n} \|\nabla e_{n,\beta}^{T} \|^{2}
            \notag \\
            &+ C \lambda^2 h^{2r+2} \big( \| u \|_{\infty,r+2}^{2} + \| w \|_{\infty,r+2}^{2} + \| p \|_{\infty,r+1}^{2} \big) 
            \sum_{n=1}^{m-1} \widehat{k}_{n} (\|u_{n,\beta}\|_{1}^4 + \| \nabla e^u_{n,\beta} \|^4 )  \notag  \\
            &+C(\theta) \big( \lambda^2 \| u \|_{\infty,1}^{4} + \nu^{2} \| u \|_{\infty,2}^{2} +\| T \|_{\infty,2}^{2} \big) k_{\max}^4 \| u_{tt} \|_{2,1}^{2}+ C(\theta)   \| u \|_{\infty,2}^{2}  k_{\max}^4 \| T_{tt} \|_{2,1}^{2}\notag
            \\
            &+ \!C(\theta) k_{\max}^4 \big( \|u_{ttt}\|_{2,0}^2 +\|T_{ttt}\|_{2,0}^2 \notag
            + \mu^{2} \|\Delta u_{tt}\|_{2,0}^2 + \kappa^2 \|\Delta T_{tt}\|_{2,0}^2+ \rho^2 \|u_{tt}\|_{2,0}^2 
            \notag \\
            &\qquad \qquad \qquad + \gamma^2 \|\Delta w_{tt}\|_{2,0}^2
            + \|\nabla p_{tt}\|_{2,0}^2 +\sigma^2 |\xi|^2 \|T_{tt}\|_{2,0}^2 \big). \notag 
        \end{align} 
        We make use of discrete Gr\"onwall inequality \cite[p.369]{MR1043610} along with \eqref{proof of max u-u^h T-T^h error} in Theorem \ref{Intermediate conclusion}
        and restrictions in \eqref{time_condition_eq2}, \eqref{time-condition-H1} to obtain
        \begin{align} \label{psi H1 conclusion}
            &\|\nabla \psi^u_m\|^2 +\|\nabla \psi^T_m\|^2 
            +\| \psi^w_m\|^2
            + \sum_{n=1}^{m-1} \frac{\widehat{k}_n}{2} \|\widehat{k}_n^{-1} \psi^u_{n,\alpha}\|^2
            + \sum_{n=1}^{m-1} \frac{\widehat{k}_n}{2} \|\widehat{k}_n^{-1} \psi^T_{n,\alpha}\|^2 \\
            &\leq \mathcal{O} (h^{2r+2}+ k_{\max}^{4}), \notag 
        \end{align}
        which implies \eqref{proof of max nabla u-u^h T-T^h error} by triangle inequality and approximation in \eqref{Stokes-type projection}.
    \end{proof}

    \subsection{Error estimate for pressure} \ 
    \begin{lemma} \label{Intermediate conclusion 3}
        Suppose Assumptions \ref{u_p_T_space}, \ref{time_condition_2} and 
        conditions in \eqref{time-diameter-condition}, \eqref{time-condition-H1} hold. 
        then \Cref{fully discrete formulations scheme} has the following estimate
        \begin{align}
        \label{lemma-p-conclusion}
        &\sum_{n=1}^{m-1}  \mu \widehat{k}_n \|e^w_{n,\beta}\|^2 
        + \frac{\gamma}{2} \sum_{n=1}^{m-1} \widehat{k}_n \|\nabla e^w_{n,\beta}\|^2 
        \leq \mathcal{O} (h^{2r+2} + k_{\max}^{4}).
        \end{align}
    \end{lemma}
    \begin{proof}
        We set \( v^h = \psi^w_{n,\beta} \) in \eqref{error equation 1}, \( \varphi^h = \psi^w_{n,\beta} \) in \eqref{error equation 2}, \( \zeta^h = \psi^p_{n,\beta}\) in \eqref{error equation 3}, and \(\zeta^h=\phi_{n,\beta}^h - \mathcal{S}_h ( \phi_{n,\beta} ) \) in \eqref{error equation 3}
        \begin{align}
            &\frac{1}{\widehat{k}_n}(e^u_{n,\alpha}, \psi^w_{n,\beta}) 
            + \mu (\nabla \psi^u_{n,\beta}, \nabla \psi^w_{n,\beta}) 
            +\gamma (\nabla \psi^w_{n,\beta},\nabla \psi^w_{n,\beta})
            + \rho (e^u_{n,\beta}, \psi^w_{n,\beta}) 
            \label{error equation 1_p} \\
            =& \nu b(u_{n,\beta}, u_{n,\beta}, \psi^w_{n,\beta})
            \!-\!\nu b(u_{n,\beta}^h, u_{n,\beta}^h, \psi^w_{n,\beta})
            \!+\! (\psi^p_{n,\beta}, \nabla \cdot\psi^w_{n,\beta})
            \!+\! \sigma\xi (e^T_{n,\beta}, \psi^w_{n,\beta})
            \notag  \\
            &+ \! \lambda (|u_{n,\beta}|^2 u_{n,\beta}  - |u_{n,\beta}^h|^2 u_{n,\beta}^h,\psi^w_{n,\beta}) 
            - \mathcal{T}_{n,1}(u,w,p,T, \psi^w_{n,\beta}), \notag \\
            &\mu (\psi^w_{n,\beta}, \psi^w_{n,\beta}) 
            - \mu \big( \phi_{n,\beta}^h - \mathcal{S}_h (\phi_{n,\beta}),\nabla \cdot \psi^w_{n,\beta} \big)
            - \mu (\nabla \psi^u_{n,\beta}, \nabla \psi^w_{n,\beta})
            =0, \label{error equation 2_p}
            \\
            &(\nabla \cdot \psi^w_{n,\beta} , \psi^p_{n,\beta})=0,
            \label{error equation 3_p}
            \\
            &\mu \big( \nabla \cdot \psi^w_{n,\beta} , \phi_{n,\beta}^h - \mathcal{S}_h (\phi_{n,\beta}) \big)=0,
            \label{error equation 4_p}
        \end{align}
        We  add \eqref{error equation 1_p} - \eqref{error equation 4_p} together and sum the resulting equation over $n$ from $1$ to $m-1$ 
        \begin{align} \label{lemma-p-eq2}
            &\sum_{n=1}^{m-1}  \mu \widehat{k}_n \|\psi^w_{n,\beta}\|^2 
            + \sum_{n=1}^{m-1}  \gamma \widehat{k}_n \|\nabla \psi^w_{n,\beta}\|^2 
            \\
            \leq&
            \sum_{n=1}^{m-1} \widehat{k}_n | (\widehat{k}_n^{-1} e^u_{n,\alpha},\psi^w_{n,\beta})|
            +\sum_{n=1}^{m-1}\widehat{k}_n |\rho| |(e^u_{n,\beta},\psi^w_{n,\beta})|
            +\sum_{n=1}^{m-1}\widehat{k}_n \sigma |\xi| |(e^T_{n,\beta},\psi^w_{n,\beta})| \notag \\
            &
            +\sum_{n=1}^{m-1} \nu\widehat{k}_n | b(u_{n,\beta},u_{n,\beta},\psi^w_{n,\beta})
            - b(u^h_{n,\beta},u^h_{n,\beta},\psi^w_{n,\beta})| \notag \\
            &
            +\sum_{n=1}^{m-1} \widehat{k}_n \lambda 
            |(|u_{n,\beta}|^2 u_{n,\beta} - |u^h_{n,\beta}|^2u^h_{n,\beta},\psi^w_{n,\beta})|
            + \sum_{n=1}^{m-1} \widehat{k}_n |\mathcal{T}_{n,1}(u,w,p,T, \psi^w_{n,\beta})|. \notag 
        \end{align}
        By Cauchy-Schwarz inequality, Poincar\'e inequality and Young's inequality 
        \begin{align} \label{lemma-p-eq3}
            \begin{split}
            &\sum_{n=1}^{m-1} \widehat{k}_n (\widehat{k}_n^{-1} e^u_{n,\alpha},\psi^w_{n,\beta})
            + \sum_{n=1}^{m-1}\widehat{k}_n \rho (e^u_{n,\beta},\psi^w_{n,\beta}) 
            +\sum_{n=1}^{m-1}\widehat{k}_n \sigma\xi(e^T_{n,\beta},\psi^w_{n,\beta})\\
            \leq& \frac{C}{\gamma} \sum_{n=1}^{m-1} \widehat{k}_n \| \widehat{k}_n^{-1} e^u_{n,\alpha} \|^{2} 
            + \frac{C |\rho|}{\gamma} \sum_{n=1}^{m-1} \widehat{k}_n \| \nabla e^u_{n,\beta} \|^2
            + \frac{C \sigma^2 |\xi |^2}{\gamma} \sum_{n=1}^{m-1} \widehat{k}_n \| \nabla e^T_{n,\beta} \|^2
            \\
            &+ \frac{\gamma}{8} \sum_{n=1}^{m-1} \widehat{k}_n  \| \nabla \psi^w_{n,\beta}\|^2 
        \end{split}
        \end{align}
        By similar argument to \eqref{cubic term inequality 1}
        \begin{align} \label{lemma-p-eq4}
            \begin{split}
            &\sum_{n=1}^{m-1}	\widehat{k}_n \lambda (|u_{n,\beta}|^2 u_{n,\beta} - |u^h_{n,\beta}|^2u^h_{n,\beta}, \psi^w_{n,\beta}) 
            \\
            \leq& C \lambda \sum_{n=1}^{m-1} \widehat{k}_n \big( \| u_{n,\beta} \|_{L^{4}}^{2} 
            + \| e_{n,\beta}^{u} \|_{1}^{2} \big) \| \nabla e_{n,\beta}^{u} \| 
            \| \nabla \psi^w_{n,\beta} \|
            \\ 
            \leq& \frac{C \lambda^2}{\gamma} \sum_{n=1}^{m-1}	\widehat{k}_n 
            \big( \|u_{n,\beta}\|_{L^4}^4 + \| e_{n,\beta}^{u} \|_{1}^{4} \big)
            \| \nabla e_{n,\beta}^u\|^2
            +\frac{\gamma}{8} \sum_{n=1}^{m-1} \widehat{k}_n  \|\nabla \psi^w_{n,\beta}\|^2 \\
            \leq& \frac{C(\theta) \lambda^2}{\gamma} \big( \| u \|_{L^{\infty}(L^4)}^{4} 
            +\max_{0 \leq n \leq M}  \| e_n^u \|_1^{4} \big)
            \sum_{n=1}^{m-1}	\widehat{k}_n \| \nabla e_{n,\beta}^u\|^2
            + \frac{\gamma}{8} \sum_{n=1}^{m-1} \widehat{k}_n  \|\nabla \psi^w_{n,\beta}\|^2.
            \end{split}
        \end{align}
        By \eqref{b 1 1 1} and Young's inequality
        \begin{align}
        %\begin{split}
            &\sum_{n=1}^{m-1} \nu\widehat{k}_n( b(u_{n,\beta},u_{n,\beta},\psi^w_{n,\beta})
            - b(u^h_{n,\beta},u^h_{n,\beta},\psi^w_{n,\beta}))
            \label{lemma-p-eq5} \\
            =& \sum_{n=1}^{m-1}  \nu \widehat{k}_n \big( - b(e_{n,\beta}^{u},u_{n,\beta},\psi^w_{n,\beta}) 
            - b(e_{n,\beta}^u,e_{n,\beta}^{u},\psi^w_{n,\beta}) 
            - b(u_{n,\beta},e_{n,\beta}^{u},\psi^w_{n,\beta})  \big) 
            \notag \\
            \leq&
            \sum_{n=1}^{m-1}C \widehat{k}_n  \nu(\|\nabla u_{n,\beta}\| + \|\nabla e_{n,\beta}^u \|)\|\nabla e^u_{n,\beta}\|\|\nabla \psi^w_{n,\beta}\| 
            \notag \\
            \leq&
            \frac{C(\theta) \nu^2}{\gamma} \big( \| \nabla u \|_{\infty,0}^2 + \max_{0 \leq n \leq M}  \| e_n^u \|_1^2 \big) \sum_{n=1}^{m-1} \widehat{k}_n \|  \nabla e_{n,\beta}^u\|^2
            +\frac{\gamma}{8} \sum_{n=1}^{m-1} \widehat{k}_n  \|\nabla \psi^w_{n,\beta}\|^2. 
            \notag
        %	\end{split}
        \end{align}
        By similar argument to \eqref{truncation error estimate equation}
        \begin{align} 
            &\sum_{n=1}^{m-1} \widehat{k}_n  | \mathcal{T}_{n,1}(u,w,p,T, \psi^w_{n,\beta})|
            \label{lemma-p-eq6} \\
            \leq&
            \frac{\gamma}{8} \sum_{n=1}^{m-1} \widehat{k}_n  \| \nabla \psi^w_{n,\beta}\|^2
            + \frac{C(\theta)}{\gamma} \big( \nu^2 \| u \|_{\infty,1} + \lambda^2 \| u \|_{L^{\infty}(L^{4})}^{4} \big) k_{\rm{max}}^4 \| \nabla u_{tt} \|_{2,0}^{2}
            \notag \\
            +& \frac{C(\theta) k_{\rm{max}}^4}{\gamma} \big( \|u_{ttt}\|_{2,0}^2 
            + \mu^2 \|\nabla u_{tt}\|_{2,0}^2 + \rho^2 \|u_{tt}\|_{2,0}^2 
            + \gamma^{2} \|\nabla w_{tt}\|_{2,0}^2 + \| p_{tt}\|_{2,0}^2 \notag \\
            &\qquad \qquad \quad + \sigma^2 |\xi |^2\|T_{tt}\|_{2,0}^2 \big). \notag 
        \end{align}
        By \eqref{lemma-p-eq2} - \eqref{lemma-p-eq6} and error estimates in \eqref{proof of max u-u^h T-T^h error}, \eqref{proof of max nabla u-u^h T-T^h error}, we have 
        \begin{align} \label{psiw-conclusion}
        &\sum_{n=1}^{m-1}  \mu \widehat{k}_n \|\psi^w_{n,\beta}\|^2 
        + \frac{\gamma}{2} \sum_{n=1}^{m-1} \widehat{k}_n \|\nabla \psi^w_{n,\beta}\|^2 
        \leq \mathcal{O} (h^{2r+2} + k_{\max}^{4}).
        \end{align}
        By triangle inequality, the approximation in \eqref{Stokes-type projection} and \eqref{psiw-conclusion}, we derive
        \eqref{lemma-p-conclusion}. 
    \end{proof}

    \begin{theorem}\label{Intermediate conclusion 4}
        Suppose Assumptions \ref{u_p_T_space}, \ref{time_condition_2} and 
        conditions \eqref{time-diameter-condition}, \eqref{time-condition-H1} hold. 
        \Cref{fully discrete formulations scheme} has the following error estimate for pressure 
        \begin{equation} \label{nabla p-p^h error}
        \sum_{n=1}^{M-1} \widehat{k}_n
        (\|p(t_{n,\beta})- p_{n,\beta}^h\|^2)
        \leq \mathcal{O}(h^{2r+2} + k_{\max}^4).
        \end{equation}
    \end{theorem}
    \begin{proof}
        By \eqref{error equation 1}, definition of Stokes-type projection $\mathcal{S}_h p_{n,\beta}$ in \eqref{Stokes-type-projection-defination}, Poincar\'e inequality and similar argument to \eqref{lemma-p-eq4} - \eqref{lemma-p-eq6} 
        \begin{align}\label{p - Shp n beta}
            &(\psi^p_{n,\beta}, \nabla \cdot v^h) 
            \\
            =&
            (\frac{1}{\widehat{k}_n} e^u_{n,\alpha}, v^h) 
            +\mu (\nabla e^u_{n,\beta}, \nabla v^h) 
            +\gamma (\nabla e^w_{n,\beta}, \nabla v^h)
            +\rho(e^u_{n,\beta},v^h)\notag
            \\
            &+\nu b(u_{n,\beta}^h, u_{n,\beta}^h, v^h) - \nu b(u_{n,\beta}, u_{n,\beta}, v^h)
            -\lambda(|u_{n,\beta}|^2 u_{n,\beta} - |u^h_{n,\beta}|^2u^h_{n,\beta},v^h)\notag
            \\
            &+(p_{n,\beta} -\mathcal{S}_h p_{n,\beta}, \nabla \cdot v^h)
            - \sigma \xi (e^T_{n,\beta}, v^h)
            + \mathcal{T}_{n,1}(u,w,p,T, v^h) \notag
            \\
            \leq& C(\theta) \| \nabla v^h \| \Big\{ \| {\widehat{k}_n}^{-1} e^u_{n,\alpha} \| 
            + \mu \| \nabla e^u_{n,\beta} \| 
            + \gamma \| \nabla e^w_{n,\beta} \| 
            + |\rho| \| \nabla e^u_{n,\beta} \|  
            + \| \eta_{n,\beta}^p \| \notag \\
            &+ \nu \big( \| \nabla u \|_{\infty,0} + \| |e^u| \|_{\infty,1} \big) 
            \|\nabla e^u_{n,\beta}\|  
            + \lambda \big( \| u \|_{L^{\infty}(L^4)}^{2}
            + \max_{0 \leq n \leq M} \| e_\| |e^u| \|_{\infty,1}^2 \big) \| \nabla e_{n,\beta}^{u} \| 
            \notag \\
            &+ \sigma |\xi| \|\nabla e_{n,\beta}^T\|+ \big( \nu \| u \|_{\infty,2} + \lambda \| u \|_{L^{\infty}(L^{4})}^{2} \big)
            \Big( k_{\rm{max}}^3 \int_{t_{n-1}}^{t_{n+1}}\|\nabla u_{tt}\|^2  \mathrm{d}t \Big)^{\frac{1}{2}}   \notag \\
            &+ \Big[ \!k_{\rm{max}}^3 \!\!\! \int_{t_{n-1}}^{t_{n+1}} \!\!\!\big(
            \|u_{ttt}\|^2 \!+\! \mu^2 \|\nabla u_{tt}\|^2 \!+\! \rho^2 \|u_{tt}\|^2 \!+\! \gamma^{2} \|\nabla w_{tt}\|^2 \!+\! \| p_{tt}\|^2 \!+\! \sigma^2|\xi|^2\|T_{tt}\|^2
            \big)	 \mathrm{d}t \! \Big]^{\frac{1}{2}} \Big\}. \notag 
        \end{align}
        where $\| |e^u| \|_{\infty,1} = \max_{0 \leq n \leq M}\| e_{n}^{u} \|_{1}$.
        By the $LBB^h$ condition in \eqref{LBB}, 
        \begin{align}\label{p-LBB}
            \|\psi_{n,\beta}^p\| \leq C \sup_{v^h \in X_h \setminus \{0\}} 
            \frac{(\nabla \cdot v^h, \psi_{n,\beta}^p)}{\|\nabla v^h\|}.
        \end{align}
        Hence we combine \eqref{p - Shp n beta} and \eqref{p-LBB} to have 
        \begin{align}
            &\sum_{n=1}^{M-1}\widehat{k}_n \| \psi_{n,\beta}^p \|^{2}  \\
            &\leq \!C(\theta) \!\sum_{n=1}^{M-1} \!\widehat{k}_n 
            \big( \| {\widehat{k}_n}^{-1} e^u_{n,\alpha} \|^{2} 
            \!+\! \mu^{2} \| \nabla e^u_{n,\beta} \|^{2} 
            \!+\! \gamma^{2} \| \nabla e^w_{n,\beta} \|^{2} \!+\! \rho^{2} \| \nabla e^u_{n,\beta} \|^{2}
            \!+\! \sigma^2 |\xi|^2 \| \nabla e^T_{n,\beta} \|^{2}  \big)  \notag \\
            &+ C(\theta) T h^{2r+2} \big( \| u \|_{\infty,r+2} + \| w \|_{\infty,r+2} 
            + \| p \|_{\infty,r+1} \big)^{2} \notag \\
            &+ C(\theta) \Big[\nu^2 \big( \| \nabla u \|_{\infty,0}^2 + \| | e^{u}| \|_{\infty,1}^2 \big)
            + \lambda^2 \big( \| u_{n,\beta} \|_{L^{\infty}(L^4)}^{4}
            + \| |e^{u}| \|_{\infty,1}^{4} \big) \Big]
            \sum_{n=1}^{M-1}\widehat{k}_n \|\nabla e^u_{n,\beta}\|^2 \notag \\
            &+ C(\theta)\big( \nu^2 \| u \|_{\infty,2}^2 + \lambda^2 \| u \|_{L^{\infty}(L^{4})}^{4} \big)
            k_{\rm{max}}^4 \|\nabla u_{tt}\|_{2,0}^2 \notag \\
            &+ C(\theta) k_{\rm{max}}^4 \big(  \|u_{ttt}\|_{2,0}^2 
            \!+\! \mu^2 \|\nabla u_{tt}\|_{2,0}^2 \!+\! \rho^2 \|u_{tt}\|_{2,0}^2 
            \!+\! \gamma^{2} \|\nabla w_{tt}\|_{2,0}^2 \!+\! \| p_{tt}\|_{2,0}^2 
            \!+\! \sigma^2 |\xi |^2 \|T_{tt}\|_{2,0}^2  \big). \notag 
        \end{align}
        By error estimates in \eqref{proof of max u-u^h T-T^h error}, \eqref{proof of max nabla u-u^h T-T^h error}, \eqref{lemma-p-conclusion},  triangle inequality, the approximation of 
        $\mathcal{S}_h p_{n,\beta}$ in \eqref{Stokes-type projection} 
        and Lemma \ref{n beta int tn-1 tn+1}, we achieve \eqref{nabla p-p^h error}.
    \end{proof}

    \section{Numerical results}  
    \label{sec:sec5}
    In this section, we present several numerical experiments to validate the theoretical analysis and assess the performance of \Cref{fully discrete formulations scheme} with the parameter $\theta = 0.3$.
    We utilize Taylor-Hood $\mathbb{P}2/\mathbb{P}1$ finite element space for spatial discretization among all the experiments. 
    We begin by conducting convergence tests to verify both the spatial and temporal accuracy 
    of \Cref{fully discrete formulations scheme} and then examine the self-organizing dynamics of active fluids in a two-dimensional domain, using random initial conditions \cite{MR4736040}. 
    This investigation aims to assess the robustness of \Cref{fully discrete formulations scheme} and the effectiveness of the time-adaptive strategy guided by the minimal dissipation criterion.

    \subsection{Convergence test}
    To validate the convergence rate of Scheme \ref{fully discrete formulations scheme} 
    in both space and time, we construct the test problem on the unit square domain \( D = [0,1] \times [0,1] \) with the following exact solution
    \begin{align*}
    &\begin{bmatrix}
    u_{1} \\ u_{2}
    \end{bmatrix} = 
    \begin{bmatrix}
    ( -\cos(2\pi x+\pi)-1)\sin(2 \pi y) \exp(2t) \\
    -\sin(2\pi x) \cos(2\pi y) \exp(2t)
    \end{bmatrix},  \\
    &\begin{bmatrix}
    w_{1} \\ w_{2}
    \end{bmatrix} = 
    \begin{bmatrix}
    (-3x^2 + 3y^2 + 8\pi^2 \sin(2\pi y) \cos(2\pi x) - 4\pi^2 \sin(2\pi y)) \exp(2t) \\
    (6xy - 8\pi^2\sin(2\pi x)\cos(2\pi y))\exp(2t)
    \end{bmatrix},  \\
    &\ \ \phi = (x^3-3xy^2)\exp(2t), \quad  p = \sin(3\pi^2x)\cos(3\pi^2 y)\exp(-t)
    \\
    &\ \ T = \big( ( -\cos(2\pi x+\pi)-1)\sin(2 \pi y)-\sin(2\pi x) \cos(2\pi y) \big)\exp(2t).
    \end{align*}
    We set the physical parameters to \( \mu = 1 \), \( \gamma = 1 \), \( \nu = 1 \), \( \rho = 1 \), \( \lambda = 1 \), \(\kappa=1\), \(\sigma=1\)  and $\xi = (0, 1)^T$and simulate the problem on the time interval $[0,1]$.
    The exact solution decides the source function and boundary conditions.

    We set the constant time step size \( \Delta t = \frac{1}{4}, \frac{1}{8}, \frac{1}{16}, \frac{1}{32} \) and fix 
    the uniform mesh diameter \( h = \frac{1}{128} \) to verify the convergence rate in time. 
    Meanwhile we adjust \( h = \frac{1}{8}, \frac{1}{16}, \frac{1}{32}, \frac{1}{64} \) and keep
    \( \Delta t = 1 \times 10^{-5} \) to confirm the convergence rate in space. 

    As shown in Table~\ref{tab:L2-errors and convergence rates in time}, the observed temporal convergence rates for the velocity \( u \), auxiliary variable \(w\), \(\phi\), pressure \( p \) and the temperature $T$ are consistent with the expected second-order accuracy. 
    Similarly, Table~\ref{tab:L2-errors and convergence rates in space} and %\ref{tab:H1-errors and convergence rates in space} 
    demonstrate that the spatial convergence rates for the velocity component \( u \), auxiliary variable \( w \) and temperature $T$ are both third-order in the \( L^2 \) and second-order in the \( H^1 \) norms, while the convergence rates for \(\phi\) and pressure \( p \) are both second-order in the \( L^2 \) norm and first-order in the \( H^1 \) norm.

    \begin{table}
        \centering
        \caption{$L^2$-errors and convergence rates in time}
        \begin{tabular}{lllllllllll}
            \hline
            $1 / \Delta t$ & $\|u - u^h\|$ & Rate & $\|w - w^h\|$ & Rate & $\|\phi- \phi^h\|$ & Rate & $\|p - p^h \|$ & Rate & $\|T- T^h\|$ & Rate \\
            \hline
            $4$   & 3.02E-01 & —         & 1.59E+01& —         & 5.16E-01& —         & 1.59E+02& —      & 2.12E+00& —   \\
            $8$   & 5.73E-02 &2.3970 &3.02E+00& 2.3970    & 9.79E-02& 2.3968 & 3.00E+01 &2.4071 & 4.58E-01&2.2113   \\
            $16$ &1.18E-02 & 2.2764  & 6.23E-01 &2.2764    & 2.03E-02 & 2.2728   &6.18E+00  &2.2798   & 1.02E-01& 2.1614 \\
            $32$ & 2.62E-03&2.1745 & 1.38E-01& 2.1744    & 4.62E-03& 2.1331  & 1.37E+00& 2.1769  & 2.40E-02& 2.0938  \\
            \hline
        \end{tabular}
        \label{tab:L2-errors and convergence rates in time}
    \end{table}

    \begin{table}
        \caption{$L^2$-errors and convergence rates in space}
        \label{tab:L2_errors_convergence_space}
        \centering
        \begin{tabular}{lllllllllll}
            \hline
            $1 / h$ & $\|u - u^h\|$ & Rate & $\|w - w^h\|$ & Rate & $\|\phi - \phi^h\|$ & Rate & $\|p - p^h\|$ & Rate & $\|T- T^h\|$ & Rate \\
            \hline
            $8$   & 6.41E-03 & —         & 4.54E-01 & —         & 2.63E-02 & —         & 1.84E+00& —      & 8.27E-03& —   \\
            $16$   & 8.12E-04 & 2.9807 & 6.07E-02& 2.9028   & 2.26E-03& 3.5440 & 2.79E-01 & 2.7226 & 1.16E-03& 2.8361   \\
            $32$ & 1.02E-04 & 2.9930  & 7.73E-03 & 2.9739   & 3.21E-04 & 2.8138  &4.67E-02  &2.5797   & 1.49E-04& 2.9543 \\
            $64$ & 1.28E-05&2.9981 & 9.71E-04 & 2.9926   & 8.27E-05& 1.9560   & 9.74E-03 & 2.2611  & 1.88E-05& 2.9902  \\
            \hline
        \end{tabular}
        \label{tab:L2-errors and convergence rates in space}
    \end{table}

    \begin{table}
        \centering
        \caption{$H^1$-errors and convergence rates in space}
        \label{tab:H1-errors-space}
        \begin{tabular}{lllllllllll}
            \hline
            $1 / h$ & $\|u - u^h\|_{1}$ & Rate & $\|w - w^h\|_{1}$ & Rate 
            & $\|\phi - \phi^h\|_{1}$ & Rate & $\|p - p^h\|_{1}$ & Rate & $\|T - T^h\|_{1}$ & Rate \\
            \hline
            $8$   & 3.80E-01 & —         & 2.92E+01 & —         & 5.30E-01& —         & 3.78E+01& —      & 5.66E-01& —   \\
            $16$   & 9.78E-02& 1.9566 & 7.53E+00& 1.9558    & 1.64E-01& 1.6955 & 2.00E+01 & 0.9181  &1.43E-01&1.9832    \\
            $32$ & 2.47E-02 &1.9881   & 1.90E+00 & 1.9883    & 7.71E-02&1.0863   &1.06E+01  &0.9215    & 3.60E-02& 1.9901 \\
            $64$ &6.18E-03&1.9969  & 4.75E-01& 1.9970   & 3.83E-02& 1.0093   & 4.96E+00 & 1.0887   &9.02E-03& 1.9971   \\
            \hline
        \end{tabular}
        \label{tab:H1-errors and convergence rates in space}
    \end{table}

    \subsection{Thermal-driven active fluid system simulation}
    \label{subsec:self-organization}
    To investigate the long-term stability of \Cref{fully discrete formulations scheme}, we perform a numerical experiment simulating the spatial thermal-driven self-organization of bacterial active fluid on a unit square domain \cite{LSMW21_Nature}.
    The setup of the experiment is illustrated in \Cref{Initial-boundary values for thermal driven cavity flow}. 
    The left and right boundaries of the domain \( D = [0, 1]\times [0, 1]\) are maintained at fixed temperatures, \(T_1 = 0.5\) and \(T_0 = -0.5\), respectively, while the top and bottom boundaries are adiabatic, enforcing the condition \(\partial_n T = 0\).
    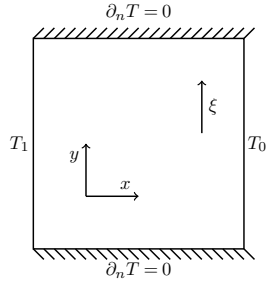
\begin{figure}
        \centering
        \scalebox{0.7}{ % 将整个图片缩放至70%
            \begin{tikzpicture}	
            
            % 绘制矩形边界
            \draw[thick] (0, 0) rectangle (4, 4);
            
            % 绘制上边界的阴影线
            \foreach \i in {0, 0.2, ..., 4} {
                \draw[thick] (\i, 4) -- ++(0.2, 0.2);
            }
            
            % 绘制下边界的阴影线
            \foreach \i in {0, 0.2, ..., 4} {
                \draw[thick] (\i, 0) -- ++(0.2, -0.2);
            }
            
            % 添加箭头
            \draw[->, thick] (3.2, 2.2) -- (3.2, 3.2) node[midway, right] {$\xi$};
            \draw[->, thick] (1, 1) -- (2, 1) node[near end, above] {$x$};
            \draw[->, thick] (1, 1) -- (1, 2) node[near end, left] {$y$};
            
            % 添加文本
            \node at (4.2, 2) {$\ T_0$};
            \node at (-0.2, 2) {$T_1 \ $};
            \node at (2, 4.5) {$\partial_n T = 0$};
            \node at (2, -0.4) {$\partial_n T = 0$};
            
            \end{tikzpicture}
        }
        \caption{Initial-boundary values for thermal driven cavity flow}
        \label{Initial-boundary values for thermal driven cavity flow}
    \end{figure}
    The velocity boundary condition is no-slip, i.e., \( u|_{\partial D} = w|_{\partial D} = 0 \), at all four edges of the domain. 
    The initial condition for the velocity field is given by:
    \[
    u_0(x,y) = \left( \text{rand}(x,y), \text{rand}(x,y) \right), \qquad (x,y) \in D,
    \]
    where 'rand' denotes a uniform random number generator over the interval \([-1, 1]\).
    The physical parameters used in the simulation are: \(\mu = 0.045\), \(\nu = 0.003\), \(\beta = 0.5\), \(\alpha = -0.81\), \(\gamma = \mu^3\), \(\xi = (0, 1)^\mathrm{T}\), \(\sigma = 1\) and \(\kappa = 1\). 

    The simulation is carried out over the time interval \([0, 1]\) with a uniform time step size of \( \Delta t = 1/100 \). After an initial period of self-adjustment, the phase space trajectory of the bacterial active fluid is expected to rotate, either clockwise or counterclockwise, at a constant angle, forming a unidirectional vortex flow that eventually reaches a steady state.
    Figure~\ref{fig:velocity_glyph_temperature_comparison} shows the evolution of the velocity vector field and the velocity field over time at selected instances. 
    The velocity field appears disordered at \(t = 0\), but by \(t = 0.15\), small vortices begin to form. 
    As time progresses to \( t = 0.30 \), the vortices become more pronounced and organized, displaying a clear flow pattern. 
    Eventually, at \( t = 1.00 \), the system stabilizes into a circular, coherent structure, indicating that \Cref{fully discrete formulations scheme} achieves long-time stability. 
    As time goes on, the temperature field becomes more structured and develops clear gradients.
    \begin{figure}[htbp] 
        \centering
        \begin{minipage}[t]{0.24\linewidth}
            \centering
            \includegraphics[width=3.2cm]{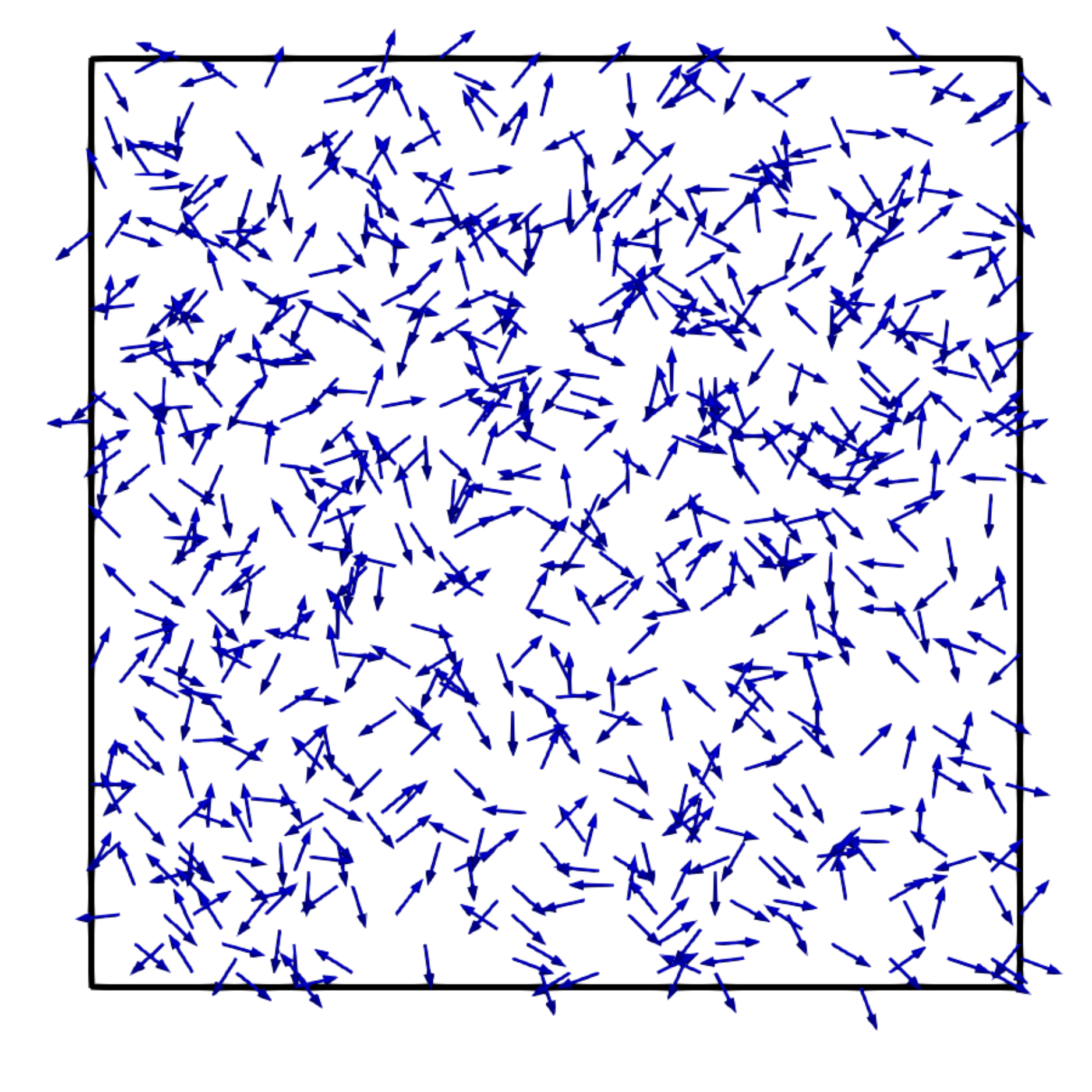} 
            % \\
            % \includegraphics[width=4cm]{Figure/velocity_t_0.eps}
            \caption*{$t=0$}
        \end{minipage}
        \begin{minipage}[t]{0.24\linewidth}
            \centering
            \includegraphics[width=3.2cm]{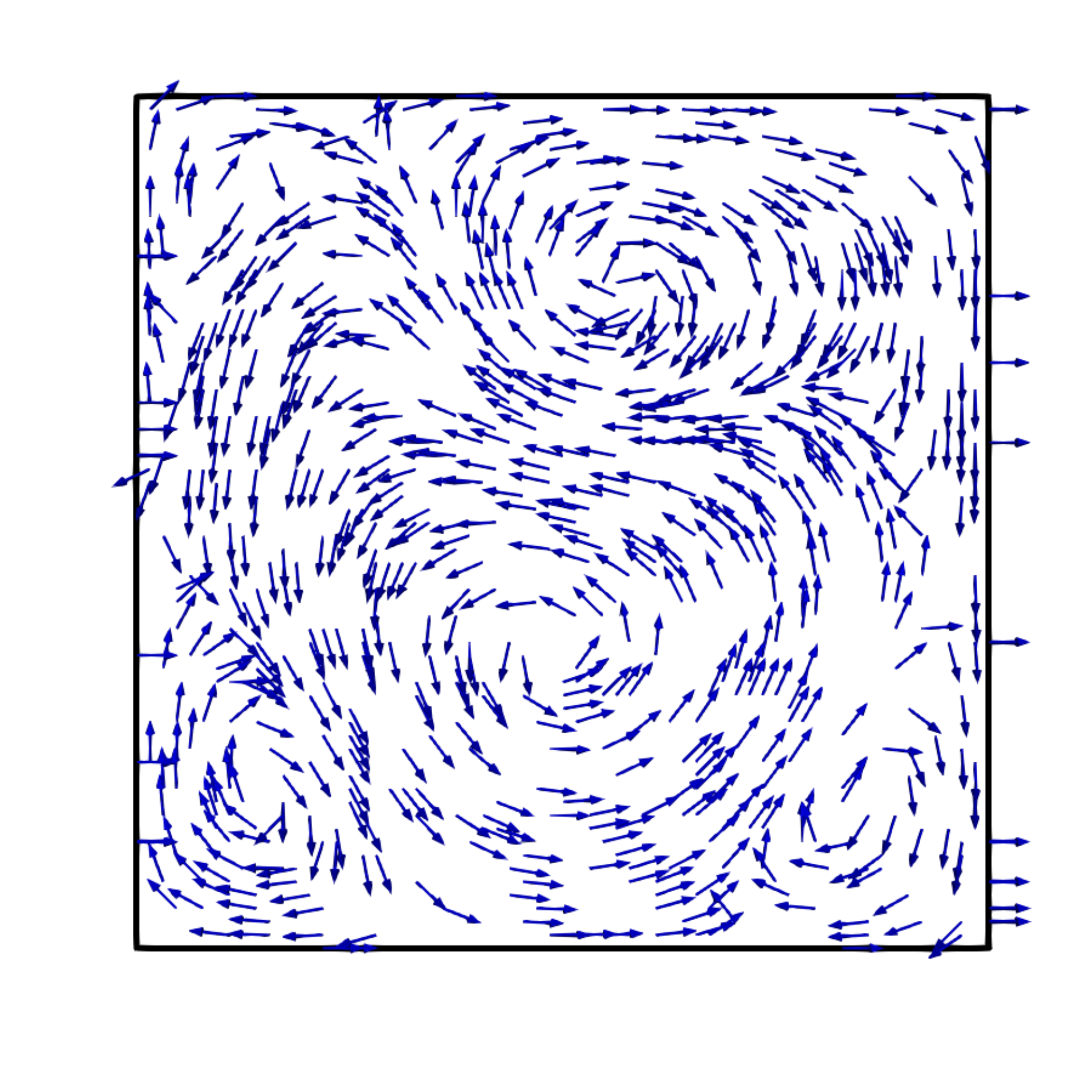} 
            % \\
            % \includegraphics[width=4cm]{Figure/velocity_t_0.15.eps}
            \caption*{$t=0.15$}
        \end{minipage}
        \begin{minipage}[t]{0.24\linewidth}
            \centering
            \includegraphics[width=3.2cm]{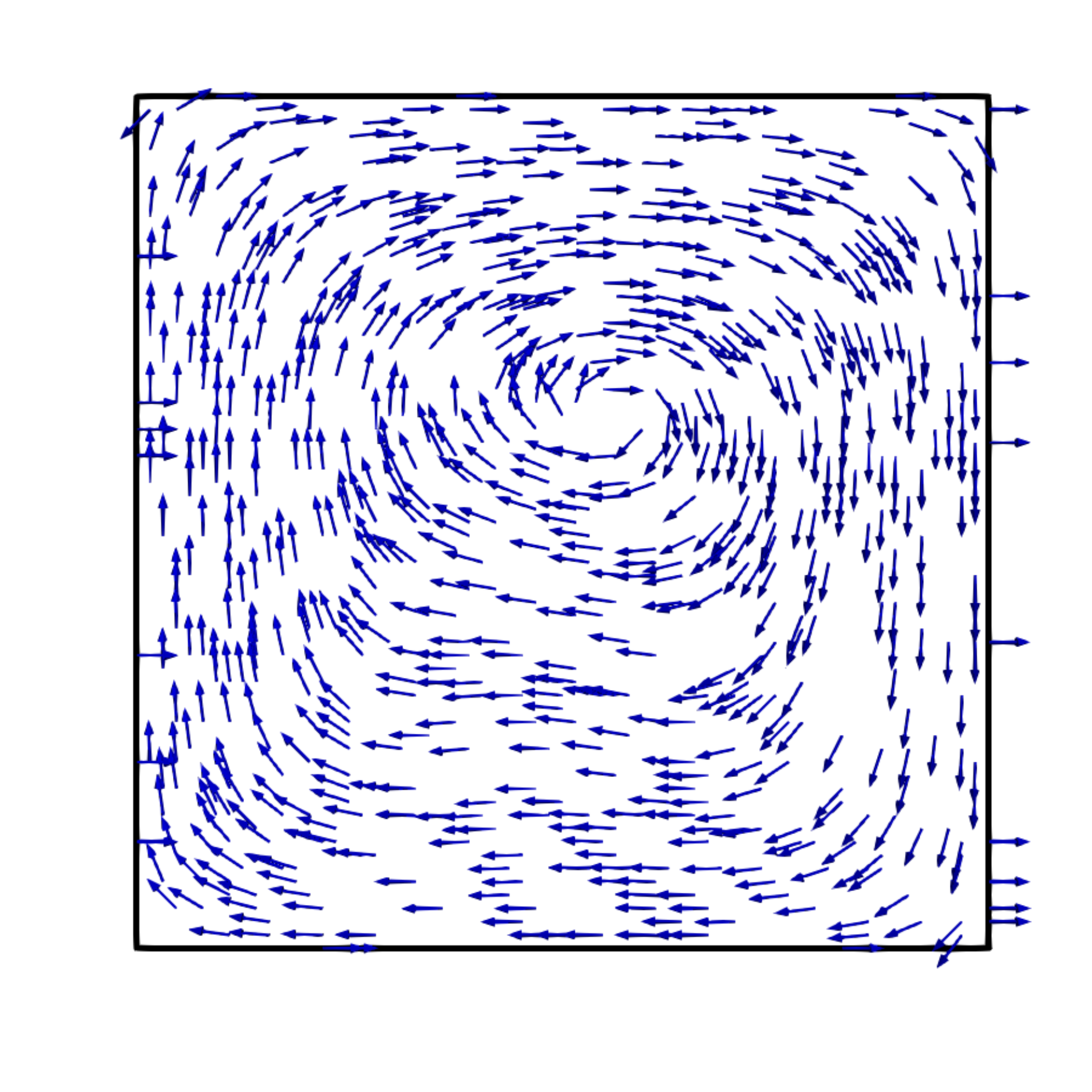} 
            % \\
            % \includegraphics[width=4cm]{Figure/velocity_t_0.3.eps}
            \caption*{$t=0.30$}
        \end{minipage} 
        \begin{minipage}[t]{0.24\linewidth}
            \centering
            \includegraphics[width=3.2cm]{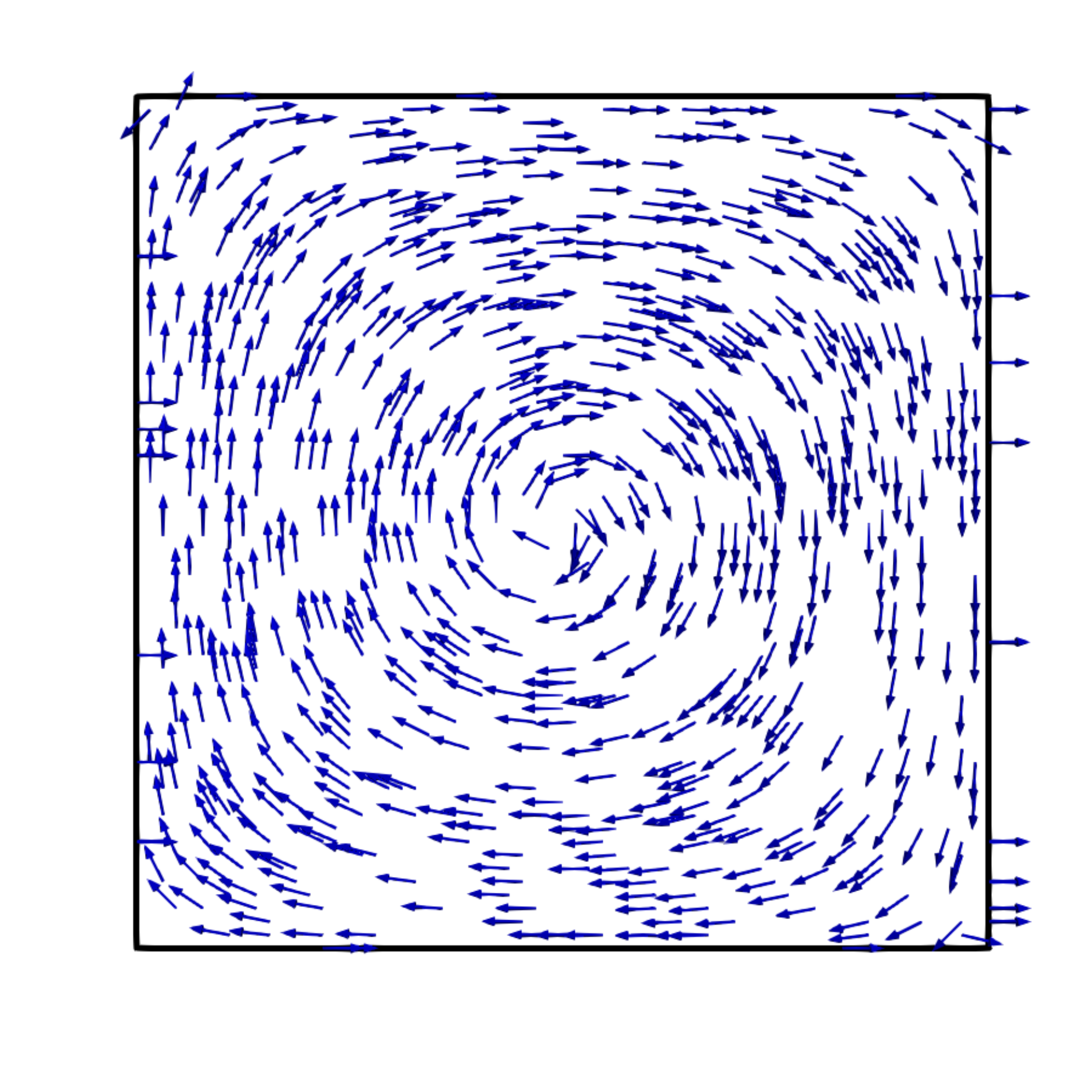} 
            % \\
            % \includegraphics[width=4cm]{Figure/velocity_t_0.3.eps}
            \caption*{$t=1.00$}
        \end{minipage} \\
        \vspace{-0.1cm}
        \begin{minipage}[t]{0.24\linewidth}
            \centering
            \includegraphics[width=3.1cm]{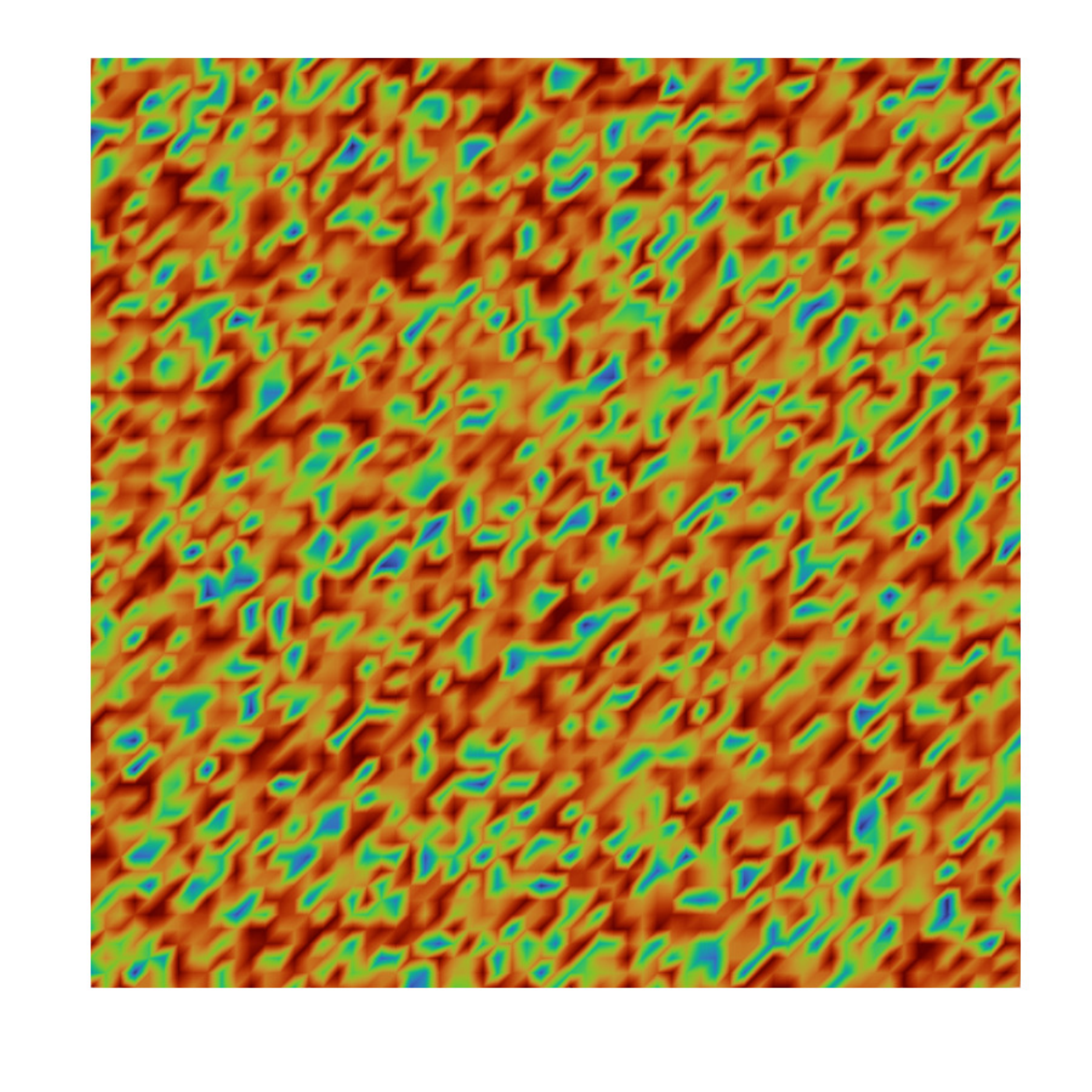} 
            % \\
            % \includegraphics[width=4cm]{Figure/velocity_t_0.eps}
            \caption*{$t=0$}
        \end{minipage}
        \begin{minipage}[t]{0.24\linewidth}
            \centering
            \includegraphics[width=3.1cm]{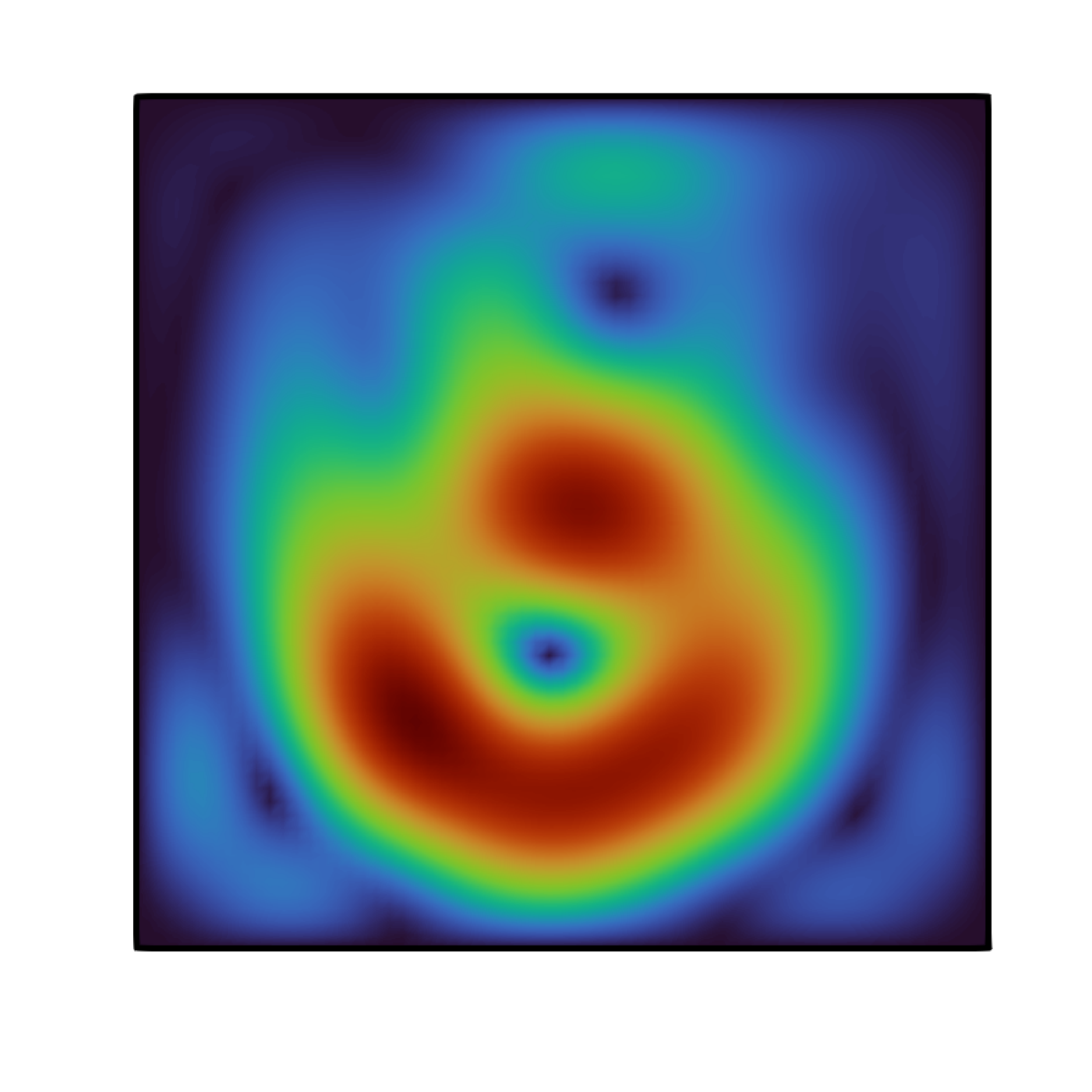} 
            % \\
            % \includegraphics[width=4cm]{Figure/velocity_t_0.15.eps}
            \caption*{$t=0.15$}
        \end{minipage}
        \begin{minipage}[t]{0.24\linewidth}
            \centering
            \includegraphics[width=3.1cm]{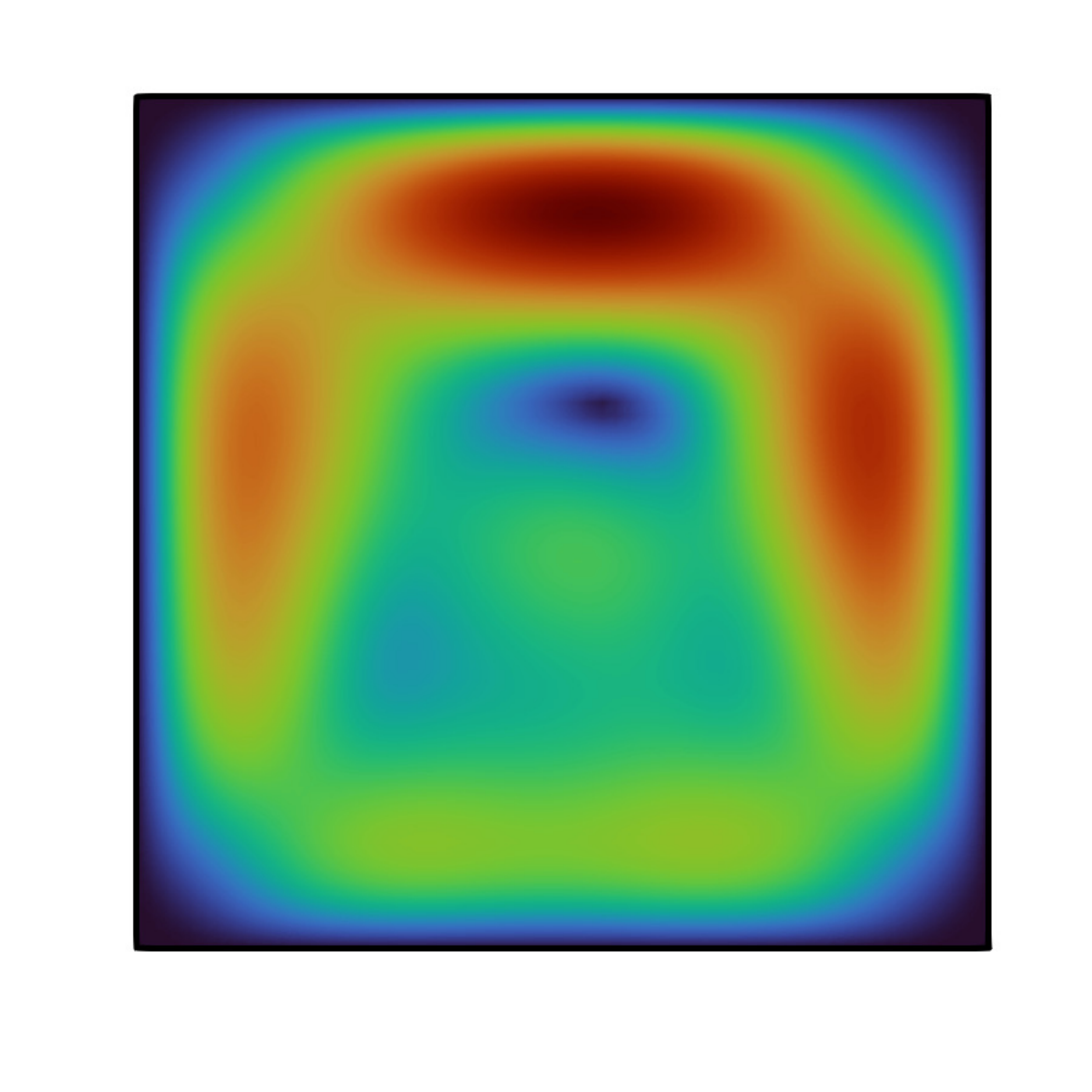} 
            % \\
            % \includegraphics[width=4cm]{Figure/velocity_t_0.3.eps}
            \caption*{$t=0.30$}
        \end{minipage} 
        \begin{minipage}[t]{0.24\linewidth}
            \centering
            \includegraphics[width=3.1cm]{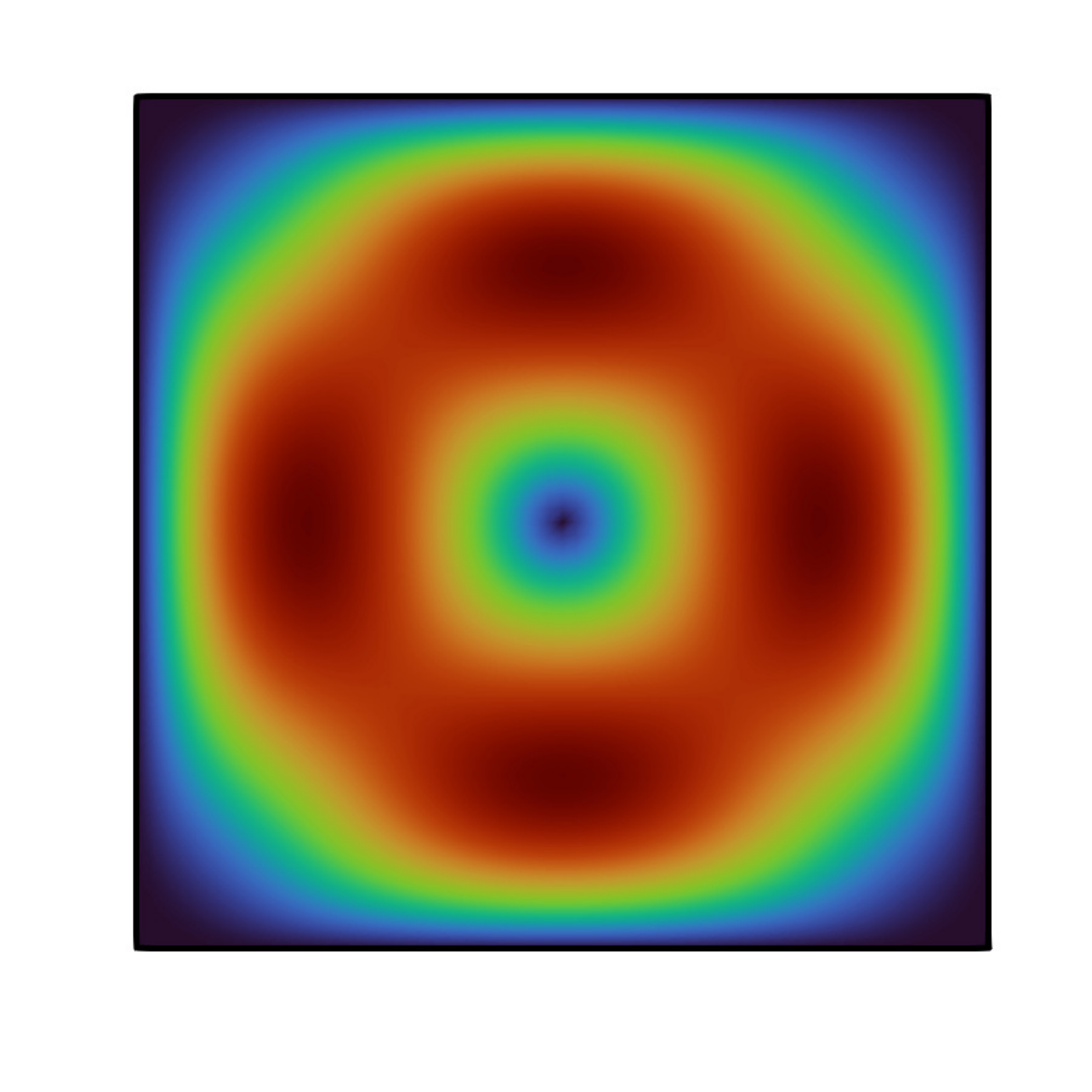} 
            % \\
            % \includegraphics[width=4cm]{Figure/velocity_t_0.1.eps}
            \caption*{$t=1.00$}
        \end{minipage}\\
        \vspace{-0.1cm}
        \begin{minipage}[t]{0.24\linewidth}
            \centering
            \includegraphics[width=3.1cm]{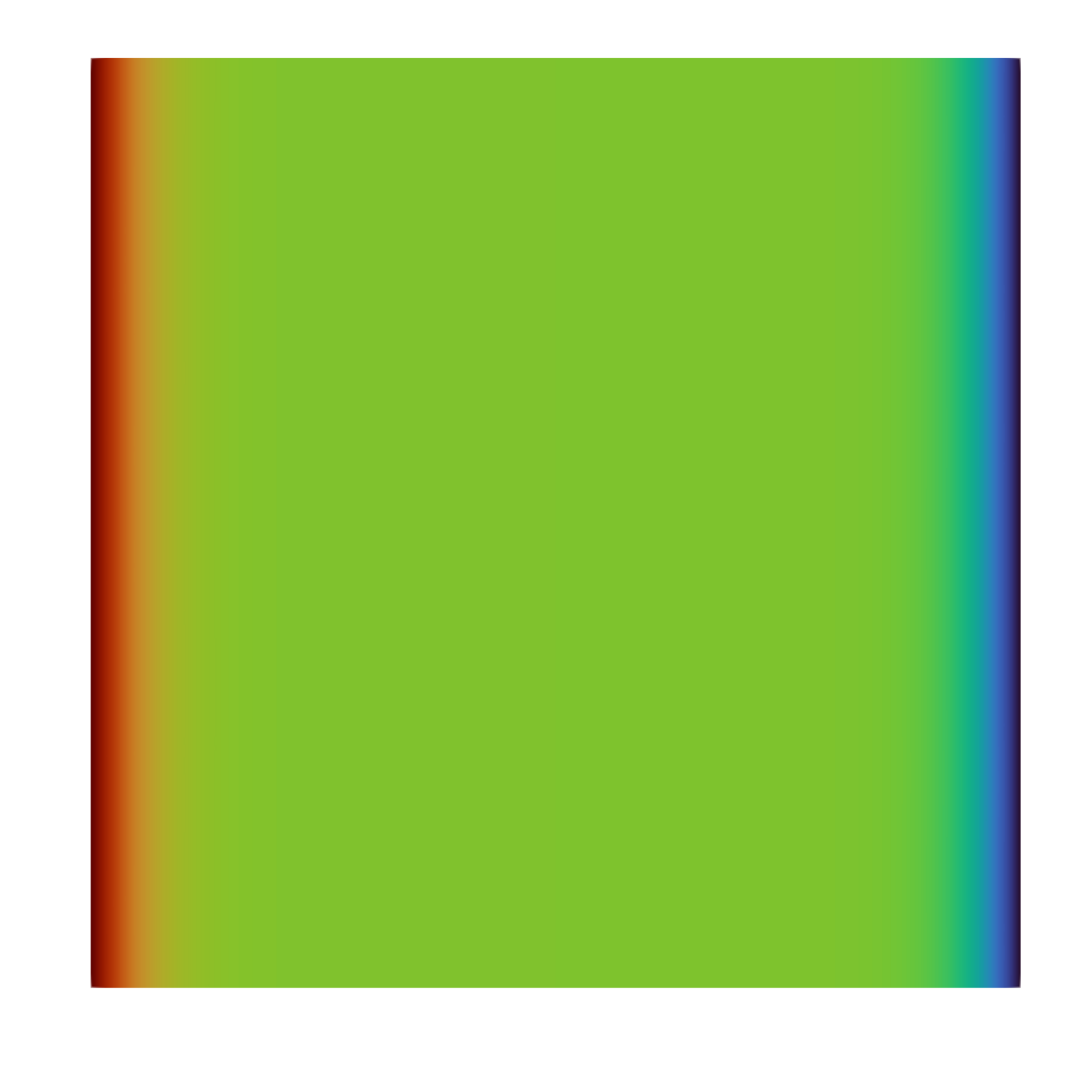} 
            % \\
            % \includegraphics[width=4cm]{Figure/velocity_t_0.eps}
            \caption*{$t=0$}
        \end{minipage}
        \begin{minipage}[t]{0.24\linewidth}
            \centering
            \includegraphics[width=3.1cm]{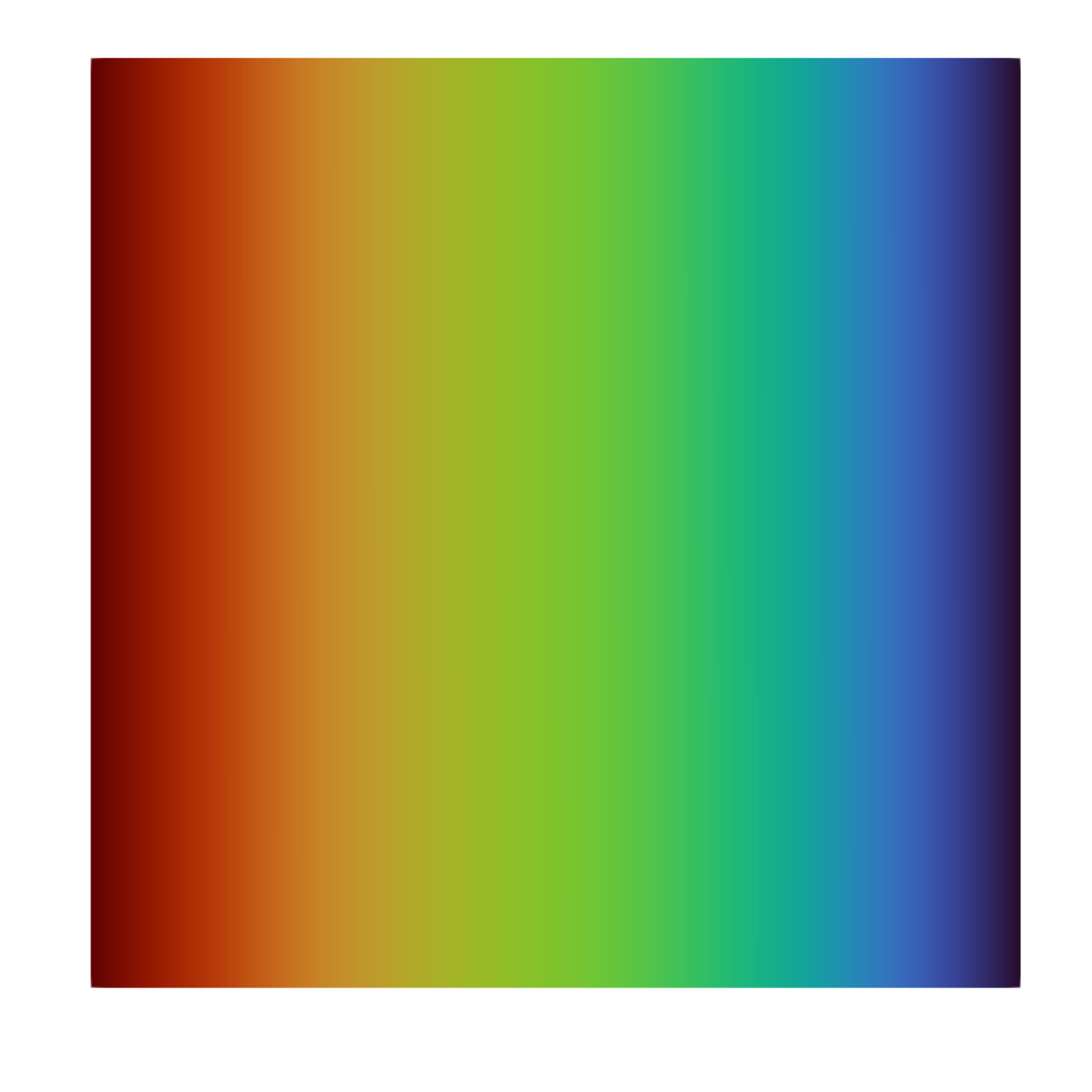} 
            % \\
            % \includegraphics[width=4cm]{Figure/velocity_t_0.15.eps}
            \caption*{$t=0.15$}
        \end{minipage}
        \begin{minipage}[t]{0.24\linewidth}
            \centering
            \includegraphics[width=3.1cm]{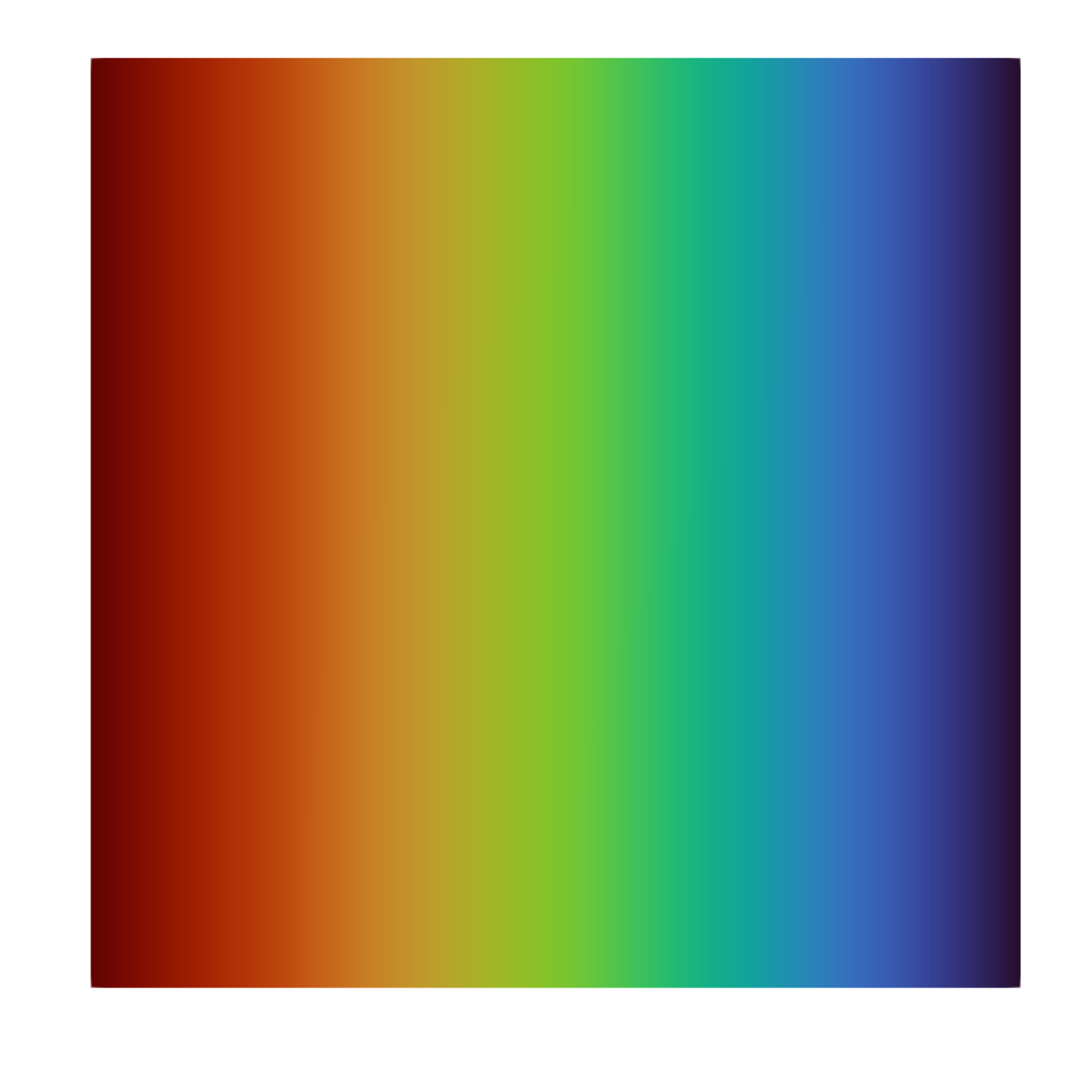} 
            % \\
            % \includegraphics[width=4cm]{Figure/velocity_t_0.3.eps}
            \caption*{$t=0.30$}
        \end{minipage} 
        \begin{minipage}[t]{0.24\linewidth}
            \centering
            \includegraphics[width=3.1cm]{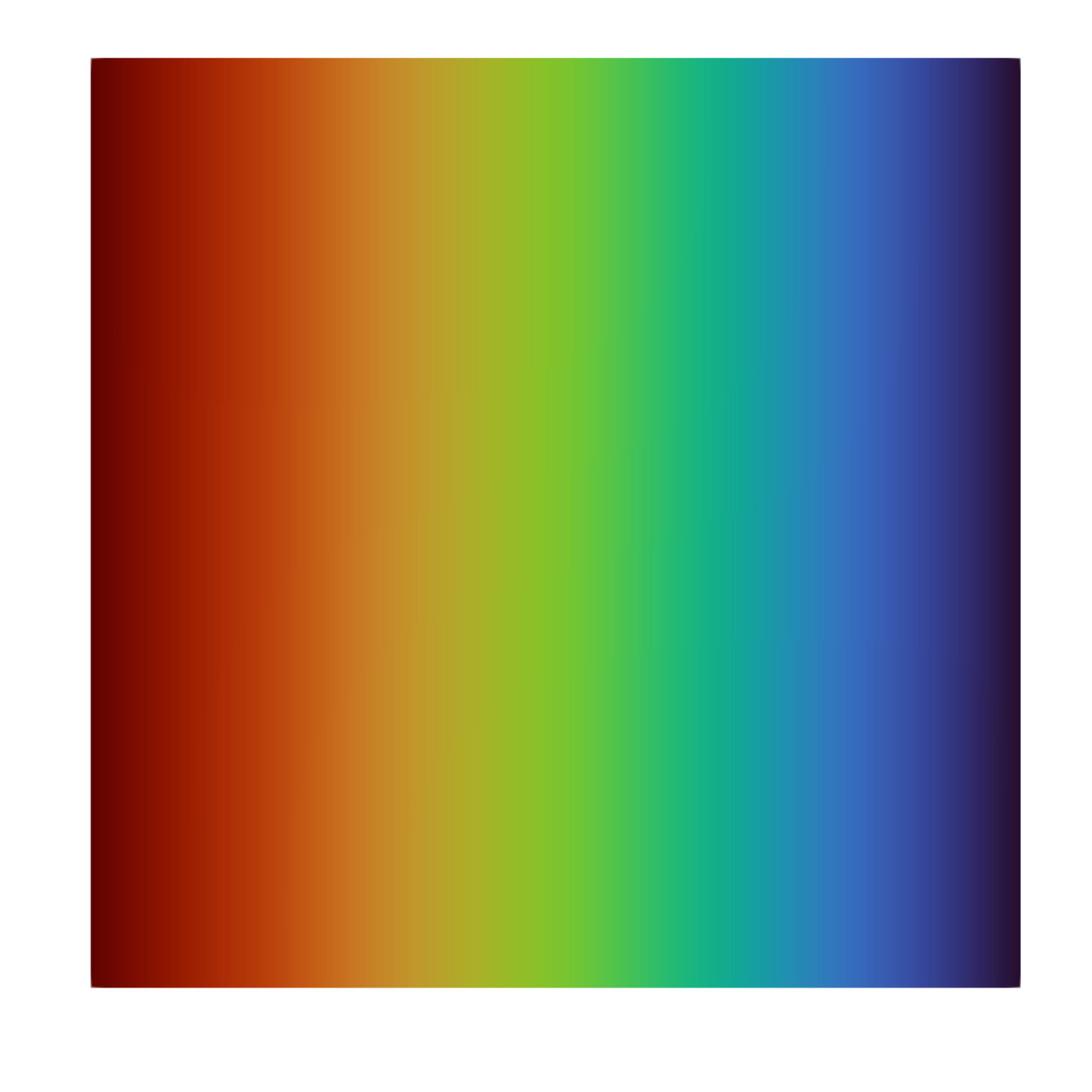} 
            % \\
            % \includegraphics[width=4cm]{Figure/velocity_t_0.1.eps}
            \caption*{$t=1.00$}
        \end{minipage}
        \vspace{-0.5cm}
        \caption{Time evolution of the vector field of velocity (top row), the velocity field (middle row), and the temperature field (bottom row) at selected time instances.}
        \label{fig:velocity_glyph_temperature_comparison}
    \end{figure}

    \subsection{Time adaptive test}
    To relieve the conflict between accuracy and computational cost as well as utilize fine properties of the DLN method under non-uniform time grids, we design a time-adaptive approach for \Cref{fully discrete formulations scheme} based on the minimum dissipation criterion proposed by Capuano, Sanderse, De Angelis, and Coppola \cite{capuano2017minimum}. 
    At each time step, we compute the numerical dissipation (ND) rate for both velocity $u$ and temperature variable $T$, the viscosity-induced dissipation (VD) for $u$, the temperature-induced dissipation (TD) for $w$
    \begin{align*}
    &\text{ND: }\ \ 
    \epsilon_{n}^{u, ND} = \frac{1}{\widehat{k}_n} \| u_{n,\alpha}^{h} \|^2, 
    \quad
    \epsilon_{n}^{T, ND} = \frac{1}{\widehat{k}_n} \| T_{n,\alpha}^{h} \|^2, 
    \\
    &\text{VD: }\ \ 
    \epsilon_{n}^{u, VD} = \mu \| \nabla u_{n,\beta}^h \|^2,
    \quad
    \text{TD: }\ \ 
    \epsilon_{n}^{T, VD} = \kappa \| \nabla T_{n,\beta}^h \|^2,
    \end{align*}
    and the ratios of ND over VD and TD: $\chi_u = \epsilon_{ND}^{u} / \epsilon_{VD}^{u}$, $\chi_w = \epsilon_{ND}^{T} / \epsilon_{VD}^{T}$.
    We adjust the next time step $k_{n+1}$ by 
    \begin{equation}
    \label{time-controller}
    \begin{split}
    &k_{n+1} = 
    \begin{cases} 
    \min \{2 k_n, k_{\max} \}, & \text{if}~  \max \{|\chi_u|,|\chi_w|\} \leq \delta, \\
    \max \{\frac{1}{2}k_n, k_{\min} \}, & \text{if}~ \max \{|\chi_u|,|\chi_w|\} > \delta,
    \end{cases}
    \end{split}
    \end{equation}
    for the required tolerance $\delta >0$.
    We observe from \eqref{time-controller} that the strategy allows a larger time step ($k_{n+1} = 2k_{n}$) for the next operation if the ratios are below the required tolerance; otherwise, decrease the next time step by half. 
    Meanwhile, we require the time step between $k_{\min}$ and $k_{\max}$ for accuracy and efficiency.
    The complete adaptive procedure incorporating the above time step selection strategy is summarized in \Cref{Adaptive_procedure}.
    \begin{algorithm}
        \caption{Adaptive procedure for \Cref{fully discrete formulations scheme}.}
        \label{Adaptive_procedure}
        \begin{algorithmic}
            \State Set $u_0^h = \mathcal{S}_h u_0, T_0^h = \mathcal{R}_h T_0$ and compute $u_1^h$, $w_1^h$, $\phi_1^h$, $p_1^h$, $T_1^h$ using a Crank--Nicolson scheme;
            \For{$n = 1,2, \dots, M-1$}
            \State Solve active fluid equations \eqref{fully discrete formulations 1}–\eqref{fully discrete formulations 5} by  \Cref{fully discrete formulations scheme}.;
            \State Compute numerical and viscous indicators $\epsilon^{u, N\!D}_{n+1}$, $\epsilon^{u, V\!D}_{n+1}$,  $\epsilon^{T, N\!D}_{n+1}$, $\epsilon^{T, V\!D}_{n+1}$;
            \State Calculate $\chi_u = \epsilon^{u, N\!D}_{n+1}/\epsilon^{u, V\!D}_{n+1}$, $\chi_T = \epsilon^{T, N\!D}_{n+1}/\epsilon^{T, V\!D}_{n+1}$;
            \If{$ \max \{|\chi_u|,|\chi_T|\} \leq \delta$}
            \State Set $k_{n+1} = \min \{2 k_n, k_{\max}\}$;
            \Else
            \State Set $k_{n+1} = \max\{0.5 k_n, k_{\min}\}$;
            \EndIf
            
            \State Set $(u_{n}^h, w_{n}^h, \phi_{n}^h, p_{n}^h, T_{n}^h) \Leftarrow (u_{n+1}^h, w_{n+1}^h, \phi_{n+1}^h, p_{n+1}^h, T_{n+1}^h)$ \\
            \quad \quad and $(u_{n-1}^h, w_{n-1}^h, \phi_{n-1}^h, p_{n-1}^h, T_{n-1}^h) \Leftarrow (u_{n}^h, w_{n}^h, \phi_{n}^h, p_{n}^h, T_{n}^h)$ ;
            \State Go to the next step;
            \EndFor
        \end{algorithmic}
    \end{algorithm}

    We evaluate the performance of the proposed time-adaptive strategy via the former experiment in Subsection \ref{subsec:self-organization}. 
    We set $k_{\max} = 0.01$, $k_{\min} = 1.\rm{e}-4$, $\delta = 0.1$, $k_{0} = k_{\min}$, carry out the experiment with different level of Reynolds number $\rm{Re} = 1/\mu$: $5.\rm{e}+2, 5.\rm{e}+3, 5.\rm{e}+4,  5.\rm{e}+5,  5.\rm{e}+6, 5.\rm{e}+7$, and make other parameters and conditions unchanged. 
    We also compare this approach against the corresponding constant time-stepping scheme with $20000$ time steps for the effectiveness of time adaptivity. 
    Both approaches achieve very similar results of the evolution of the vector field of velocity and the velocity field over the time interval $[0,2]$.
    However, \eqref{Comparison of Different Reynolds Numbers with Adaptive Time Steps and Constant Time Steps} shows that the constant time-stepping scheme costs many more time steps under all levels of Reynolds number selected, which emphasizes the superiority of the time-adaptive approach.
    \begin{table}
        \centering
        \caption{The constant time-stepping scheme costs many more time steps under all levels of Reynolds number selected, which emphasizes the superiority of the time-adaptive approach.}
        \begin{tabular}{lrrrrrr}
            \hline
            Re$^a$  & 5.\rm{e}+2 &  5.\rm{e}+3 & 5.\rm{e}+4&  5.\rm{e}+5 &  5.\rm{e}+6&  5.\rm{e}+7 \\
            \hline
            Adaptive$^b$   &  499 &  491& 491& 491  & 504&500\\
            Constant$^c$  &20000& 20000 & 20000 & 20000  & 20000& 20000  \\
            \hline
        \end{tabular}
        \label{Comparison of Different Reynolds Numbers with Adaptive Time Steps and Constant Time Steps}
        
        \vspace{1ex}
        {\footnotesize
            $^a$ Reynolds number; \\
            $^b$ Number of computational steps with adaptive time stepping; \\
            $^c$ Number of computational steps with fixed time steps.
        }
    \end{table}

    \section{Conclusion}
    \label{sec:sec6}
    In this paper, we have developed an efficient spatial-temporal discretization scheme for an equivalent second-order reformulation of the thermally driven active fluid system.
    The variable time-stepping DLN method, which is second-order accurate and nonlinearly stable, is employed as the time integrator.
    For spatial discretization, we introduce two additional auxiliary variables, and thus construct a divergence-free preserving and easily-implemented mixed finite element method.
    With the help of appropriate regularity assumptions and mild time-diameter restrictions, we have rigorously proved that the fully discrete DLN scheme is long-time stable in the model energy and established error estimates for velocity and temperature in $L^2$ and $H^1$-norm and pressure in $L^2$-norm.
    Furthermore, a time-adaptive strategy is designed to maintain robustness of \Cref{fully discrete formulations scheme} and improve computational efficiency.
    Several numerical experiments validate our theoretical findings, demonstrating that \Cref{fully discrete formulations scheme}, along with the time-adaptive approach in \Cref{Adaptive_procedure}, provides an efficient framework for solving thermally driven active fluid dynamics and other more complex systems.

    \section*{Declarations}
    \vskip 0.5cm
    \ \\
    \noindent\textbf{Funding.}
    Wenju Zhao was partially supported by National Key R\&D Program of China (No. 2023YFA1008903), Natural Science Foundation of Shandong Province (No. ZR2023ZD38), National Natural Science Foundation of China (No. 12131014).

    \noindent\textbf{Conflict of interest.}
    The authors declare no potential conflict of interests.

    \noindent\textbf{Code available.}
    Upon request.

    \noindent\textbf{Ethics and consent to participate.} Not applicable

    \bibliographystyle{abbrv}

    \bibliography{references}

\end{document}